\documentclass[a4paper,10pt]{article}

\usepackage[pdftex]{graphicx} 
\usepackage{amsfonts,latexsym,amssymb, mathrsfs,amsmath, mathtools, longtable} 
\usepackage{subfigure, hyperref} 
\usepackage{color} 
\usepackage{caption}

\graphicspath{{figures/}}

\newenvironment{modification}{}{}

\newenvironment{proof}{{\par\medskip 
\noindent\bf Proof. }}{ $\blacksquare$\\
}

\newcommand{\vecxi}{\boldsymbol\xi}

\newtheorem{theorem}{Theorem}[section] 
\newtheorem{lemma}[theorem]{Lemma} 
 
\newtheorem{problem}{Problem} 
\newtheorem{proposition}[theorem]{Proposition} 
\newtheorem{hypo}{Hypothesis}

\newtheorem{definition}{Definition}[section]

\newtheorem{remark}{Remark}[section]

\title{Optimal Strokes for Driftless Swimmers: A General Geometric Approach}

\author{Thomas Chambrion \footnote{Universit\'e de Lorraine, IECL, B.P. 239, F-54506 Vand{\oe}uvre-l\`es-Nancy Cedex, CNRS, UMR7502, B.P. 239, F-54506 Vand{\oe}uvre-l\`es-Nancy Cedex, Inria, F-54600 Villers, \textbf{ \texttt{thomas.chambrion@univ-lorraine.fr}}} \and Laetitia Giraldi \footnote{Unit\'e de Math\'ematiques Pures et Appliqu\'ees, UMPA, ENS de Lyon, 46 all\'ee d'Italie, 69364 LYON, France, \textbf{\texttt{laetitia.giraldi@ens-lyon.fr}}, supported by Direction G\'en\'erale de l'Armement} \and Alexandre Munnier \footnote{ Universit\'e de Lorraine, IECL, B.P. 239, F-54506 Vand{\oe}uvre-l\`es-Nancy Cedex, CNRS, UMR7502, B.P. 239, F-54506 Vand{\oe}uvre-l\`es-Nancy Cedex, Inria, F-54600 Villers, \textbf{ \texttt{alexandre.munnier@inria.fr}}}}

\begin{document} 
\maketitle

\begin{abstract}
	Swimming consists by definition in propelling through a fluid by means of bodily movements. Thus, from a mathematical point of view, swimming turns into a control problem for which the controls are the deformations of the swimmer. The aim of this paper is to present a unified geometric approach for the optimization of the body deformations of so-called driftless swimmers. The class of driftless swimmers includes, among other, swimmers in a 3D Stokes flow (case of micro-swimmers in viscous fluids) or swimmers in a 2D or 3D potential flow. A general framework is introduced, allowing the complete analysis of five usual nonlinear optimization problems to be carried out. The results are illustrated with examples coming from the literature and with an in-depth study of a swimmer in a 2D potential flow. Numerical tests are also provided. 
\end{abstract}

\newpage 

\setcounter{tocdepth}{2} 
\tableofcontents

\newpage

\section{Introduction} 
\subsection{High and low Reynolds swimmers} Understanding the mechanics of swimming has been an issue in Mathematical Physics for a long time. Aside from improving the academic understanding of locomotion in fluid, this interest growth from the observation that fish and aquatic mammals evolved swimming capabilities far superior to those achieved by human technology and consequently provide an attractive model for the design of biomimetic robots. Significant contributions to this matter are due to Taylor \cite{Taylor51}, J. Lighthill \cite{Lighthill75}, E.M. Purcell \cite{Purcell77} and T. Y. Wu \cite{Wu01}.

Among the many models available in the literature, let us focus on those for which the Reynolds number of the fluid is either very low or very high. The main interest of these cases lies in the fact that the dynamics governing the fluid-swimmer system are simple enough to allow theoretical results to be proved. Theses two cases can be referred to as ``driftless models''. The first one for which the fluid is assumed to be very viscous is called ``resistive model''. It is relevant for microswimmers (like microorganisms) and consists in neglecting the inertial effects in the modeling. The second one called ``reactive model'', is obtained by neglecting rather the viscous forces. Although mostly academic (because it is vortex free), this model could have some relevance for swimmers with elongated bodies (like eels). Surprisingly, the dynamics are very close for both cases and their study fall under the same general abstract framework. 

The well-posedness of the system of equations for these models was established for instance in \cite{DalMasoDesimone10} and \cite{LoheacMunnier12} in a Stokesian flow and in \cite{ChambrionMunnier11a} in a perfect fluid. The controllability is addressed in \cite{ChambrionMunnier11a} and \cite{LoheacMunnier12} where the authors prove the generic controllability of 3D driftless swimmers in a perfect and Stokesian flow respectively. Earlier results for particular cases were established in \cite{AlougesDeSimone08} and next generalized in \cite{AlougesGiraldi12}, in which 3D three or four-spheres mechanisms are shown to be controllable. 

To our knowledge, still few theoretical studies have been conducted about optimal swimming (although numerical approach of the problem are many). In \cite{LoheacScheid11}, J. Loheac et al. are interested in optimizing the swimming of a 3D slightly deformable sphere in order to minimize its displacement time. In \cite{Alouges10}, the authors describe an algorithm allowing optimizing the strokes for a three-sphere swimmer, based on the theory of calculus of variations. 

\subsection{Contribution}

The aim of this work is to provide a general framework to study the optimal controllability of driftless swimmers. In particular, every aforementioned paper falls within this framework. After recalling minimal hypotheses ensuring the controllability of the system under consideration, we shall focus on the study of optimal strokes i.e. periodic shape changes. More precisely, we will be interested in the following points: 
\begin{itemize}
	\item Existence of optimal strokes, minimizing or maximizing various cost functionals (related to the energy of the system, the efficiency, the time). As in \cite{LoheacScheid11}, contraints on the state of the system are taken into account (for instance, the deformations can be required to be {\it not too large}). 
	\item Qualitative properties of the optimal strokes (or, differently stated, of the corresponding optimal controls). In particular, we will show how the optimal controls corresponding to different cost functionals can actually be deduced one from the others. 
	\item Regularity and monotonicity of the value functions (does the cost increases along with the covered distance?). 
\end{itemize}
Most of the proofs rely on the following arguments: 
\begin{itemize}
	\item The analyticity of the system; 
	\item The Riemannian and sub-Riemannian underlying structures. 
\end{itemize}
For pedagogical purposes, the results will be applied to examples from the literature and to the toy model introduced in \cite{ChambrionMunnier11a} and \cite{Munnier:2011aa}. The main interest of this model, dealing with a swimmer in a 2D potential flow, is that the governing equations, although not trivial, can be made fully explicit by means of complex calculus. 

\subsection{Outline and Main Achievements}

Section \ref{SEC_analysis} is devoted to the definition and the abstract analysis of the optimization problems. An abstract framework is introduced (Subsection~\ref{SEC_framework}), and classical sufficient conditions for controllability are recalled (Subsection \ref{SEC:control}). After stating five classical optimization problems (Subsection \ref{SEC_optimal_problems}), we show that every one admits minimizers or maximizers (Subsection \ref{SEC_existence_minimizers}). We prove that most of these problems are actually equivalent (for instance, it is completely equivalent to minimize the time, as in \cite{LoheacScheid11} and to minimize the efficiency as in \cite{AlougesDeSimone08}). Then we focus on the value function that associates with every distance the minimal cost of the corresponding stroke. It turns out that this function is increasing when one considers only small strokes but it is not necessarily monotonic in general (Subsection \ref{SEC_continuity_value_function}). In other words, it may sometimes be cheaper to reach a further point (this surprising result is numerically illustrated in Subsection~\ref{swim:num}). In Subsection~\ref{caseN=2}, we restrict our study to driftless swimmers with two degrees of freedom, in which case the stroke optimization problem turns into an isoperimetric problem on the shape manifold. 

In Section \ref{MathematicalSetting}, we show that two important cases of swimming problems, namely the locomotion of a single swimmer in an infinite extend of fluid at zero Reynolds number (Subsection~\ref{subsec:Low_Reynolds}) and infinite Reynolds number (Subsection~\ref{subsec:High_Reynolds}) fit within the framework introduced in Subsection \ref{SEC_framework}. Several possible cost functionals are defined in Subsection~\ref{ex:cost_func}.

Section~\ref{sec:examples} is dedicated to the examples. New optimization results for the models of swimmer studied in \cite{Alouges10} and \cite{LoheacScheid11} are provided in Subsection~\ref{subsec:n_spheres} and Subsection~\ref{subsec:legendre} respectively. A study from scratch of a 2D swimmer in a potential flow is carried out in Subsection~ \ref{subsec:potential}. Numerical simulations are presented in Subection~\ref{swim:num}.

For the ease of the reader, we present two short surveys at the end of the paper. The first one deals with Riemannian geometry (Appendix \ref{SEC_appendix_riemannian}) while the second one presents the Orbit theorem (Appendix \ref{SEC_Appendix_Orbit_Th}).

\section{Seeking of optimal strokes}\label{SEC_analysis} 
\subsection{Abstract Framework and Notation} \label{SEC_framework} \label{subsection_Framework and notations} We introduce in this subsection the general framework of our study. In the sequel, we call {\it swimmer} any 5-uple ${\mathscr S}=(\mathcal{S}, \mathbf{g},\mathcal{Q},\mathbf{s}_{\mathsf i},\mathcal{L})$, where: 
\begin{itemize}
	\item ($\mathcal S,\mathbf g$) is a $N$-dimensional ($N\geqslant 1$), connected, analytic manifold (the ``shape manifold'') endowed with an analytic Riemannian structure $\mathbf{g}$. Every element $\mathbf s$ of $\mathcal S$ stands for a possible shape of the swimmer. The shape changes of the swimmer over a time interval $[0,T]$ will be described by a function $\mathbf s:[0,T]\mapsto\mathbf s(t)\in\mathcal S$. 
	\item The metric $\mathbf g$ will be used to measure the cost required to achieve this shape change. The cost of a shape change $\mathbf s:[0,T]\mapsto\mathbf s(t)\in\mathcal S$ could be, for instance, the length of the curve parameterized by the function $\mathbf s$, i.e. 
	\begin{subequations}
		\label{def_cost} 
		\begin{equation}
			\label{cost_1} \int_0^T\|\dot{\mathbf s}(t)\|_{\mathcal S} \,{\rm d}t, 
		\end{equation}
		where $\|\dot{\mathbf s}(t)\|_{\mathcal S} :=\sqrt{\mathbf g_{\mathbf s(t)}(\dot{\mathbf s}(t),\dot{\mathbf s}(t))}$, or something more energy-like, usually called the action: 
		\begin{equation}
			\frac{1}{2}\int_0^T\|\dot{\mathbf s}(t)\|_{\mathcal S}^2\,{\rm d}t. 
		\end{equation}
	\end{subequations}
	\item The reference shape $\mathbf{s}_{\mathsf i}$ is a point of $\mathcal S$ which could be thought of as the natural shape of the swimmer, when it is at rest for instance. It will be the starting point for every shape change we will consider. 
	\item The mapping $\mathcal{Q}:T\mathcal{S}\to\mathbb{R}^n$ is an analytic vector valued 1-form. It accounts for the physical constraints that every shape change has to satisfy to physically make sense. Let us be more specific:
	
	\begin{definition}
		An admissible shape change is any absolutely continuous curve $\mathbf{s}: [0,T]\to\mathcal{S},$ with essentially bounded first derivative, and which satisfies for almost every time, 
		\begin{equation}
			\label{self_prop} \mathcal{Q}_{\mathbf{s}(t)}\dot{\mathbf{s}}(t)=0. 
		\end{equation}
	\end{definition}
	This last identity means that for a given shape (i.e. a given element of $\mathcal S$), not every direction on $\mathcal S$ is admissible. For instance, by self-deforming, the swimmer will not be allowed to modify the position of its center of mass.
	
	Among admissible shape changes, we will mostly focus on strokes: 
	\begin{definition}
		An admissible shape change $\mathbf s:[0,T]\mapsto\mathcal S$ will be termed a stroke if $\mathbf s(0)=\mathbf s(T)$. 
	\end{definition}
	\item We are only interested in the motion of the swimmer in one given direction. The displacement in this direction is measured thanks to the analytic differential 1-form $\mathcal L$ on $\mathcal S$. When undergoing the admissible shape change $\mathbf{s}:[0,T]\to\mathcal{S}$, the displacement of the swimmer is given by: 
	\begin{equation}
		\label{distance_1} \int_0^T {\mathcal L}_{\mathbf{s}(t)}(\dot{\mathbf{s}}(t)) \mathrm{d}t. 
	\end{equation}
\end{itemize}
Most of our results will rest on the following elementary but fundamental observation: 
\begin{remark}
	The constraint \eqref{self_prop} as well as the quantities \eqref{cost_1} and \eqref{distance_1} are time reparameterization invariant. They depend only on the oriented curve $\Gamma\subset\mathcal S$, a parameterization of which being $\mathbf{s}:[0,T]\to\mathcal{S}$. 
\end{remark}

The famous {\it Scallop Theorem} (see for instance \cite{Purcell77}) can be seen as a straightforward consequence of this remark. Indeed, it states that if the shape change is a parameterization back and forth of a curve on $\mathcal S$, then the resulting displacement is null. 

\subsection{Geometric controllability assumptions} \label{SEC:control}

Let us consider an abstract swimmer $\mathscr S =(\mathcal{S}, \mathbf{g},\mathcal{Q},\mathbf{s}_{\mathsf i}, \mathcal{L})$ as described in Section~\ref{subsection_Framework and notations}, and denote, for every $\mathbf s\in\mathcal S$: $$\ker\mathcal Q=\Delta^{\mathcal S}_{\mathbf s}\subset T_{\mathbf s}\mathcal S.$$ Assume that the dimension of $\Delta_{\mathbf s}^{\mathcal S}$ is not always zero (otherwise, it would mean that there is no shape change satisfying the self-propelled constraints \eqref{self_prop}). For every $\mathbf s\in\mathcal S$, we denote by $\{\mathbf X_1(\mathbf s),\ldots,\mathbf X_p(\mathbf s)\}$ ($p>0$) a spanning set of $\Delta_{\mathbf s}^{\mathcal S}$ where the vector fields $\mathbf s\in\mathcal S\mapsto\mathbf X_j(s)\in T_{\mathbf s}\mathcal S$ ($1\leqslant j\leqslant p$) are assumed to be analytic. We denote by $\mathcal X:=\{\mathbf X_j,\,j=1,\ldots,p\}\subset\mathfrak{X}(\mathcal S)$ and we shall call $\mathcal X$ an analytic spanning family of the distribution $\Delta^{\mathcal S}$. Notice that is general, the spanning family $\mathcal X$ cannot be required to be a basis of $\Delta_{\mathbf s}^{\mathcal S}$ at every point $\mathbf s$ of $\mathcal S$ because the analytic vector fields $\mathbf X_j$ cannot be prevented from vanishing at some point of the manifold (see for instance the example in Subsection~\ref{subsec:potential}). 
\begin{proposition}
	Any absolutely continuous function with essentially bounded derivatives $\mathbf{s}:[0,T]\to\mathcal{S}$ is an admissible shape change if and only if it is solution (in the sense of Carath\'eodory) of a Cauchy problem 
	\begin{subequations}
		\label{EQ_dyn_cadre_mobile} 
		\begin{alignat}
			{3} \dot{\mathbf s}(t)&=\sum_{j=1}^p u_j(t)\mathbf X_j(\mathbf s(t))&\quad&(t>0),\\
			\mathbf s(0)&=\mathbf{s}_{\mathsf i}, 
		\end{alignat}
	\end{subequations}
	for some $\mathbf u=(u_1,\ldots,u_p)\in L^\infty([0,T],\mathbb R^p)$. 
\end{proposition}
\begin{proof}
	The proof is elementary: For every admissible shape change, the function $\mathbf u=(u_1,\ldots,u_p)\in L^\infty([0,T],\mathbb{R}^p)$ gives the coordinates of $\dot{\mathbf{s}}$ in the spanning family $\mathcal X(\mathbf s)$. 
\end{proof}

System \eqref{EQ_dyn_cadre_mobile} allows associating with every measurable function $\mathbf u\in L^\infty([0,T],\mathbb{R}^p)$ an admissible shape change, at least for times small enough.

We define $\mathcal M$ as being the analytic $(N+1)$-dimensional manifold $\mathcal S\times\mathbb R$. Then, we introduce the projectors $\Pi_{\mathcal S}$ and $\Pi_{\mathbb R}$ by: $$ 
\begin{array}{rrcl}
	\Pi_{\mathcal S} :& \mathcal{M} & \rightarrow & \mathcal{S}\\
	& (\mathbf s,r) & \mapsto &\Pi_{\mathcal S}(\boldsymbol\xi)=\mathbf s 
\end{array}
\qquad\text{and}\qquad 
\begin{array}{llcl}
	\Pi_{\mathbb{R}} :& \mathcal{M} & \rightarrow & \mathbb{R}\\
	& (\mathbf s,r) & \mapsto &\Pi_{\mathbb R}(\boldsymbol\xi)=r. 
\end{array}
$$ So, for every $\boldsymbol\xi\in\mathcal M$, $\Pi_{\mathcal S} \boldsymbol\xi$ gives the shape of the swimmer and $\Pi_{\mathbb{R}} \boldsymbol\xi$ its position. On $\mathcal M$, we define the analytic vectors fields: 
\begin{equation}
	\label{def_vectors_Z} \mathbf Z_j(\boldsymbol\xi):= 
	\begin{pmatrix}
		\mathbf X_j(\Pi_{\mathcal S}\vecxi)\\
		\mathcal L_\mathbf s\mathbf X_j(\Pi_{\mathcal S}\vecxi) 
	\end{pmatrix}
	,\qquad(j=1,\ldots,p), 
\end{equation}
we denote $\mathcal Z:=\{\mathbf Y_j,\,j=1,\ldots,p\}\subset\mathfrak{X}(\mathcal M)$ and we define the distribution $$\Delta^{\mathcal M}_{\vecxi}={\rm span}\,\mathcal Z(\vecxi),\qquad\vecxi\in\mathcal M.$$ System \eqref{EQ_dyn_cadre_mobile} and the dynamics 

(see Section \ref{SEC:modelling} for details) can be gathered as a unique dynamical system on $\mathcal M$: 
\begin{subequations}
	\label{EQ_dyn_complete} 
	\begin{alignat}
		{3} \dot{\boldsymbol\xi}(t)&=\sum_{j=1}^p u_j(t)\mathbf Z_j(\boldsymbol\xi(t))&\quad&(t>0),\\
		\boldsymbol\xi(0)&=\boldsymbol\xi_{\mathsf i} 
	\end{alignat}
\end{subequations}
where $\boldsymbol\xi_{\mathsf i}=(\mathbf s_{\mathsf i},0)$. 
\begin{remark}
	The choice of the initial condition $(\mathbf{s}_{\mathsf i},0)$ (and not $(\mathbf{s}_{\mathsf i},\delta_{\mathsf i})$ for some $\delta_{\mathsf i}\neq 0$) is physically irrelevant, since the vector fields $\mathbf{Z}_j$ ($1\leqslant j\leqslant p$), and hence the dynamics \eqref{EQ_dyn_complete}, do not depend upon the $\mathbb{R}$ component of the variable $\boldsymbol\xi$. 
\end{remark}
\begin{definition}
	For every positive time $T$, every swimmer ${\mathscr S}=(\mathcal{S}, \mathbf{g}, \mathcal{Q},\mathbf{s}_{\mathsf i}, \mathcal{L})$ and every analytic spanning family $\mathcal X$ of $\Delta^{\mathcal S}$, we denote by $\mathcal{U}^{\mathcal X}_{\mathscr S}(T)$ the set of all the controls $\mathbf{u}=(u_j)_{1\leqslant j \leqslant p}\in L^\infty([0,T],\mathbb R^p)$ for which the solution of \eqref{EQ_dyn_cadre_mobile} (and hence of \eqref{EQ_dyn_complete}) is defined on $[0,T]$. 
\end{definition}

For any given control $\mathbf u\in\mathcal{U}^{\mathcal X}_{\mathscr S}(T)$, we denote $$ t\in[0,T]\mapsto {\boldsymbol \xi}_{\mathscr S}^{\mathcal X}(t,\mathbf{u})\in\mathcal M,$$ the solution to \eqref{EQ_dyn_complete} with control $\mathbf u$. 
\begin{remark}
	According to this notation and the definition \eqref{distance_1} of the displacement, we have: $$\Pi_{\mathbb{R}}\boldsymbol \xi^{\mathcal X}_{\mathscr S}(T,\mathbf{u}) =\int_0^T \mathcal{L}_{\Pi_{\mathcal S}\boldsymbol \xi^{\mathcal X}_{\mathscr S}(t, \mathbf{u})}\frac{\mathrm{d}}{\mathrm{d}t}{\Pi_{\mathcal S}\boldsymbol \xi^{\mathcal X}_{\mathscr S}} (t,\mathbf{u}) \mathrm{d}t.$$ 
\end{remark}
One hypothesis required in order to ensure that the swimmer is controllable is that $\mathcal X$ is bracket generating on $\mathcal S$. Observe that this condition does not depend on the particular choice of the spanning family $\mathcal X$ but only on the distribution $\Delta^{\mathcal S}$ and hence on the vector valued 1-form $\mathcal Q$. It can be easily verified that if $\mathcal X$ and $\mathcal X'$ are two smooth spanning families of $\Delta^{\mathcal S}$, then for every $\mathbf s\in\mathcal S$: $${\rm Lie}_{\mathbf s}\mathcal X={\rm Lie}_{\mathbf s}\mathcal X'={\rm Lie}_{\mathbf s}\Delta^{\mathcal S}.$$ Taking into account this obversation, we introduce: 
\begin{hypo}
	\label{HYP_control} The swimmer ${\mathscr S}=(\mathcal{S}, \mathbf{g},\mathcal{Q},\mathbf{s}_{\mathsf i}, \mathcal{L})$ is such that 
	\begin{enumerate}
		\item The distribution $\Delta^{\mathcal S}$ is bracket generating on $\mathcal S$, i.e. $$\dim{\rm Lie}_{\mathbf s}\Delta^{\mathcal S}=\dim\mathcal S,\quad\forall\,\mathbf s\in\mathcal S;$$ 
		\item There exists $\boldsymbol\xi\in\mathcal M$ such that ${\rm dim}\,{\rm Lie}_{\boldsymbol\xi}\Delta^{\mathcal M}=\dim{\mathcal{M}}$. 
	\end{enumerate}
\end{hypo}
\begin{lemma}
	\label{Lem_lie_bracket} Hypothesis~\ref{HYP_control} leads to: $$\dim{\rm Lie}_{\vecxi}\Delta^{\mathcal M}=\dim\mathcal M,\quad\forall\,\vecxi\in\mathcal M.$$ 
\end{lemma}
\begin{proof}
	Let $\vecxi^\ast=(\mathbf s^\ast,0)$ be such that ${\rm dim}\,{\rm Lie}_{\boldsymbol\xi^\ast}\Delta^{\mathcal M}=\dim{\mathcal{M}}$. As already mentioned, the choice of $0$ for the $\mathbb R$ component of $\vecxi^\ast$ is irrelevant regarding the Lie algebra ${\rm Lie}_{\boldsymbol\xi^\ast}\Delta^{\mathcal M}={\rm Lie}_{\boldsymbol\xi^\ast}\mathcal Z$ since, for every $j=1,\ldots,p$, $\mathbf Z_j(\vecxi)=\mathbf Z_j(\Pi_{\mathcal S}\vecxi)$. Consider now any $\vecxi=(\mathbf s,r)\in\mathcal M$ and denote by $\mathcal O(\vecxi)$ the orbit of $\mathcal Z$ through $\vecxi$. Since $\Delta^{\mathcal S}$ is bracket generating on $\mathcal S$, Rashevsky-Chow Theorem ensures that for any $T>0$, there exists a control $\mathbf u\in L^\infty([0,T],\mathbb R^p)$ such that the solution to the EDO \eqref{EQ_dyn_cadre_mobile} with Cauchy data $\mathbf s$ is equal to $\mathbf s^\ast$ at the final time $T$. Using this control in EDO \eqref{EQ_dyn_complete} with Cauchy data $\vecxi=(\mathbf s, r)$, we deduce that the solution reaches a point $\tilde{\vecxi}^{\ast}=(\mathbf s^\ast,r^\ast)$ at the time $T$ for some $r^\ast\in\mathbb R$ ($\vecxi$ and $\tilde{\vecxi}^\ast$ are both in $\mathcal O(\vecxi)$). But since $\Pi_{\mathcal S}\tilde{\vecxi}^{\ast}=\Pi_{\mathcal S}{\vecxi}^{\ast}$, we have the equality ${\rm dim}\,{\rm Lie}_{\tilde{\vecxi}^{\ast}}\mathcal Z={\rm dim}\,{\rm Lie}_{{\vecxi}^{\ast}}\mathcal Z=\dim{\mathcal{M}}$. According to the Orbit Theorem, the dimension of the Lie algebra of $\mathcal Z$ is constant on $\mathcal O(\vecxi)$ and hence we have also ${\rm dim}\,{\rm Lie}_{\boldsymbol\xi}\mathcal Z=\dim{\mathcal{M}}$. The proof is now completed. 
\end{proof}

Assuming only, in the second point of Hypothesis~\ref{HYP_control}, that the equality holds for one point of $\mathcal M$ may seem a somehow useless mathematical refinement. Quite the reverse, the explicit computation of ${\rm Lie}_{\boldsymbol\xi}\Delta^{\mathcal M}$ in concrete cases is often very involved and can still be hardly carried out for one particular $\vecxi$ (see for instance \cite{Alouges10,AlougesGiraldi12,ChambrionMunnier11,Gerard-VaretGiraldi13,LoheacMunnier12}). 
\begin{definition}
	A swimmer satisfying Hypothesis~\ref{HYP_control} will be called {\it controllable}. 
\end{definition}
This definition is justified by the following Theorem: 
\begin{theorem}
	\label{theo_control} Let ${\mathscr S}$ be a swimmer satisfying Hypothesis~\ref{HYP_control}. Then, for every $T>0$, every analytic spanning family $\mathcal X$ of $\Delta^{\mathcal S}$, every $\boldsymbol\xi_{\mathsf f}$ in $\mathcal M$ and every open, connected set $\mathcal O\subset\mathcal M$ containing $\vecxi_{\mathsf i}$ and $\vecxi_{\mathsf f}$, there exists a control $\mathbf u\in L^\infty([0,T],\mathbb{R}^p)$ such that $\boldsymbol\xi^{\mathcal X}_{\mathscr S}(T,\mathbf{u})=\boldsymbol\xi_{\mathsf f}$ and $\vecxi^{\mathcal X}_{\mathscr S}(t,\mathbf{u})\in\mathcal O$ for every $t\in[0,T]$. 
\end{theorem}
\begin{proof}
	This is a straightforward consequence of the analytic Orbit theorem. 
\end{proof}

Notice that this theorem applies for the models (high and low Reynolds numbers swimmers) introduced in Section~\ref{SEC:modelling}. For these models, controllability is ensured as soon as Hypothesis~\ref{HYP_control} is fulfilled. 

\subsection{Statement of optimal problems}\label{SEC_optimal_problems}

In this section, we address the main problems that we are interested in. A controllable swimmer ${\mathscr S}$ being given, and a cost being chosen (among those presented in \eqref{def_cost}), what is the {\it best} possible stroke? By {\it best}, it is understood that the swimmer is wished to swim as far as possible with a corresponding cost as low as possible.

To be rigorously stated, the question has to be split into several closely related but not always equivalent problems: 
\begin{enumerate}
	\item What is the stroke minimizing the cost among those allowing traveling a given, fixed distance? 
	\item What is the stroke maximizing the travelled distance among those whose cost is not greater than a given fixed bound? 
\end{enumerate}
In the case where the cost is not important, we can also be interested in seeking the stroke maximizing the mean swimming velocity, or the simple stroke (i.e., stroke without self intersection) that allows the maximum traveled distance.

We shall conduct a detailed study on every one of these problems, focusing on the existence of optimal strokes and deriving their main properties. 

To begin with, let us restrict slightly the scope of our study by introducing a new hypothesis that the swimmer has to satisfy: 
\begin{hypo}
	\label{Hypo:2} The swimmer $\mathscr S =(\mathcal{S}, \mathbf{g},\mathcal{Q},\mathbf{s}_{\mathsf i}, \mathcal{L})$ is such that there exists an analytic basis $\mathcal X=\{\mathbf X_j,j=1,\ldots,p\}$ of the distribution $\Delta^{\mathcal S}$. 
\end{hypo}
\begin{definition}
	A swimmer $\mathscr S$ satisfying Hypothesis~\ref{Hypo:2} will be termed trivialized. 
\end{definition}
Notice that most of the models of swimmers considered in the literature satisfy this hypothesis (it is the case for the swimmers in \cite{Alouges10} and \cite{LoheacScheid11}).

Applying a Gram-Schmidt process, we can assume that for every $\mathbf s\in\mathcal S$, the family $\{\mathbf X_j(\mathbf s),j=1,\ldots,p\}$ in Hypothesis~\ref{Hypo:2} is an orthonormal basis (for the Riemannian scalar product $\mathbf g$ of $\mathcal S$) of $\Delta^{\mathcal S}_{\mathbf s}$. As already mentioned before, it is in general not possible to extract from any smooth spanning family of $\Delta^{\mathcal S}$ a smooth basis on the whole manifold $\mathcal S$. Nevertheless, any swimmer can be locally trivialized: 
\begin{proposition}
	Let $\mathscr S =(\mathcal{S}, \mathbf{g},\mathcal{Q},\mathbf{s}_{\mathsf i}, \mathcal{L})$ be a swimmer. Then, there exists an open connected subset $\mathcal S'$ (for the topology of $\mathcal S$) containing $\mathbf{s}_{\mathsf i}$ such that $\mathscr S' =(\mathcal{S}', \mathbf{g},\mathcal{Q},\mathbf{s}_{\mathsf i}, \mathcal{L})$ is a trivialized swimmer. 
\end{proposition}
Notice in particular that any open subset of an analytic manifold is still an analytic manifold.

Let $\mathscr S =(\mathcal{S}, \mathbf{g},\mathcal{Q},\mathbf{s}_{\mathsf i}, \mathcal{L})$ be a trivialized, controllable swimmer, ${\mathcal K}$ be a compact of $\mathcal S$ containing $\mathbf{s}_{\mathsf i}$ and $\mathcal X$ be an orthonormal basis of $\Delta^{\mathcal S}$. For every $\boldsymbol\xi_{\mathsf f}$ in $\mathcal M$ and $T\geqslant 0$, we define the following control sets: 
\begin{align*}
	\mathcal U^{\mathcal X}_{\mathscr S}(\boldsymbol\xi_{\mathsf f},T)&:= \left\{ \mathbf u\in \mathcal{U}^{\mathcal X}_{\mathscr S}(T) \,:\,{\boldsymbol\xi}_{\mathscr S}^{\mathcal X}(T,\mathbf u)=\boldsymbol\xi_{\mathsf f} \right\};\\
	\widehat{\mathcal U}^{\mathcal X}_{\mathscr S}(T)&:= \left\{ \mathbf u\in \mathcal U^{\mathcal X}_{\mathscr S}(T) \,:\,\|\mathbf u(t)\|_{\mathbb R^p}=1\,\,\forall\,t\in[0,T]\right\};\\
	\mathcal B^{\mathcal X}_{{\mathscr S},{\mathcal K}}(T)&:= \left\{\mathbf u\in \mathcal U^{\mathcal X}_{\mathscr S}(T) \,:\, \mathbf s(t):=\Pi_{\mathcal S}{\boldsymbol\xi}_{\mathscr S}^{\mathcal X}(t,\mathbf u) \in {\mathcal K} \quad \forall t\in [0,T]\right\}. 
\end{align*}
The set $\mathcal U^{\mathcal X}_{\mathscr S}(\boldsymbol\xi_{\mathsf f},T)$ consists in the controls $\mathbf u$ which steer the swimmer from the state $\vecxi_{\mathsf i}$ to the state $\vecxi_{\mathsf f}$. The set $\widehat{\mathcal U}^{\mathcal X}_{\mathscr S}(T)$ contains the controls for which the rate of shape changes is normalized. Notice that, since $\mathcal X$ is assumed to be orthonormal, we have for every $t\in[0,T]$: $$\|\mathbf u(t)\|_{\mathbb R^p}=\|\dot{\mathbf s}(t)\|_{\mathcal S},$$ where $\mathbf s(t):=\Pi_{\mathcal S}{\boldsymbol\xi}_{\mathscr S}^{\mathcal X}(t,\mathbf u)$. Finally, in $\mathcal B^{\mathcal X}_{{\mathscr S},{\mathcal K}}(T)$, the constraint taken into account is that the shape variable of the swimmer remains at every moment in a given compact of $\mathcal S$ (this is a constraint on the state of the system).

We will be more precisely interested in the following sets of controls: 
\begin{align*}
	\widehat{\mathcal U}^{\mathcal X}_{\mathscr S}(\boldsymbol\xi_{\mathsf f},T)&:= \mathcal U^{\mathcal X}_{\mathscr S}(\boldsymbol\xi_{\mathsf f},T)\cap \widehat{\mathcal U}^{\mathcal X}_{\mathscr S}(T);\\
	\mathcal B^{\mathcal X}_{{\mathscr S},{\mathcal K}}(\boldsymbol\xi_{\mathsf f},T)&:= \mathcal U^{\mathcal X}_{\mathscr S}(\boldsymbol\xi_{\mathsf f},T)\cap \mathcal B^{\mathcal X}_{{\mathscr S},{\mathcal K}}(T);\\
	\widehat{\mathcal B}^{\mathcal X}_{{\mathscr S},{\mathcal K}}(\boldsymbol\xi_{\mathsf f},T)&:= \mathcal U^{\mathcal X}_{\mathscr S}(\boldsymbol\xi_{\mathsf f},T)\cap \widehat{\mathcal U}^{\mathcal X}_{\mathscr S}(T)\cap \mathcal B^{\mathcal X}_{{\mathscr S},{\mathcal K}}(T). 
\end{align*}

The following Lemma holds true: 
\begin{lemma}
	\label{lem_exist} If there exists an orthonormal basis $\mathcal X$ of $\Delta^{\mathcal S}$ and $T>0$ such that the set ${\mathcal U}^{\mathcal X}_{{\mathscr S},{\mathcal K}}(\vecxi_{\mathsf f},T)$ is empty, then it is empty for every orthonormal basis $\mathcal X$ of $\Delta^{\mathcal S}$ and every $T>0$. 
\end{lemma}
\begin{proof}
	
	Saying that $\mathcal B^{\mathcal X}_{{\mathscr S},{\mathcal K}}(\vecxi_{\mathsf f},T)$ is nonempty means that there exists an allowable curve on $\mathcal M$, whose projection on $\mathcal S$ is contained in ${\mathcal K}$ and which links $\vecxi_{\mathsf i}$ to $\vecxi_{\mathsf f}$. The existence of such a curve depends neither on $\mathcal X$ nor on $T$. 
\end{proof}

We are now in position to state the optimization problems. Since we are interested in seeking optimal strokes, the endpoint $\vecxi_{\mathsf f}$ will always have the form $\vecxi_{\mathsf f}=(\mathbf s_{\mathsf i},\delta_{\mathsf f})$ for some given $\delta_{\mathsf f}\in\mathbb R$. The compact ${\mathcal K}\subset\mathcal S$ and the time $T>0$ are given as well. 
\begin{problem}
	[Minimizing the Riemannian length] \label{PROB_Moins_cher} To determine: 
	\begin{equation}
		\label{def_distance} \Phi^{\mathcal X}_{\mathscr S,{\mathcal K}}(\delta_{\mathsf f},T)=\inf\left\{\int_0^T\|\mathbf u(t)\|_{\mathbb R^p}\,{\rm d}t\,:\,\mathbf u\in\mathcal U_{{\mathscr S},{\mathcal K}}^{\mathcal X}(\boldsymbol\xi_{\mathsf f},T)\right\}. 
	\end{equation}
\end{problem}
As already mentioned earlier, denoting $\mathbf s(t):=\Pi_{\mathcal S}\boldsymbol\xi^{\mathcal X}_{\mathscr S}(t,\mathbf u)$, we have: $$\int_0^T\|\mathbf u(t)\|_{\mathbb R^p}\,{\rm d}t=\int_0^T\|\dot{\mathbf s}(t)\|_{\mathcal S}\,{\rm d}t,$$ which is the length of the curve $\Gamma\subset{\mathcal K}$ parameterized by $t\in[0,T]\mapsto \mathbf s(t)\in\mathcal S$. The lengths of the curves on $\mathcal S$ do not depend on the parameterization, so Problem~\ref{PROB_Moins_cher} is actually time parameterization invariant.

Modifying the cost leads to: 
\begin{problem}
	[Minimizing the action] \label{PROB_action} To determine: 
	\begin{equation}
		\label{def_action} \Theta^{\mathcal X}_{\mathscr S,{\mathcal K}}(\delta_{\mathsf f},T)=\inf\left\{\frac{1}{2}\int_0^T\|\mathbf u(t)\|^2_{\mathbb R^p}\,{\rm d}t\,:\,\mathbf u\in\mathcal U_{{\mathscr S},{\mathcal K}}^{\mathcal X}(\boldsymbol\xi_{\mathsf f},T)\right\}. 
	\end{equation}
\end{problem}
Remark that, unlike the cost in \eqref{def_distance}, the cost in \eqref{def_action} is not time parameterization independent. 

\begin{problem}
	[Optimizing the time] \label{PROB_time} To determine: $$T^{\mathcal X}_{\mathscr S,{\mathcal K}}(\delta_{\mathsf f})=\inf\{T\,:\,\widehat{\mathcal U}_{\mathscr S,{\mathcal K}}^{\mathcal X}(\boldsymbol\xi_{\mathsf f},T)\neq\varnothing\}.$$ 
\end{problem}
\begin{problem}
	[Maximizing traveling distance with bounded Riemannian lenght] \label{PROB_Plus_loin} For any $l\geqslant 0$ and $T\geqslant 0$, determine: 
	\begin{equation}
		\label{def_plus_loin} \Psi^{\mathcal X}_{\mathscr S,{\mathcal K}}(l,T)=\sup \left\{\Pi_{\mathbb{R}}\boldsymbol \xi_{\mathscr S}^{\mathcal X}(T,\mathbf{u})\,:\,\mathbf u\in\mathcal B^{\mathcal X}_{{\mathscr S},{\mathcal K}}(T), \,\Pi_{\mathcal S}\xi_{\mathscr S}^{\mathcal X}(T,\mathbf{u})=\mathbf s_{\mathsf i}\text{ and }\int_0^T\|\mathbf u(t)\|_{\mathbb R^p}\,{\rm d}t\leqslant l\right\}. 
	\end{equation}
\end{problem}
\begin{problem}
	[Maximizing traveling distance with bounded action] \label{PROB_Plus_loin2} For any $l\geqslant 0$ and $T\geqslant 0$, determine: 
	\begin{equation}
		\label{def_plus_loin2} \Lambda^{\mathcal X}_{\mathscr S,{\mathcal K}}(l,T)=\sup \left\{\Pi_{\mathbb{R}}\boldsymbol \xi_{\mathscr S}^{\mathcal X}(T,\mathbf{u})\,:\,\mathbf u\in\mathcal B^{\mathcal X}_{{\mathscr S},{\mathcal K}}(T), \,\Pi_{\mathcal S}\xi_{\mathscr S}^{\mathcal X}(T,\mathbf{u})=\mathbf s_{\mathsf i} \text{ and }\frac{1}{2}\int_0^T\|\mathbf u(t)\|^2_{\mathbb R^p}\,{\rm d}t\leqslant l\right\}. 
	\end{equation}
\end{problem}

\subsection{Firsts Properties of the Optimal Strokes}\label{SEC_existence_minimizers} In this Subsection, we will derive properties of the optimal strokes resting on the Riemannian structure of $\mathcal S$. Then, in the following Subsection, we will introduce an make use of the sub-Riemannian structure of $\mathcal M$.

To begin with, let us focus on the firsts three problems, making the convention that the infimum of an empty set is equal to $+\infty$: 
\begin{theorem}
	\label{main_theorem_properties} Let ${\mathscr S}=(\mathcal{S}, \mathbf{g},\mathcal{Q}, \mathbf{s}_{\mathsf i},\mathcal{L})$ be a controllable, trivialized swimmer, and ${\mathcal K}$ be a compact of $\mathcal S$ containing $\mathbf s_{\mathsf i}$. Then 
	\begin{enumerate}
		\item For every $\delta_{\mathsf f}\in\mathbb R$, the quantities $\Phi^{\mathcal X}_{{\mathscr S},{\mathcal K}}(\delta_{\mathsf f},T)$, $\Theta^{\mathcal X}_{{\mathscr S},{\mathcal K}}(\delta_{\mathsf f},T)$, and $T^{\mathcal X}_{{\mathscr S},{\mathcal K}}(\delta_{\mathsf f})$ are either all of them infinite or all of them finite, for every $T\geqslant 0$ and every orthonormal basis $\mathcal X$ of $\Delta^{\mathcal S}$. 
		\item If $\mathbf s_{\mathsf i}\in\mathring{{\mathcal K}}$ then $\Phi^{\mathcal X}_{{\mathscr S},{\mathcal K}}(\delta_{\mathsf f},T)$, $\Theta^{\mathcal X}_{{\mathscr S},{\mathcal K}}(\delta_{\mathsf f},T)$, and $T^{\mathcal X}_{{\mathscr S},{\mathcal K}}(\delta_{\mathsf f})$ are all of them finite, for every $T\geqslant 0$, every orthonormal basis $\mathcal X$ of $\Delta^{\mathcal S}$ and every $\delta_{\mathsf f}\in\mathbb R$. 
		\item If $\Phi^{\mathcal X}_{{\mathscr S},{\mathcal K}}(\delta_{\mathsf f},T)$, $\Theta^{\mathcal X}_{{\mathscr S},{\mathcal K}}(\delta_{\mathsf f},T)$, and $T^{\mathcal X}_{{\mathscr S},{\mathcal K}}(\delta_{\mathsf f})$ are finite, then there exist minimizers or maximizers to every Problem~\ref{PROB_Moins_cher}, \ref{PROB_action} and \ref{PROB_time}. 
		\item For every $T\geqslant 0$ and every $\delta_{\mathsf f}\in\mathbb R$, $\Phi^{\mathcal X}_{{\mathscr S},{\mathcal K}}(\delta_{\mathsf f},T)$, $\Theta^{\mathcal X}_{{\mathscr S},{\mathcal K}}(\delta_{\mathsf f},T)$ and $T^{\mathcal X}_{{\mathscr S},{\mathcal K}}(\delta_{\mathsf f})$ do not depend on $\mathcal X$. So from now on, we drop $\mathcal X$ in the notation. 
		\item $\Phi_{{\mathscr S},{\mathcal K}}(\delta_{\mathsf f},T)$ does not depend on $T$. So from now on, we drop $T$ in the notation. 
		\item The following identities hold for every $\delta_{\mathsf f}\in\mathbb R$, and every $T>0$: 
		\begin{subequations}
			\label{rel_phi_theta} 
			\begin{align}
				\Phi_{{\mathscr S},{\mathcal K}}(\delta_{\mathsf f})&=T_{{\mathscr S},{\mathcal K}}(\delta_{\mathsf f})\\
				\Theta_{{\mathscr S},{\mathcal K}}(\delta_{\mathsf f},T)&=(1/2)(T_{{\mathscr S},{\mathcal K}}(\delta_{\mathsf f}))^2/T 
			\end{align}
		\end{subequations}
		\item Any minimizer $\mathbf u\in\mathcal B^{\mathcal X}_{{\mathscr S},{\mathcal K}}(\vecxi_{\mathsf f},T)$ to Problem~\ref{PROB_action}: 
		\begin{enumerate}
			\item is such that $\|\mathbf u(t)\|_{\mathbb R^p}$ is constant at every moment; 
			\item is also a minimizer to Problem~\ref{PROB_Moins_cher}; 
			\item is proportional to a minimizer of Problem~\ref{PROB_time}. 
		\end{enumerate}
	\end{enumerate}
\end{theorem}
\begin{proof}
	\begin{enumerate}
		\item According to Lemma~\ref{lem_exist}, if ${\mathcal U}^{\mathcal X}_{{\mathscr S},{\mathcal K}}(\vecxi_{\mathsf f},T)$ is empty for some $T>0$, then it is empty for every $T>0$ and therefore every quantity $\Phi^{\mathcal X}_{{\mathscr S},{\mathcal K}}(\delta_{\mathsf f},T)$, $\Theta^{\mathcal X}_{{\mathscr S},{\mathcal K}}(\delta_{\mathsf f},T)$, and $T^{\mathcal X}_{{\mathscr S},{\mathcal K}}(\delta_{\mathsf f})$ is infinite. Reciprocally, if for some $T>0$ there exists $\mathbf u\in {\mathcal U}^{\mathcal X}_{{\mathscr S},{\mathcal K}}(\vecxi_{\mathsf f},T)$, then, for every $T>0$, the set ${\mathcal U}^{\mathcal X}_{{\mathscr S},{\mathcal K}}(\vecxi_{\mathsf f},T)$ is nonempty as well and $\Phi^{\mathcal X}_{{\mathscr S},{\mathcal K}}(\delta_{\mathsf f},T)$ and $\Theta^{\mathcal X}_{{\mathscr S},{\mathcal K}}(\delta_{\mathsf f},T)$ are both finite. Moreover, from the control $\mathbf u$, we can build $\tilde{\mathbf u}\in \widehat{\mathcal B}^{\mathcal X}_{{\mathscr S},{\mathcal K}}(\vecxi_{\mathsf f},T')$ for any $T'>0$ by setting: 
		\begin{subequations}
			\label{repara} 
			\begin{alignat}
				{3} \alpha&=\frac{\|\mathbf u\|_{L^1([0,T],\mathbb R^p)}}{T'};\\
				\phi(t)&=\frac{1}{\alpha}\int_0^t\|\mathbf u(s)\|_{\mathbb R^p}\,{\rm d}s,&\qquad&t\in[0,T];\\
				\tilde{\mathbf u}(t)&=\alpha\frac{\mathbf u(\phi^{-1}(t))}{\|\mathbf u(\phi^{-1}(t))\|_{\mathbb R^p}},&&t\in[0,T']. 
			\end{alignat}
		\end{subequations}
		We deduce that $T^{\mathcal X}_{{\mathscr S},{\mathcal K}}(\delta_{\mathsf f})$ is finite and the first assertion of the Theorem is proved. 
		\item Denote by $\mathcal O_1$ the connected component of $\mathring{{\mathcal K}}$ containing $\mathbf s_{\mathsf i}$ and for every $\delta_{\mathsf f}\in\mathbb R$, take $\mathcal O=\mathcal O_1\times]-|\delta_{\mathsf f}|-1,|\delta_{\mathsf f}|+1[$ in Theorem~\ref{theo_control}. The Theorem ensures that for every $T>0$ and every orthonormal basis $\mathcal X$ of $\Delta^{\mathcal S}$, the set $\mathcal B^{\mathcal X}_{{\mathscr S},{\mathcal K}}(\vecxi_{\mathsf f},T)$ is nonempty. 
	\end{enumerate}
	We prove now all the remaining points of the Theorem. Let $\mathcal X$ (an orthonormal basis of $\Delta^{\mathcal S}$), $T>0$ and $\delta_{\mathsf f}\in\mathbb R$ be given. For any control $\mathbf u\in \mathcal B^{\mathcal X}_{{\mathscr S},{\mathcal K}}(\vecxi_{\mathsf f},T)$, denote by $\tilde{\mathbf u}\in \mathcal B^{\mathcal X}_{{\mathscr S},{\mathcal K}}(\vecxi_{\mathsf f},T)$ the control defined in \eqref{repara} with $T'=T$. One can easily verify that: 
	\begin{align}
		\int_0^T\|\mathbf u(s)\|_{\mathbb R^p}\,{\rm d}s&=\int_0^T\|\tilde{\mathbf u}(s)\|_{\mathbb R^p}\,{\rm d}s\label{same_value} \intertext{and} \int_0^T\|\mathbf u(s)\|^2_{\mathbb R^p}\,{\rm d}s&\geqslant\frac{1}{T}\left(\int_0^T\|\mathbf u(s)\|_{\mathbb R^p}\,{\rm d}s\right)^2\label{CS}\\
		&=\int_0^T\|\tilde{\mathbf u}(s)\|^2_{\mathbb R^p}\,{\rm d}s,\nonumber 
	\end{align}
	with equality in \eqref{CS} if and only if $\|\mathbf u(s)\|_{\mathbb R^p}$ is constant, i.e. $\mathbf u=\tilde{\mathbf u}$. So, replacing $\mathbf u$ by $\tilde{\mathbf u}$ does not modify the cost functional of Problem~\ref{PROB_Moins_cher} and does not increase the cost functional of Problem~\ref{PROB_action}. Moreover, since 
	\begin{equation}
		\label{same_value_for_both} \int_0^T\|\tilde{\mathbf u}(s)\|^2_{\mathbb R^p}\,{\rm d}s=\frac{1}{T}\left(\int_0^T\|\tilde{\mathbf u}(s)\|_{\mathbb R^p}\,{\rm d}s\right)^2, 
	\end{equation}
	if $(\mathbf u_n)_n$ is a minimizing sequence for either Problem~\ref{PROB_Moins_cher} or Problem~\ref{PROB_action}, then $(\tilde{\mathbf u}_n)_n$ is a minimizing sequence for both Problem~\ref{PROB_Moins_cher} and Problem~\ref{PROB_action}. By construction, the sequence $(\tilde{\mathbf u}_n)_n$ is bounded in $L^\infty([0,T],\mathbb R^p)$. Hence, up to a subsequence extraction, we can assume that $(\tilde{\mathbf u}_n)_n$ weakly converges, for instance, in $L^2([0,T],\mathbb R^p)$ toward $\mathbf u^\ast$. In particular, the following inequality holds: 
	\begin{equation}
		\int_0^T\|\mathbf u^\ast(s)\|^2_{\mathbb R^p}\,{\rm d}s\leqslant\liminf_{n\to+\infty}\int_0^T\|\tilde{\mathbf u}_n(s)\|^2_{\mathbb R^p}\,{\rm d}s. 
	\end{equation}
	Let us verify that $\mathbf u^\ast\in\mathcal B^{\mathcal X}_{{\mathscr S},{\mathcal K}}(\vecxi_{\mathsf f},T)$.
	
	The functions $t\in[0,T]\mapsto\vecxi^{\mathcal X}_{\mathscr S}(t,\tilde{\mathbf u}_n)\in\mathcal M$ are equi-Lipschitz continuous on $[0,T]$, because the vector fields $\mathbf Z_i$ are analytic on the compact ${\mathcal K}$ and hence bounded. According to Ascoli Theorem, we can assume that, up to a subsequence extraction, the sequence $(t\mapsto\vecxi^{\mathcal X}_ {\mathscr S}(t,\tilde{\mathbf u}_n))_n$ converges uniformly on $[0,T]$ toward a Lipschitz continuous function $t\in[0,T]\mapsto\vecxi^\ast\in\mathcal M$. Furthermore, the curve $t\in[0,T]\mapsto \Pi_ {\mathcal S}\boldsymbol \xi^\ast\in\mathcal S$ is absolutely continuous, with bounded derivative and thus is an admissible shape change (with support in ${\mathcal K}$). For every $t\in[0,T]$ and every $n\in\mathbb N$, we have: 
	\begin{equation}
		\label{rela} \boldsymbol\xi^{\mathcal X}_{\mathscr S}(t,\tilde{\mathbf u}_n)= \vecxi_{\mathsf i}+\sum_{i=1}^p\int_0^t \tilde u_i^n(s)\mathbf Z_i(\vecxi^{\mathcal X}_{\mathscr S}(s,\tilde{\mathbf u}_n))\,{\rm d}s. 
	\end{equation}
	Since $\tilde u_i^n\rightharpoonup u_i^\ast$ in $L^2([0,T],\mathbb R)$ and $\mathbf Z_i(\vecxi^{\mathcal X}_{\mathscr S}(\cdot,\tilde{\mathbf u}_n))\to \mathbf Z_i(\vecxi^\ast)$ uniformly on $[0,T]$ (and hence also in $L^2([0,T],T\mathcal M)$), passing to the limit as $n\to+\infty$ in \eqref{rela} leads to: $$\boldsymbol\xi^\ast(t)=\vecxi_{\mathsf i}+\sum_{i=1}^p\int_0^t u_i^\ast(s)\mathbf Z_i(\vecxi^\ast(s))\,{\rm d}s,\quad t\in[0,T].$$ We have now proved that $\mathbf u^\ast$ is indeed a minimizer to Problems~\ref{PROB_action}. Moreover, since equality in \eqref{CS} holds if and only if $\|\mathbf u^\ast(t)\|_{\mathbb R^p}$ is constant for every $t\in[0,T]$, we infer that $\mathbf u^\ast=\tilde{\mathbf u}^\ast$. This equality leads to the following estimates: 
	\begin{align*}
		\int_0^T\|\mathbf u^\ast(s)\|_{\mathbb R^p}\,{\rm d}s&=\sqrt{T}\left(\int_0^T\|\mathbf u^\ast(s)\|^2_{\mathbb R^p}\,{\rm d}s\right)^{1/2}\\
		&\leqslant\liminf_{n\to+\infty}\sqrt{T}\left(\int_0^T\|\tilde{\mathbf u}_n(s)\|^2_{\mathbb R^p}\,{\rm d}s\right)^{1/2}\\
		&=\liminf_{n\to+\infty}\int_0^T\|\tilde{\mathbf u}_n(s)\|_{\mathbb R^p}\,{\rm d}s, 
	\end{align*}
	and $\mathbf u^\ast$ is also a minimizer to Problem~\ref{PROB_Moins_cher}. Using this control in \eqref{same_value_for_both}, we obtain the equality: 
	\begin{equation}
		\label{relab} \Theta^{\mathcal X}_{{\mathscr S},{\mathcal K}}(\delta_{\mathsf f},T)=\frac{1}{2T}(\Phi^{\mathcal X}_{{\mathscr S},{\mathcal K}}(\delta_{\mathsf f},T))^2. 
	\end{equation}
	Eventually, as already mentioned earlier, for every $T>0$ and from any control $\mathbf u\in\mathcal U_{{\mathscr S},{\mathcal K}}^{\mathcal X}(\boldsymbol\xi_{\mathsf f},T)$, we can build $\tilde{\mathbf u}\in \widehat{\mathcal B}^{\mathcal X}_{{\mathscr S},{\mathcal K}}(T',\vecxi_{\mathsf f})$ with $T':=\|\mathbf u\|_{L^1([0,T],\mathbb R^p)}$ by setting $\alpha=1$ in \eqref{repara}. The identity \eqref{same_value} becomes: $$\int_0^T\|\mathbf u(s)\|_{\mathbb R^p}\,{\rm d}s=\int_0^{\|\mathbf u\|_{L^1([0,T],\mathbb R^p)}}\,{\rm d}s,$$ whence we deduce that $$T^{\mathcal X}_{{\mathscr S},{\mathcal K}}(\delta_{\mathsf f})=\Phi^{\mathcal X}_{{\mathscr S},{\mathcal K}}(\delta_{\mathsf f},T).$$ This equality tells us that $\Phi^{\mathcal X}_{{\mathscr S},{\mathcal K}}(\delta_{\mathsf f},T)$ does not depend on $T$. We conclude the proof of the theorem by observing again that $\Phi^{\mathcal X}_{{\mathscr S},{\mathcal K}}(\delta_{\mathsf f},T)$ is the length of the curve on $\mathcal S$ parameterized by $t\in[0,T]\mapsto\Pi_{\mathcal S}\vecxi_{{\mathscr S},{\mathcal K}}^{\mathcal X}(t,\mathbf u^\ast)\in\mathcal S$ and that this length does not depend on $\mathcal X$. 
\end{proof}

We address now Problems~\ref{PROB_Plus_loin} and \ref{PROB_Plus_loin2}: 
\begin{theorem}
	\label{main_theorem_properties_1} Let ${\mathscr S}=(\mathcal{S}, \mathbf{g},\mathcal{Q}, \mathbf{s}_{\mathsf i},\mathcal{L})$ be a controllable, trivialized swimmer, and ${\mathcal K}$ be a compact of $\mathcal S$ containing $\mathbf s_{\mathsf i}$. Then 
	\begin{enumerate}
		\item Problems~\ref{PROB_Plus_loin} and \ref{PROB_Plus_loin2} admit maximizers for every $T\geqslant 0$, every orthonormal basis $\mathcal X$ of $\Delta^{\mathcal S}$ and every $l\geqslant 0$. 
		\item Any maximizer $\mathbf u\in\mathcal B^{\mathcal X}_{{\mathscr S},{\mathcal K}}(\vecxi_{\mathsf f},T)$ to Problem~\ref{PROB_Plus_loin2} 
		\begin{enumerate}
			\item is such that $\|\mathbf u(t)\|_{\mathbb R^p}$ is constant at every moment; 
			\item is proportional to a maximizer of Problem~\ref{PROB_Plus_loin}. 
		\end{enumerate}
		\item $\Psi_{{\mathscr S},{\mathcal K}}^\mathcal X(l,T)$ and $\Lambda_{{\mathscr S},{\mathcal K}}^\mathcal X(l,T)$ do not depend on $\mathcal X$ so we drop it in the notation. 
		\item $\Psi_{{\mathscr S},{\mathcal K}}(l,T)$ does not depend on $T$ so we drop it in the notation. 
		\item The following identity holds for every $T>0$ and $l\geqslant 0$: 
		\begin{equation}
			\label{equal_phi_lam} \Lambda_{{\mathscr S},{\mathcal K}}((1/2)l^2/T,T)=\Psi_{{\mathscr S},{\mathcal K}}(l). 
		\end{equation}
	\end{enumerate}
\end{theorem}
\begin{proof}
	To prove the existence of maximizers to Problems~\ref{PROB_Plus_loin} and \ref{PROB_Plus_loin2}, we follow the lines of the proof of Theorem~\ref{main_theorem_properties}: Let $\mathcal X$, $l\geqslant 0$ and $T\geqslant 0$ be given and consider first a maximizing sequence $(\mathbf u_n)_n$ to Problem~\ref{PROB_Plus_loin2}. Then notice that the renormalized and reparameterized control $(\tilde{\mathbf u}_n)_n$ is actually not only a maximizing sequence to Problem~\ref{PROB_Plus_loin2} but also to Problem~\ref{PROB_Plus_loin} with $l'=\sqrt{lT}$. Then, up to subsequences extractions and invoking Ascoli Theorem and the weak convergence in $L^2([0,T],\mathbb R^p)$ of $(\tilde{\mathbf u}_n)_n$, we prove the existence of a common maximizer $\mathbf u^\ast=\tilde{\mathbf u}^\ast$ to Problems~\ref{PROB_Plus_loin2} (with $l$) and \ref{PROB_Plus_loin} (with $l'=\sqrt{lT}$). Using the control $\mathbf u^\ast$, we also get the equality \eqref{equal_phi_lam}.
	
	Once more, the time reparameterization invariance of Problem~\ref{PROB_Plus_loin} leads to infer that $\Psi_{{\mathscr S},{\mathcal K}}^{\mathcal X}$ does not depend neither on $\mathcal X$ nor $T$. Eventually, identity \eqref{equal_phi_lam} ensures that $\Lambda_{{\mathscr S},{\mathcal K}}^{\mathcal X}(l,T)$ does not depend on $\mathcal X$. 
\end{proof}

\subsection{Further Properties of the Optimal Strokes}\label{SEC_continuity_value_function} \label{sub_riemannian_struct} In order to prove further properties on Problems~\ref{PROB_Moins_cher}-\ref{PROB_Plus_loin2}, we need to introduce the Sub-Riemannian structure on $\mathcal M$.

Let a trivialized swimmer ${\mathscr S}=(\mathcal{S}, \mathbf{g},\mathcal{Q}, \mathbf{s}_{\mathsf i},\mathcal{L})$ and $\mathcal X:=\{\mathbf X_j, j=1,\ldots,p\}$ an orthonormal basis of $\Delta^{\mathcal S}_{\mathbf s}$ be given. From the analytic vectors fields $ \mathbf X_j$ on $\mathcal S$, we build the analytic vector fields $\mathbf Z_j$ ($j=1,\ldots,p$) on $\mathcal M=\mathcal S\times\mathbb R$ as described in \eqref{def_vectors_Z}. Then, we define $$\mathcal Z=\{\mathbf Z_j,j=1,\ldots,p\},$$ and the distribution on $\mathcal M$: $$\Delta^{\mathcal M}_{\vecxi}:={\rm span}\,\mathcal Z(\vecxi)\subset T_{\boldsymbol\xi}\mathcal M,\qquad\vecxi\in\mathcal M.$$ We denote by $\mathbb Z(\boldsymbol\xi)$ the matrix whose column vectors are the $\mathbf Z_j(\vecxi)$ and we introduce the Euclidean bundle $\mathbf U:=\mathcal M\times\mathbb R^p$ endowed with the Euclidean norm of $\mathbb R^p$ and the morphism of vector bundles $$f:(\vecxi,\mathbf u)\in\mathbf U\mapsto (\vecxi,\mathbb Z(\vecxi)\mathbf u)\in T\mathcal M.$$ Following \cite[Definition 3.1]{AgrachevBarilari12}, we claim that the manifold $\mathcal M$ endowed with the triple $(\mathcal M,\mathbf U,f)$ is an analytic sub-Riemannian manifold. According to Definition 3.6 from the same booklet, we define the admissible curves as being the Lipschitz curves $\vecxi:[0,T]\mapsto\mathcal M$ for which there exists a control function $\mathbf u\in L^\infty([0,T],\mathbb R^p)$ such that, for a.e. $t\in[0,T]$: $$\dot\vecxi=\mathbb Z(\vecxi)\mathbf u.$$ Notice that it is exactly the dynamics \eqref{EQ_dyn_complete} that we are dealing with. The sub-Riemannian manifold is equipped with the so-called {\it Carnot-Caratheodory distance} (see \cite[Definition 3.13]{AgrachevBarilari12}) denoted by $\,{\rm d}(\cdot,\cdot)$. In particular, the following identity holds: $$\,{\rm d}(\vecxi_{\mathsf i},\vecxi_{\mathsf f})=\inf\left\{\int_0^T\|\mathbf u(s)\|_{\mathbb R^p}\,{\rm d}s,\,\mathbf u\in\mathcal U_{\mathscr S}^{\mathcal X}(\vecxi_{\mathsf f},T)\right\}.$$ Be aware that actually, neither $\mathcal X$ nor $T$ matters in this definition. 
\begin{theorem}
	\label{main_theorem_properties_2} Let ${\mathscr S}=(\mathcal{S}, \mathbf{g},\mathcal{Q}, \mathbf{s}_{\mathsf i},\mathcal{L})$ be a controllable, trivialized swimmer, and ${\mathcal K}$ be a compact of $\mathcal S$ such that $\mathbf s_{\mathsf i}\in\mathring{{\mathcal K}}$. Then: 
	\begin{enumerate}
		\item For every $l>0$, we have $\Psi_{{\mathscr S},{\mathcal K}}(l)>0$. For every $l>0$ and $T>0$, we have $\Lambda_{{\mathscr S},{\mathcal K}}(l,T)>0$. 
		\item The function 
		\begin{subequations}
			\begin{equation}
				\label{function_psi} l\in\mathbb R_+\mapsto \Psi_{{\mathscr S},{\mathcal K}}(l)\in\mathbb R_+ 
			\end{equation}
			is increasing and right continuous. 
			\item For every $T>0$, the function 
			\begin{equation}
				\label{function_lam} l\in\mathbb R_+\mapsto \Lambda_{{\mathscr S},{\mathcal K}}(l,T)\in\mathbb R_+ 
			\end{equation}
			is increasing and right continuous. 
			\item For every maximizer to Problems~\ref{PROB_Plus_loin} or \ref{PROB_Plus_loin2}, the constraints are saturated. 
			\item The function 
			\begin{equation}
				\label{function_phi} \delta\in\mathbb R\mapsto \Phi_{{\mathscr S},{\mathcal K}}(\delta)\in\mathbb R 
			\end{equation}
			is even and uniformly continuous. 
			\item For every $T>0$, the function 
			\begin{equation}
				\label{function_theta} \delta\in\mathbb R\mapsto \Theta_{{\mathscr S},{\mathcal K}}(\delta,T)\in\mathbb R 
			\end{equation}
		\end{subequations}
		is even and uniformly continuous. 
		\item For every $l\geqslant 0$ and every $T\geqslant 0$: 
		\begin{subequations}
			\begin{align}
				\Phi_{{\mathscr S},{\mathcal K}}(\Psi_{{\mathscr S},{\mathcal K}}(l))&=l\\
				\Theta_{{\mathscr S},{\mathcal K}}(\Lambda_{{\mathscr S},{\mathcal K}}(l,T),T)&=l. 
			\end{align}
		\end{subequations}
	\end{enumerate}
\end{theorem}
Regarding the last point of the theorem, notice that, for every $\delta_{\mathsf f}\in\mathbb R_+$, we have $\Psi_{{\mathscr S},{\mathcal K}}(\Phi_{{\mathscr S},{\mathcal K}}(\delta_{\mathsf f}))\geqslant\delta_{\mathsf f}$ but it may happen that $\Psi_{{\mathscr S},{\mathcal K}}(\Phi_{{\mathscr S},{\mathcal K}}(\delta_{\mathsf f}))>\delta_{\mathsf f}$ for some $\delta_{\mathsf f}$. Indeed, according to the definition of $\Psi_{{\mathscr S},{\mathcal K}}(\Phi_{{\mathscr S},{\mathcal K}}(\delta_{\mathsf f}))$, we have: $$ \Psi_{{\mathscr S},{\mathcal K}}(\Phi_{{\mathscr S},{\mathcal K}}(\delta_{\mathsf f}))=\max\left\{\Pi_{\mathbb R}\vecxi_{\mathscr S}^{\mathcal X}(T,\mathbf u)\,:\,\mathbf u\in\mathcal U_{{\mathscr S},{\mathcal K}}^{\mathcal X}(T),\, \Pi_{\mathcal S}\vecxi_{\mathscr S}^{\mathcal X}(T,\mathbf u)=\mathbf s_{\mathsf i}, \int_0^T\|\mathbf u(s)\|_{\mathbb R^p}\,{\rm d}s\leqslant \Phi_{{\mathscr S},{\mathcal K}}(\delta_{\mathsf f})\right\}. $$ In words: with the minimal cost allowing to reach the point $\delta_{\mathsf f}$, it is possible to go further that $\delta_{\mathsf f}$. This is not as counterintuitive as it seems: indeed, it can be costly to maneuver in order to stop exactly at the point $\delta_{\mathsf f}$. 
\begin{proof}
	\begin{enumerate}
		\item For any $a\geqslant 0$ and $\vecxi\in\mathcal M$, denote by $B_{\mathcal M}(\vecxi,a)$ the sub-Riemannian ball, of radius $a$ and centered at $\vecxi$. 
		
		According to \cite[Theorem 3.8]{AgrachevBarilari12}, the Carnot-Caratheodory distance induces the manifold topology on $\mathcal M$. We deduce that, for every $l>0$, the set $B_{\mathcal M}(\vecxi_{\mathsf i},l)\cap (\mathring{{\mathcal K}}\times\mathbb R)$ is open an nonempty (it contains $\vecxi_{\mathsf i}=(\mathbf s_{\mathsf i},0)$) and hence $$(\{\mathbf s_{\mathsf i}\}\times\mathbb R)\cap B_{\mathcal M}(\vecxi_{\mathsf i},l)\cap (\mathring{{\mathcal K}}\times\mathbb R)$$ contains a set $\{\mathbf s_{\mathsf i}\}\times]-\varepsilon,\varepsilon[$ for some $\varepsilon>0$. We infer that, for every $l>0$, we have $\Psi_{{\mathscr S},{\mathcal K}}(l)>\varepsilon>0.$ We use the relation \eqref{equal_phi_lam} to deduce that, for every $l>0$ and $T>0$, we also have $\Lambda_{{\mathscr S},{\mathcal K}}(l,T)>0$. 
		\item The function \eqref{function_psi} is clearly nondecreasing. Moreover, for every $l,l'\geqslant 0$, we have: 
		\begin{equation}
			\label{supper_additive} \Psi_{{\mathscr S},{\mathcal K}}(l+l')\geqslant \Psi_{{\mathscr S},{\mathcal K}}(l)+\Psi_{{\mathscr S},{\mathcal K}}(l'). 
		\end{equation}
		Indeed, recall that every minimizer to Problem~\ref{PROB_Plus_loin} is time parameterization invariant and consider the curves $\Gamma\subset\mathcal S$ of length $l$ and $\Gamma'\subset\mathcal S$ of length $l'$ corresponding to the maximizers of Problem~\ref{PROB_Plus_loin} for $l$ and $l'$ respectively. Then denote $\Gamma''=\Gamma\cup\Gamma'$. This curve is admissible, closed, of length $l+l'$ and produces a displacement no greater than $\Psi_{{\mathscr S},{\mathcal K}}(l+l')$. Inequality \eqref{supper_additive} together with the first point of the Theorem yield the increasing property of function \eqref{function_psi}.
		
		In order to prove the right continuity, let $l\in\mathbb R_+$ be given and consider a decreasing sequence $(l_n)_n$ converging to $l$. For every $l_n$ ($n\in\mathbb N$), denote by $\mathbf u_n$ the minimizer to Problem~\ref{PROB_Plus_loin} such that $\|\mathbf u_n(t)\|_{\mathbb R^p}$ is constant for every $t\in[0,T]$. The sequence $(\mathbf u_n)_n$ is bounded in $L^2([0,T],\mathbb R^p)$ and the sequence $(\vecxi_{\mathscr S}^{\mathcal X}(\cdot,\mathbf u_{n}))_n$ is bounded in $C^0([0,T],\mathcal M)$, therefore there exists a subsequence $(l_{n_k})_k$ such that $(\Psi_{{\mathscr S},{\mathcal K}}(l_{n_k}))_k$ converges to $\limsup \Psi_{{\mathscr S},{\mathcal K}}(l_n)$ while $(\mathbf u_{n_k})_k$ weakly converges in $L^2([0,T],\mathbb R^p)$ to $\mathbf u^\ast$ and $(\vecxi_{\mathscr S}^{\mathcal X}(\cdot,\mathbf u_{n_k}))_k$ uniformly converges to some $\vecxi^\ast\in C^0([0,T],\mathcal M)$.
		
		On the one hand, arguing as for \eqref{rela}, we deduce that $\vecxi^\ast=\vecxi_{\mathscr S}^{\mathcal X}(\cdot,\mathbf u^\ast)$ and then that $$\Psi_{{\mathscr S},{\mathcal K}}(l_{n_k})\to\Pi_{\mathbb R}\vecxi_{\mathscr S}^{\mathcal X}(T,\mathbf u^\ast)\quad\text{as }k\to+\infty$$ with $\|\mathbf u^\ast\|_{L^1([0,T],\mathbb R^p)}\leqslant \liminf l_n=l$. Therefore we obtain that: $$\limsup\Psi_{{\mathscr S},{\mathcal K}}(l_n)\leqslant \Psi_{{\mathscr S},{\mathcal K}}(l).$$ On the other hand, since $\Psi_{{\mathscr S},{\mathcal K}}$ is increasing and $l\leqslant l_n$ for every $n$, we deduce that $$\Psi_{{\mathscr S},{\mathcal K}}(l)\leqslant \Psi_{{\mathscr S},{\mathcal K}}(l_n),$$ and the conclusion arises by taking the limit inf in both sides of the inequality. 
		\item The proof of this point is a straightforward consequence of the preceding point and relation \eqref{equal_phi_lam}.
		
		\item The same reasoning as for the second point proves that the constraints are saturated in Problems~\ref{PROB_Plus_loin} or \ref{PROB_Plus_loin2} for the maximizers. Indeed, if the constraints were not saturated, then it would be possible to add to the optimal curve on $\mathcal S$ a small loop that would produce an extra displacement (according to the first point).
		
		\item For $\varepsilon$ small enough, we have $B_{\mathcal M}(\vecxi_{\mathsf i},\varepsilon)\subset\mathring{{\mathcal K}}$ (because the sub-Riemannian topology co\"\i ncides with the manifold topology). In this case, for every $\delta_{\mathsf f}\in\Pi_{\mathbb R}^{-1} B_{\mathcal M}(\vecxi_{\mathsf i},\varepsilon)$, we have: $$\Phi_{{\mathscr S},{\mathcal K}}(\delta_{\mathsf f})=\,{\rm d}(\vecxi_{\mathsf i},\vecxi_{\mathsf f}).$$ According to \cite[Theorem 3.18]{AgrachevBarilari12}, the Carnot-Caratheodory distance is continuous for the manifold topology, so we deduce that the function \eqref{function_phi} is continuous in a neighborhood of $0$. It is even because, for every $\delta$ and every minimizer $\mathbf u$ on $[0,T]$, the control $t\mapsto\mathbf u(T-t)$ is a minimizer with the same Riemannian length, associated to $-\delta$.
		
		Observe now that for every $\delta_{\mathsf f},h\in\mathbb R$: $$\Phi_{{\mathscr S},{\mathcal K}}(\delta_{\mathsf f}+h)\leqslant \Phi_{{\mathscr S},{\mathcal K}}(\delta_{\mathsf f})+\Phi_{{\mathscr S},{\mathcal K}}(h),$$ whence we infer that: $$\Phi_{{\mathscr S},{\mathcal K}}(\delta_{\mathsf f}+h)-\Phi_{{\mathscr S},{\mathcal K}}(\delta_{\mathsf f})\leqslant \Phi_{{\mathscr S},{\mathcal K}}(h).$$ Writing that $$\Phi_{{\mathscr S},{\mathcal K}}(\delta_{\mathsf f})\leqslant \Phi_{{\mathscr S},{\mathcal K}}(\delta_{\mathsf f}+h)+\Phi_{{\mathscr S},{\mathcal K}}(-h),$$ and since the function is even, we finally get: $$|\Phi_{{\mathscr S},{\mathcal K}}(\delta_{\mathsf f}+h)-\Phi_{{\mathscr S},{\mathcal K}}(\delta_{\mathsf f})|\leqslant \Phi_{{\mathscr S},{\mathcal K}}(h),$$ which, with the continuity in $0$, proves that the function is uniformly continuous on $\mathbb R$. 
		\item The proof of this point is a straightforward consequence of the preceding point and relation \eqref{rel_phi_theta}.
		
		\item For every $l\geqslant 0$, we clearly have $\Phi_{{\mathscr S},{\mathcal K}}(\Psi_{{\mathscr S},{\mathcal K}}(l))\leqslant l$. The inequality $\Phi_{{\mathscr S},{\mathcal K}}(\Psi_{{\mathscr S},{\mathcal K}}(l))< l$ for some $l\geqslant 0$ and the existence of minimizers to Problem~\ref{PROB_Moins_cher} would contradict the fact that the constraint is saturated for every maximizer of Problem~\ref{PROB_Plus_loin}. 
	\end{enumerate}
\end{proof}
\begin{remark}
	The distance-cost function \eqref{function_phi} is not monotonic in the general case (see Subsection~\ref{swim:num}, Fig.~\ref{fig225}). The cost-distance function \eqref{function_psi} is not continuous in the general case (see Subsection~\ref{swim:num}, Fig.~\ref{fig226}). 
\end{remark}
Let us investigate further the monotonicity of the distance-cost function \eqref{function_phi}. To do so, we need to introduce a new hypothesis: 
\begin{hypo}
	\label{HYPO:pas_de_contraintes} The swimmer ${\mathscr S}=(\mathcal{S}, \mathbf{g},\mathcal{Q}, \mathbf{s}_{\mathsf i},\mathcal{L})$ is such that $\mathcal{Q}=\mathbf 0$ (there is no self-propelled constraints). 
\end{hypo}
Every swimmer satisfying \eqref{HYPO:pas_de_contraintes} is called {\it unconstrained }. For instance, the swimmers in \cite{AlougesDeSimone08} and \cite{LoheacScheid11} are unconstrained. For unconstrained swimmers, every absolutely continuous curve on $\mathcal S$, with essentially bounded first derivative, is an admissible shape change. Under Hypothesis~\ref{HYPO:pas_de_contraintes}, we can state: 
\begin{theorem}
	\label{main_theorem_properties_3} Let ${\mathscr S}=(\mathcal{S}, \mathbf{g},\mathcal{Q}, \mathbf{s}_{\mathsf i},\mathcal{L})$ be a controllable, trivialized, unconstrained swimmer, and ${\mathcal K}$ be a compact of $\mathcal S$ such that $\mathbf s_{\mathsf i}\in\mathring{{\mathcal K}}$. Then there exists $\varepsilon>0$ such that the function \eqref{function_phi} is increasing on $[0,\varepsilon[$. 
\end{theorem}
\begin{proof}
	Denote by $B_{\mathcal S}(\mathbf s_{\mathsf i},r)$ the Riemannian ball on $\mathcal S$, where the radius $r$ is given by Lemma~\ref{LEM_retract_monotonic}. Let $\varepsilon>0$ be small enough such that $B_{\mathcal M}(\vecxi_{\mathsf i},\varepsilon)\subset\mathring{{\mathcal K}}$ (it is always possible because the sub-Riemannian topology co\"\i ncides with the manifold topology) and $\Pi_{\mathcal S}B_{\mathcal M}(\vecxi_{\mathsf i},\varepsilon)\subset B_{\mathcal S}(\mathbf s_{\mathsf i},r)$.
	
	Assume now that there exist $0<\delta_0<\delta_1<\varepsilon$ such that 
	\begin{equation}
		\label{hypp} \Phi_{{\mathscr S},{\mathcal K}}(\delta_1)\leqslant \Phi_{{\mathscr S},{\mathcal K}}(\delta_0). 
	\end{equation}
	Denote $\vecxi^1=(\mathbf s_{\mathsf i},\delta_1)$ and for some $\mathcal X=\{\mathbf X_j,\,j=1,\ldots,p\}$ (an orthonormal basis of $\Delta^{\mathcal S}$) and $T>0$, denote by $\mathbf u^1\in\mathcal B^{\mathcal X}_{{\mathscr S},{\mathcal K}}(\vecxi^1,T)$ a control minimizing Problem~\ref{PROB_Moins_cher}. Introduce as well $\mathbf s^1=\Pi_{\mathcal S}\vecxi(\cdot,\mathbf u^1)$ and $\Gamma_1$ the curve on $\mathcal S$ parameterized by $\mathbf s^1$. The following identity holds: $$\Phi_{{\mathscr S},{\mathcal K}}(\delta_1)=\int_0^T\|\mathbf u^1(t)\|_{\mathbb R^p}\,{\rm d}t=\ell(\Gamma_1).$$ According to Lemma~\ref{LEM_retract_monotonic} with $x_0=\mathbf s_{\mathsf i}$ and $\gamma=\mathbf s^1$, there exists a continuous function $\psi:[0,1]\times[0,T]\to\mathcal S$ such that, for every $s\in[0,1]$, $\psi(s,\cdot)$ is absolutely continuous with essentially bounded first derivative, $\psi(1,\cdot)=\mathbf s^1$ and $\psi(0,\cdot)=\mathbf s_{\mathsf i}$. Denoting by $\mathbf u^s=(u^s_j)_{1\leqslant j\leqslant p}\in\mathcal B^{\mathcal X}_{{\mathscr S},{\mathcal K}}(T)$ the control such that $$u_j^s(t)=\mathbf g_{\psi(s,t)}( 
	\partial_t\psi(s,t),\mathbf X_j(\psi(s,t)),\quad t\in[0,T],$$ and by $\Gamma_s$ the curve parameterized by $t\in[0,T]\mapsto\psi(s,t)=\Pi_{\mathcal S}\vecxi(t,\mathbf u^s)$ ($s\in[0,1]$), we have, for every $s\in[0,1]$: $$\Phi_{{\mathscr S},{\mathcal K}}(\Xi(s))\leqslant \int_0^T\|\mathbf u^s(t)\|_{\mathbb R^p}\,{\rm d}t=\ell(\Gamma_s),$$ where we have set: $$\Xi:s\in [0,1]\mapsto \Pi_{\mathbb R}\vecxi^{\mathcal X}_{\mathscr S}(T,\mathbf u^s)\in\mathbb R.$$ The function $\Xi$ is continuous and such that $\Xi(1)=\delta_1$ and $\Xi(0)=0$, so there exists $s^\ast\in]0,1[$ such that $\Xi(s^\ast)=\delta_0$. Since, according to Lemma~\ref{LEM_retract_monotonic}, the function $s\in[0,1]\mapsto\ell(\Gamma_s)$ is increasing, we get: $$\Phi_{{\mathscr S},{\mathcal K}}(\delta_0)\leqslant \ell(\Gamma_{s^\ast})<\ell(\Gamma_1)=\Phi_{{\mathscr S},{\mathcal K}}(\delta_1),$$ which is in contradiction with \eqref{hypp}. The proof is now completed. 
\end{proof}

\subsection{Optimal Strokes and Isoperimetric Inequalities} \label{caseN=2} In this section, we focus on the case where the swimmer is controllable, trivialized, unconstrained and where the dimension $N$ of the manifold $\mathcal S$ is equal to $2$. The three balls swimmer of the article \cite{AlougesDeSimone08} and the example of swimmer in a perfect fluid given in subsection~\ref{subsec:potential} fulfilled these requirements.

We wish to give a hint of how the optimal stroke problem can be interpretable as an isoperimetric problem on the manifold $\mathcal S$.

Recall that a stroke is a closed (oriented) admissible curve $\Gamma$ on $\mathcal S$. Let $\mathbf s:[0,T]\to \mathcal S$ be a parameterization of $\Gamma$. The traveled distance resulting from this stroke is: $$\int_0^T\mathcal L_{\mathbf s(t)}\dot{\mathbf s}(t)\,{\rm d}t=\int_{\Gamma}\mathcal L.$$ Denoting by $\Omega$ the area enclosed by $\Gamma$, one gets from Stokes formula: $$ \int_{\Gamma}\mathcal L=\int_{\Omega}\,{\rm d}\mathcal L.$$ Notice that these formula are metric independent but require $\mathcal S$ to be orientable. The $2$-form $\,{\rm d}\mathcal L$ defined on the 2 dimensional manifold $\mathcal S$ can be seen as a signed measure on $\mathcal S$. 

An example of cost functional considered in this paper is just the Riemannian length of $\Gamma$. For this cost, seeking optimal strokes consists in minimizing the Riemannian length of $\Gamma$ while maximizing the measure of the area $\Omega$ for the signed measure $\,{\rm d}\mathcal L$. 
\begin{definition}
	A stroke $\Gamma \subset \mathcal{S}$ is \emph{simple} if there exists a parametrization $\mathbf{s}:[0,T]\to \mathcal{S}$ of $\Gamma$ such that $t_1\neq t_2$ implies $\mathbf{s}(t_1)\neq \mathbf{s}(t_2)$ for every $(t_1,t_2)$ in $[0,T]\times (0,T)$. We denote by ${\mathsf {Si}}$ the set of simple strokes on $\mathcal{S}$. 
\end{definition}
Let us consider a new problem: 
\begin{problem}
	[Optimizing the traveled distance] \label{PROB_simple} To maximize the traveled distance among simple strokes (no cost functional involved). 
\end{problem}
The solution is given in the following proposition: 
\begin{proposition}
	\label{PRO_stroke_long_maximale} Let $f:\mathcal{S}\to \mathbb{R}$ be the density of the measure $\,{\rm d}\mathcal L$ (i.e. $\,{\rm d}\mathcal{L}=f \mathbf g$). Then the supremum $$\sup_{\Gamma \in {\mathsf {Si}}} \int_{\Gamma}\mathcal{L},$$ is a maximum if and only if the set $f^{-1}(\{0\})$ is a connected compact submanifold of $\mathcal{S}$. In this case, the maximum is reached for the unique stroke $f^{-1}(\{0\})$ (endowed with the suitable orientation). 
\end{proposition}
\begin{proof}
	The following equality holds true: $$\sup_{\Gamma \in {\mathsf {Si}}} \int_{\Gamma}\mathcal{L} = \int_{f^{-1}([0,+\infty))} \,{\rm d}\mathcal L.$$ Indeed, for any simple stroke $\Gamma$, it is clear that $$\displaystyle{\int_\Gamma \mathcal{L}\leqslant \int_{f^{-1}([0,+\infty))} \,{\rm d}\mathcal L},$$ with equality if, and only if, $\Gamma$ is the boundary of the set $f^{-1}((0,+\infty))$. This is the case when $f^{-1}(\{0\})$ is a connected compact submanifold of $\mathcal{S}$. If $f^{-1}(\{0\})$ is not connected, a maximizing sequence of simple strokes can be build by approximating (in Hausdorf sense) the set $f^{-1}((0,+\infty))$ with a connected simple domain (see Figure \ref{FIG_approximation_domaine}). 
\end{proof}
\begin{figure}
	\centerline{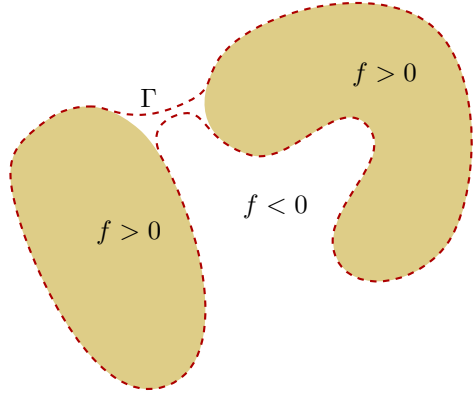} \caption{Approximation of the double connected boundary $f^{-1}(\{0\})$ by a simple stroke.} \label{FIG_approximation_domaine} 
\end{figure}

We insist on the fact that the result of Proposition~\ref{PRO_stroke_long_maximale} does not depend on the Riemannian structure $\mathbf{g}$ chosen on the manifold $\mathcal{S}$. The simple stroke with maximal displacement, if any, is fully characterized by the manifold $\mathcal{S}$ and the 1-form $\mathcal{L}$.

\section{Modeling} \label{MathematicalSetting}

\label{SEC:modelling} In this Section we aim to establish the dynamics governing the motion of low and high Reynolds numbers swimmers and show that they fall within the abstract framework introduced in Section~\ref{SEC_framework}. 

Our purpose it to highlight that, although the properties of the fluid are different in both cases, the equations of motion turn out to have the exact same general form. The modeling is carried out in 3D, the 2D case being similar. 

We assume that the swimmer is alone in the fluid and that the fluid-swimmer system fills the whole space. The buoyant force is not taken into account. 

We consider a Galilean fixed frame $(\mathbf E_1,\mathbf E_2,\mathbf E_3)$ and an attached coordinate system $\mathcal R_f:=(O,\mathbf E_1,\mathbf E_2,\mathbf E_3)$ where $O$ is a fixed point of $\mathbb R^3$. In $\mathcal R_f$, coordinates will be written with capital letters (as $X:=(X_1,X_2,X_3)$). To simplify somehow, we will consider only shape changes that make the center of mass of the swimmer remains on the $X_1$-axis. Thus, the swimmer is compelled to swim along a straight line. 

\subsubsection*{Kinematics} We denote by $\mathbf r$ the center of mass of the swimmer (lying on the $X_1$-axis) and we consider the coordinate system $\mathcal R_m:=(\mathbf r,\mathbf E_1,\mathbf E_2,\mathbf E_3)$, attached to the swimmer. 

We assume that every possible shape of the swimmer, when described in $\mathcal R_m$, can be characterized by a so-called shape variable $\mathbf s$ belonging to some connected analytic hypersurface $\mathcal S$ of $\mathbb R^{N+1}$ (for some integer $N\geqslant 1$).

Thus, we denote by $\mathcal A_m(\mathbf s)$ the domain of $\mathbb R^3$ occupied by the swimmer, in the coordinate system $\mathcal R_m$, and hence $\mathcal A_f(\mathbf s)=\mathbf r+\mathcal A_m(\mathbf s)$ is the same domain expressed in $\mathcal R_f$. For every $\mathbf s\in\mathcal S$, the set $\mathcal A_m(\mathbf s)$ is the image of the unit ball $B$ by a $C^1$ diffeomorphism $\chi(\mathbf s,\cdot)$ depending on the parameter $\mathbf s$. Knowing every diffeomorphism $\chi(\mathbf s,\cdot)$ for every shape variable $\mathbf s$, the shape changes over a time interval $[0,T]$ can be merely described by means of a (shape) function: $$t\in[0,T]\mapsto\mathbf s(t)\in\mathcal S.$$ The Eulerian velocity $\mathbf W$ at any point $X\in\mathcal A_f(\mathbf s)$ of the swimmer is the sum of the rigid velocity $\dot{\mathbf r}:=\dot r\mathbf E_1$ ($\dot r\in\mathbb R$) and the velocity of deformation $$\mathbf W_d(\mathbf s,\dot{\mathbf s},X):=\nabla_{\mathbf s}\chi\left({\mathbf s},\chi({\mathbf s},(X-\mathbf r))^{-1}\right)\cdot\dot{\mathbf s}.$$ Thus, we get: $$\mathbf W=\dot{\mathbf r}+\mathbf W_d\qquad \text{in }\mathcal A_f(\mathbf s).$$ In the coordinate system $\mathcal R_m$, this equality turns out to be: $$\mathbf w=\dot{\mathbf r}+\mathbf w_d\qquad \text{in }\mathcal A_m(\mathbf s),$$ where, for every $x\in\mathcal A_m(\mathbf s)$ we have set $\mathbf w(x):=\mathbf W(x+\mathbf r)$ and $\mathbf w_d(\mathbf s,\dot{\mathbf s},x):=\mathbf W_d(\mathbf s,\dot{\mathbf s},x+\mathbf r)$. 

\subsubsection*{Dynamics} In $\mathcal R_m$, the density $\varrho(\mathbf s,\cdot)$ of the body can be deduced from a given constant density $\varrho_0>0$, defined in $B$, according to the conservation of mass principle: $$\varrho(\mathbf s,\chi({\mathbf s},x))=\frac{\varrho_0}{|\det \nabla_x\chi(\mathbf s,x)|},\qquad x\in B.$$ The volume of the swimmer is ${\rm Vol}(\mathbf s)=\int_B|\det \nabla_x\chi(\mathbf s,x)|\,{\rm d}x$ and its mass $m=\varrho_0{\rm Vol}(B)$.

Although prescribed, the deformations should be interpretable as produced by some internal forces. It means that in the absence of fluid, the swimmer is not able to modify its linear momentum, which reads: 
\begin{equation}
	\label{self_prop_1} \int_{\mathcal A_m(\mathbf s)}\varrho(\mathbf s,x)\mathbf w_d(\mathbf s,\dot{\mathbf s},x)\,{\rm d}x=\varrho_0\int_{B}\nabla_{\mathbf s}\chi({\mathbf s},x)\cdot\dot{\mathbf s}\,{\rm d}x=\mathbf 0. 
\end{equation}
We introduce the $3\times N$ matrix: 
\begin{equation}
	\label{def_Q} {\mathcal Q}(\mathbf s):=\varrho_0\int_{B}\nabla_{\mathbf s}\chi({\mathbf s},x){\rm d}x, 
\end{equation}
and we rewrite \eqref{self_prop} as: 
\begin{equation}
	\label{self_prop_1} {\mathcal Q}(\mathbf s)\dot{\mathbf s}=\mathbf 0. 
\end{equation}
This equation has to be understood as a constraint on the shape variable and is referred to as the {\it self-propulsion hypothesis}. 

The fluid obeys, in the whole generality, to the Navier-Stokes equations for incompressible fluid: 
\begin{subequations}
	\label{NS_equations} 
	\begin{alignat}
		{3} \varrho_f\frac{D}{Dt}\mathbf U(t,X)-\nabla_X\cdot{\mathbb T}_f(\mathbf U,P)(t,X)&=0&\qquad&t>0,\,X\in\mathcal F_f(\mathbf s(t));\\
		\nabla_X\cdot\mathbf U(t,X)&=0&\qquad&t>0,\,X\in\mathcal F_f(\mathbf s(t)); 
	\end{alignat}
\end{subequations}
where 
\begin{enumerate}
	\item For every $\mathbf s\in\mathcal S$, $\mathcal F_f(\mathbf s):=\mathbb R^3\setminus\overline{{\mathcal A}_f(\mathbf s)}$ is the domain of the fluid; 
	\item $\varrho_f>0$ is the fluid's density; 
	\item $\mathbf U(t,X)$ is the Eulerian velocity of the fluid at the time $t>0$ and the point $X\in\mathcal F_f(\mathbf s(t))$; 
	\item $D/Dt:= 
	\partial / 
	\partial t+(\mathbf U(t,X)\cdot\nabla_X)$ is the convective derivative; 
	\item ${\mathbb T}_f(\mathbf U,P)(t,X):= \mu (\nabla_X\mathbf U(t,X)+\nabla_X\mathbf U^T(t,X))-P(t,X)\mathbb I\,{\rm d}$ is the stress tensor, $\mu$ the dynamic viscosity and $P$ the pressure. 
\end{enumerate}
The rigid displacement of the body is governed by Newton's laws for the linear momentum: $$m\ddot r(t)=-\int_{ 
\partial{\mathcal A_f(\mathbf s)}}\mathbf E_1\cdot{\mathbb T}_f(\mathbf U,P)(t,X) \mathbf n\,{\rm d}\sigma_X,\qquad(t>0),$$ where $\mathbf n$ is the unit normal vector to $ 
\partial\mathcal A_f(\mathbf s)$ directed towards the interior of $\mathcal A_f(\mathbf s)$. 

These equations have to be supplemented with boundary conditions on $ 
\partial\mathcal A_f(\mathbf s)$, which can be either $$\mathbf U\cdot \mathbf n =\mathbf W\cdot \mathbf n\qquad\text{on } 
\partial\mathcal A_f(\mathbf s),$$ known as {\it slip} or Navier boundary conditions or $$\mathbf U =\mathbf W\qquad\text{on } 
\partial\mathcal A_f(\mathbf s),$$ referred to as {\it no-slip} boundary conditions. Eventually, for the system to be well-posed, initial data are needed: $$\mathbf U(0)=\mathbf U_0,\,\mathbf r(0)=\mathbf \delta_{\mathsf i}\text{ and }\dot{\mathbf r}(0)=\dot{\mathbf r}_0.$$

As mentioned in the introduction, we focus on two limit problems connecting to the value of the Reynolds number ${\rm Re}:=\varrho \bar{\mathbf U}L/\mu$ ($\bar{\mathbf U}$ is the mean fluid velocity and $L$ is a characteristic linear dimension). The first case ${\rm Re}\ll1$ concerns low Reynolds swimmers like bacteria (or more generally so-called micro swimmers whose size is about $1\mu m$). For the second ${\rm Re}\gg1$, we will restrain our study to irrotational flows and so it is relevant for large animals swimming quite slowly, a case where vorticity can be neglected. 
\subsection{Low Reynolds numbers swimmers} \label{subsec:Low_Reynolds}

For micro-swimmers, scientists agree that inertia (for both the fluid and the body) can be neglected in the dynamics. It means that in the modeling, we can set $\varrho_0=\varrho_f=0$. In this case, the Navier-Stokes equations reduce to the steady Stokes equations 
\begin{alignat*}
	{3} -\nabla_X\cdot{\mathbb T}_f(\mathbf U,P)(t,X)&=0&\qquad&t>0,\,X\in\mathcal F_f(\mathbf s(t));\\
	\nabla_X\cdot\mathbf U(t,X)&=0&\qquad&t>0,\,X\in\mathcal F_f(\mathbf s(t)); 
\end{alignat*}
supplemented with no-slip boundary conditions $$\mathbf U=\mathbf W\text{ on } 
\partial\mathcal A_f(\mathbf s).$$ Introducing, for all $x\in\mathcal F_m(\mathbf s):=\mathbb R^3\setminus\overline{{\mathcal A}_m(\mathbf s)}$, $$\mathbf u(t,x):=\mathbf U(t,x+\mathbf r(t))\quad\text{ and }\quad p(t,x):=P(t,x+\mathbf r(t)),$$ the equations keep the same form when expressed in the coordinate system $\mathcal R_m$, namely, with evident notation: 
\begin{subequations}
	\label{stokes_eq} 
	\begin{alignat}
		{3} \label{stokes_1} -\nabla_x\cdot{\mathbb T}_m(\mathbf u,p)(t,x)&=0&\qquad&t>0,\,x\in\mathcal F_m(\mathbf s(t));\\
		\label{stokes_2} \nabla_x\cdot\mathbf u(t,x)&=0&\qquad&t>0,\,x\in\mathcal F_m(\mathbf s(t));\\
		\label{stokes_3} \mathbf u(t,x)&=\mathbf w(t,x)&\qquad&t>0,\,x\in 
		\partial\mathcal A_m(\mathbf s(t)). 
	\end{alignat}
\end{subequations}
From a mathematical point of view, the advantage is two folds: 
\begin{enumerate}
	\item The equations are now linear; 
	\item The fluid has no more proper degree of freedom. Indeed, the fluid equations simplify from an initial and boundary value problem into merely a boundary value problem. In particular, no more initial data is required. 
\end{enumerate}
Newton's law for linear momentum reads: $$\int_{ 
\partial\mathcal A_m(\mathbf s)}\mathbf E_1\cdot {\mathbb T}_m(\mathbf u,p)(t,x)\mathbf n\,{\rm d}\sigma= 0.$$ The solution $(\mathbf u,p)$ being linear with respect to the boundary data $\mathbf w$ it can be decomposed as follows: 
\begin{alignat*}
	{3} \mathbf u(t,x)&=\dot r(t)\mathbf u_r(\mathbf s(t),x)+\sum_{j=1}^N \dot s_j u^j_d(\mathbf s(t),x);\\
	p(t,x)&=\dot r(t)p_r(\mathbf s(t),x)+\sum_{j=1}^N \dot s_j p^j_d(\mathbf s(t),x);&\qquad&(t>0,\,x\in\mathcal F_m(\mathbf s)), 
\end{alignat*}
where we are written $\mathbf s=(s_1,\ldots,s_N)$ in a local chart of $\mathcal S$ and $\dot s=(\dot s_1,\ldots,\dot s_N)$ in the basis $( 
\partial_{s_1},\ldots, 
\partial_{s_N})$ of the tangent space $T_{\mathbf s}\mathcal S$. It entails that the stress tensor: $$ {\mathbb T}_m(\mathbf u,p):=\mu (\nabla_x\mathbf u+\nabla_x\mathbf u^T)-p\,\mathbb I\,{\rm d},$$ can also be decomposed as: $${\mathbb T}_m(\mathbf u,p)=\dot r {\mathbb T}_m(\mathbf u_r,p_r)+\sum_{j=1}^N\dot s_j {\mathbb T}_m(\mathbf u^j_d,p^j_d).$$ The {\it elementary solutions} $(\mathbf u_r,p_r)$ and $(\mathbf u_d,p^j_d)$ satisfy the Stokes system (\ref{stokes_1}-\ref{stokes_2}) with the boundary conditions: 
\begin{alignat*}
	{3} \mathbf u_r(t,x)&=\mathbf E_1&\quad&t>0,\,x\in 
	\partial\mathcal A_m(\mathbf s(t));\\
	\mathbf u^j_d(t,x)&=\mathbf w_d(\mathbf s(t), 
	\partial_{s_j},x)&\quad&t>0,\,x\in 
	\partial\mathcal A_m(\mathbf s(t)),\,j=1,\ldots,N.\\
\end{alignat*}
Notice that the elementary solutions $(\mathbf u_r,p_r)$ and $(\mathbf u_d^j,p_d^j)$ ($j=1,\ldots,N$) depend on time through the shape variable $\mathbf s$ only. We next introduce the scalar: $$M^r(\mathbf s):=\int_{ 
\partial\mathcal A(\mathbf s)}\mathbf E_1\cdot\left(\mathbb T_m(\mathbf u_r,p_r)\right) \mathbf n {\rm d}\sigma,$$ and the row vector $\mathbf N(\mathbf s)$ whose entries are: $$N_j(\mathbf s)=\int_{ 
\partial\mathcal A(\mathbf s)}\mathbf E_1\cdot\left(\mathbb T_m(\mathbf u_d^j,p_d^j)\right) \mathbf n {\rm d}\sigma.$$ We can rewrite Newton's laws as $$\dot r M^r({\mathbf s})+\mathbf N({\mathbf s})\dot{\mathbf s}=0.$$ Upon an integration by parts, we get the equivalent definition $M^r(\mathbf s)$: 
\begin{subequations}
	\label{def_MN_S} 
	\begin{equation}
		M^r(\mathbf s):=\mu\int_{\mathcal F_m(\mathbf s)}D(\mathbf u_r):D(\mathbf u_r)\,{\rm d}x, 
	\end{equation}
	where $D(\mathbf u_r):=(\nabla_x\mathbf u+\nabla_x\mathbf u^T)$. We deduce that $M^r(\mathbf s)$ is always positive. The same arguments for $\mathbf N(\mathbf s)$ lead to the identity: 
	\begin{equation}
		N_j(\mathbf s)=\mu\int_{\mathcal F_m(\mathbf s)}D(\mathbf u^r):D(\mathbf u^j_d){\rm d}x. 
	\end{equation}
	Later on, we will also need the matrix: 
	\begin{equation}
		\mathbb M^d(\mathbf s)=\left(\mu \int_{\mathcal F_m(\mathbf s)}D(\mathbf u^i):D(\mathbf u^j_d){\rm d}x\right)_{1\leqslant i\leqslant N\atop 1\leqslant j\leqslant N}. 
	\end{equation}
\end{subequations}
We eventually obtain the Euler-Lagrange equation governing the rigid displacement with respect to the shape changes: 
\begin{equation}
	\label{dynamics} \dot r=\mathcal L_{\mathbf s}\dot{\mathbf s}\qquad t>0, 
\end{equation}
where we have set set, for every $\mathbf s\in\mathcal S$: $$\mathcal L_{\mathbf s}=-M^r(\mathbf s)^{-1}\mathbf N(\mathbf s).$$ Considering the expressions \eqref{def_MN_S} and \eqref{dynamics}, we deduce: 
\begin{proposition}
	The dynamics of a micro-swimmer is independent of the viscosity of the fluid. Or, in other words, the same shape changes produce the same rigid displacement, whatever the viscosity of the fluid is. 
\end{proposition}
\begin{proof}
	Let $(\mathbf u,p)$ be an elementary solution (as defined in the modeling above) to the Stokes equations corresponding to a viscosity $\mu>0$, then $(\mathbf u,(\tilde\mu/\mu) p)$ is the solution corresponding to the viscosity $\tilde\mu>0$. Since the Euler-Lagrange equation depends only on the Eulerian velocities $\mathbf u$, the proof is completed. 
\end{proof}
In the sequel, we will set $\mu=1$.

The self-propelled constraint \eqref{self_prop} does not make sense any longer for low Reynolds number swimmers because $\varrho_0=0$. However, since we still do not want the swimmer to be able to translate itself just by self-deforming, we require the shape function to satisfy \eqref{self_prop} in which we define the matrix ${\mathcal Q}(\mathbf s)$ by: 
\begin{equation}
	\label{self_prop_2} {\mathcal Q}(\mathbf s):=\int_{\Sigma}\nabla_{\mathbf s}\chi({\mathbf s},x)\,{\rm d}x, 
\end{equation}
where $\Sigma= 
\partial B$.

\subsection{High Reynolds number swimmers} \label{subsec:High_Reynolds}

Assume now that the inertia is preponderant with respect to the viscous force (it is the case when ${\rm Re}\ll 1$). The Navier-Stokes equations \eqref{NS_equations} simplify into the Euler equations: 
\begin{subequations}
	\label{euler_eq} 
	\begin{alignat}
		{3} \label{euler_1} \varrho_f\frac{D}{Dt}\mathbf U(t,X)-\nabla_X\cdot{\mathbb T}_f(\mathbf U,P)(t,X)&=0&\qquad&t>0,\,X\in\mathcal F_f(\mathbf s(t));\\
		\label{euler_2} \nabla_X\cdot\mathbf U(t,X)&=0&&t>0,\,X\in\mathcal F_f(\mathbf s(t));\\
		\label{euler_3} \mathbf U(t,X)\cdot\mathbf n-\mathbf W(t,X)\cdot\mathbf n&=0&&t>0,\,X\in 
		\partial\mathcal A_f(\mathbf s(t)). 
	\end{alignat}
\end{subequations}
where the stress tensor reads: $$\mathbb T_f(\mathbf U,P)(t,X)=-P(t,X){\rm I}d\qquad t>0,\,X\in\mathcal F_f(\mathbf s(t)).$$ Like in the preceding Subsection, we will assume that according to Kelvin's circulation theorem, if the flow is irrotational at some moment (i.e. $\nabla\times\mathbf U=0$) then, it has always been and will always remain irrotational. Hence, we can suppose that $\nabla\times\mathbf U=0$ for all times and then, according to the Helmholtz decomposition, that there exists for all time $t>0$ a potential scalar function $\varPhi(t,\cdot)$ defined in $\mathcal F_f(\mathbf s)$, such that $$\mathbf U(t,X)=\nabla_X\varPhi(t,X)\qquad t>0,\,X\in\mathcal F_f(\mathbf s(t)).$$ The divergence-free condition leads to $$\Delta_X\varPhi(t,X)=0\qquad t>0,\,X\in\mathcal F_f(\mathbf s(t)),$$ and the boundary condition reads: $$ 
\partial_{\mathbf n}\varPhi(t,X)=\mathbf W(t,X)\cdot \mathbf n\qquad t>0,\,X\in 
\partial\mathcal A_f(\mathbf s(t)).$$ The function $\varphi(t,\cdot)$ defined by: $$\varphi(t,x):=\varPhi(t,x-\mathbf r)\qquad t>0,\,x\in\mathcal F_m(\mathbf s(t)),$$ is harmonic in $\mathcal F_m(\mathbf s(t))$ and satisfies the boundary condition: $$ 
\partial_{\mathbf n}\varphi(t,x)=\mathbf w(t,x)\cdot \mathbf n\qquad t>0,\,x\in 
\partial\mathcal A_m(\mathbf s(t)).$$ The potential $\varphi$ is linear in $\mathbf w$, so it can be decomposed into $$\varphi(t,x)=\dot r\varphi_r(t,x)+\sum_{j=1}^N \dot s_j\varphi_d(t,x)\qquad t>0,\,x\in\mathcal F_m(\mathbf s(t)),$$ where at every moment the elementary potentials $\varphi_r(t,\cdot)$ and $\varphi_d(t,\cdot)$ are harmonics in $\mathcal F_m(\mathbf s(t))$ and satisfy the boundary conditions: 
\begin{alignat*}{3} 
	\partial_{\mathbf n}\varphi_r(t,x)&=\mathbf E_1\cdot\mathbf n,\\
	\partial_{\mathbf n}\varphi_d(t,x)&=\mathbf w_d(\mathbf s(t), 
	\partial_{s_j},x)&\qquad& t>0,\,x\in 
	\partial\mathcal A_m(\mathbf s(t)). 
\end{alignat*}
This process is usually referred to as Kirchhoff's law. At this point, we do not invoke Newton's laws to derive the Euler-Lagrange equation but rather use the formalism of Analytic Mechanics. Both approaches (Newton's laws of Classical Mechanics and the Least Action principle of Analytic Mechanics) are equivalent, as proved in \cite{Munnier08}, but the latter is simpler and shorter. 

In the absence of buoyant force, the Lagrangian function $L$ of the body-fluid system coincides with the kinetic energy: $$L=m\frac{1}{2}|\dot r|^2+\frac{1}{2}\int_{\mathcal A_m(\mathbf s)}\varrho(\mathbf s,x)|\mathbf w_d(t,x)|^2\,{\rm d}x+\frac{1}{2}\int_{\mathcal F_m(\mathbf s)}\varrho_f|\mathbf u(t,x)|^2\,{\rm d}x.$$ In this sum, one can identify, from the left to the right: the kinetic energy of the body connecting to the rigid motion, the kinetic energy resulting from the deformations and the kinetic energy of the fluid. We can next compute that, upon a change of variables: $$\int_{\mathcal A_m(\mathbf s)}\varrho(\mathbf s,x)|\mathbf w_d(t,x)|^2\,{\rm d}x= \varrho_0\int_{B}|\nabla_{\mathbf s}\chi({\mathbf s},x)\cdot\dot{\mathbf s}|^2\,{\rm d}x,$$ and $$\int_{\mathcal F_m(\mathbf s)}\varrho_f|\mathbf u(t,x)|^2\,{\rm d}x=\varrho_f\int_{\mathcal F_m(\mathbf s)}|\nabla\varphi(t,x)|^2\,{\rm d}x.$$ It leads us to introduce the scalar: 
\begin{subequations}
	\label{def_MN_E} 
	\begin{equation}
		M^r(\mathbf s)=m+\varrho_f\int_{\mathcal F_m(\mathbf s)}|\nabla\varphi_r(\mathbf s,x)|^2\,{\rm d}x, 
	\end{equation}
	the row vector $\mathbf N(\mathbf s)$ whose entries are: 
	\begin{equation}
		N_j(\mathbf s)=\varrho_f\int_{\mathcal F_m(\mathbf s)}\nabla\varphi_r(\mathbf s,x)\cdot\nabla\varphi_d^j(\mathbf s,x)\,{\rm d}x,\qquad j=1,\ldots,N, 
	\end{equation}
	and the matrix $\mathbb M^d(\mathbf s)$: $$ \mathbb M^d(\mathbf s)=\varrho_0\int_{B}\nabla_{\mathbf s}\chi({\mathbf s},x)\otimes\nabla_{\mathbf s}\chi({\mathbf s},x)\,{\rm d}x +\varrho_f\left(\int_{\mathcal F_m(\mathbf s)}\nabla\varphi_d^i(\mathbf s,x)\cdot\nabla\varphi_d^j(\mathbf s,x)\,{\rm d}x\right)_{1\leqslant i\leqslant N\atop 1\leqslant j\leqslant N}. $$ 
\end{subequations}
Observe the similarity between relations \eqref{def_MN_S} and \eqref{def_MN_E}. We can rewrite the Lagrangian function as: $$L(\dot r,\mathbf s,\dot{\mathbf s})=\frac{1}{2}M^r(\mathbf s)|\dot r|^2+\dot r\mathbf N(\mathbf s)\dot{\mathbf s}+\dot {\mathbf s} \cdot\mathbb M^d(\mathbf s)\dot{\mathbf s}.$$ Invoking now the Least Action principle, we claim that the Euler-Lagrange equation is: $$\frac{d}{dt}\frac{ 
\partial L}{ 
\partial \dot r}(\dot r,\mathbf s,\dot{\mathbf s})-\frac{ 
\partial L}{ 
\partial r}(\dot r,\mathbf s,\dot{\mathbf s})=0,$$ which reduces to, since $L$ does not depend on $r$: $$\frac{d}{dt}\left(M^r(\mathbf s)\dot r+\mathbf N(\mathbf s)\dot{\mathbf s})\right)=0.$$ Assuming that the impulse $M^r(\mathbf s)\dot r+\mathbf N(\mathbf s)\dot{\mathbf s}$ is zero at the initial time, we get eventually for the dynamics the exact same expression as \eqref{dynamics}: $$\dot r = \mathcal L_{\mathbf s}\dot{\mathbf s}\qquad t>0,$$ where, for every $\mathbf s\in\mathcal S$: 
\begin{equation}
	\label{1_forme_L} \mathcal L_{\mathbf s}=-M^r(\mathbf s)^{-1}\mathbf N(\mathbf s). 
\end{equation}
It is easy to verify that the dynamics does not depend on $\varrho_0$ and $\varrho_f$ independently but only on the relative density $\varrho_0/\varrho_f$, which is assumed to be equal to 1 in the sequel.

\subsection{Examples of cost functionals} \label{ex:cost_func} For low Reynolds number swimmers, a classical notion of swimming efficiency (see \cite{Lighthill:1952aa} and \cite{AlougesDeSimone08}) is defined as the inverse of the ratio between the average power expended by the swimmer during a stroke starting and ending at the shape $\mathbf s_{\mathsf i}$ and the power that an external force would spend to translate the system rigidly at the same average speed: $${\rm Eff}^{-1}:=\frac{\frac{1}{T}\int_0^T\left(\int_{ 
\partial\mathcal A_m(\mathbf s)}\mathbf F(\mathbf s,\dot{\mathbf s},\dot r,x)\cdot\mathbf u(\mathbf s,\dot{\mathbf s},\dot r,x)\,{\rm d}\sigma_x\right)\,{\rm d}t}{\bar{\mathbf v}\cdot \int_{ 
\partial\mathcal A_m(\mathbf s)}\mathbf F(\mathbf s_{\mathsf i},\mathbf 0,\bar{\mathbf v},x)\,{\rm d}\sigma_x},$$ where $$\mathbf F(\mathbf s,\dot{\mathbf s},\dot r,x):=\left(\dot r\mathbb T_m(\mathbf u_r,p_r)(\mathbf s,x)+\sum_{j=1}^N\dot s_j\mathbb T_m(\mathbf u^d_j,p_d^j)(\mathbf s,x)\right)\mathbf n$$ is the force in the normal direction exerted by the fluid at the point $x$ of the surface of the swimmer, with shape $\mathbf s$, shape change rate $\dot{\mathbf s}$ and rigid velocity $\dot r$. In the same way: $$\mathbf u(\mathbf s,\dot{\mathbf s},x):= \dot r\mathbf E_1+\sum_{j=1}^N\dot s_j\mathbf w_d(\mathbf s, 
\partial_{s_j},x),$$ is the velocity of the swimmer. Eventually $\bar{\mathbf v}$ is the average speed: $$\bar{\mathbf v}:=\left(\frac{1}{T}\int_0^T\dot{r}\,{\rm d}t\right)\mathbf E_1.$$ With the notation \eqref{def_MN_S}, the efficiency can be rewritten as: 
\begin{equation}
	\label{efficiency} {\rm Eff}^{-1}:=\frac{\frac{1}{T}\int_0^T \left(M^r(\mathbf s)|\dot r|^2+\dot r\mathbf N(\mathbf s)\dot{\mathbf s} +\dot{\mathbf s}\cdot\mathbb M^d(\mathbf s)\dot{\mathbf s}\right)\,{\rm d}t}{|\bar{\mathbf v}|^2M^r(\mathbf s_{\mathsf i})}. 
\end{equation}
It can easily be verified that: $$M^r(\mathbf s)|\dot r|^2+\dot r\mathbf N(\mathbf s)\dot{\mathbf s} +\dot{\mathbf s}\cdot\mathbb M^d(\mathbf s)\dot{\mathbf s}=\int_{\mathcal F_m(\mathbf s)} D(\mathbf u,p):D(\mathbf u,p)\,{\rm d}x>0,$$ where $(\mathbf u,p)$ is the solution to the Stokes system \eqref{stokes_eq}. 

For high Reynolds number swimmers, we can choose the same expression \eqref{efficiency} for the efficiency, in which we use the definitions \eqref{def_MN_E}. In this case, the efficiency is the inverse of the ratio between the mean energy expended by the swimmer divided by the energy required to translate rigidly the swimmer at the same average speed.

Taking into account the dynamics and replacing $\dot r$ by $-M^r(\mathbf s)^{-1}\mathbf N(\mathbf s)\dot{\mathbf s}$ in \eqref{efficiency}, it leads us to consider on $T_{\mathbf s}\mathcal S$ the following scalar product: 
\begin{equation}
	\label{scalar_prod} \mathbf g_{\mathbf s}(\dot{\mathbf s}_1,\dot{\mathbf s}_2)=\dot{\mathbf s}_1\cdot \left(\mathbb M^d(\mathbf s)-\frac{\mathbf N(\mathbf s)\otimes\mathbf N(s)}{M^r(\mathbf s)}\right)\dot{\mathbf s}_2, \quad(\dot{\mathbf s}_1,\dot{\mathbf s}_2\in T_{\mathbf s}\mathcal S). 
\end{equation}
According to the abstract framework introduced in Section~\ref{SEC_framework}, the cost of an admissible shape change $\mathbf s:[0,T]\mapsto\mathcal S$ will be: 
\begin{equation}
	\label{exemp_cost_fonct} \frac{1}{2}\int_0^T \mathbf g_{\mathbf s(t)}(\dot{\mathbf s}(t),\dot{\mathbf s}(t))\,{\rm d}t. 
\end{equation}

\subsection{Regularity results} In Section~\ref{SEC_framework}, the manifold $\mathcal S$ and the differential forms are all of them assumed to be analytic. The following Lemma ensures that, under a simple hypothesis, this regularity is ensured for swimmers in a perfect fluid and Stokesian swimmers.

We denote by $M(N_1,N_2)$ the Euclidian space of the $N_1\times N_2$ matrices and we claim: 
\begin{lemma}
	Assume that the map $\mathbf s\in\mathcal S\mapsto \chi(\mathbf s,\cdot)\in C^1(\bar B,\mathbb R^3)$ is analytic (we refer to \cite{Whittlesey:1965aa} for the definitions and the properties of analytic functions valued in Banach spaces), then for both cases (low and high Reynolds number swimmers) the maps $\mathbf s\in\mathcal S\mapsto \mathcal Q(\mathbf s)\in M(3,N)$, $\mathbf s\in\mathcal S\mapsto M^r(\mathbf s)\in\mathbb R$, $\mathbf s\in\mathcal S\mapsto \mathbf N(\mathbf s)\in\mathbb R^N$ and $\mathbf s\in\mathcal S\mapsto \mathbb M^d(\mathbf s)\in M(N,N)$ are analytic. 
\end{lemma}
The proofs can be found in \cite{ChambrionMunnier11} for swimmers in a perfect fluid and in \cite{LoheacMunnier12} for Stokesian swimmers. 

\section{Examples} \label{sec:examples}

\subsection{The $N$-spheres swimmer} \label{subsec:n_spheres} The $N$-sphere swimmer in a Stokes flow is an important example which was used for instance with $N=3$ in \cite{AlougesDeSimone08} to provide the first available positive controllability result for a swimmer at low Reynolds number. It consists in a set of $N$ spheres immersed in a 3-D infinite extent of fluid. One can control the lengths of the rigid bonds linking the spheres. Moreover, these bonds are assumed to be thin and light enough to have no effect neither on the dynamics of the swimmer nor on the flow. 

For the linear $3$-spheres swimmer appearing in \cite{AlougesDeSimone08}, see Figure \ref{FIG_trisphere}, the two degrees of freedom are the length $s_1$ of the bond between spheres 1 and 2 and the length $s_2$ of the bond between spheres 2 and 3. The shape manifold $\mathcal{S}$ is $(0,+\infty)^2$. This swimmer is hence trivialized. Since it is shown to be controllable in \cite{AlougesDeSimone08}, Theorems~\ref{main_theorem_properties}, \ref{main_theorem_properties_1} and \ref{main_theorem_properties_2} apply straightforwardly and provide new results about optimal control for this swimmer. 
\begin{figure}
	[h] \centerline{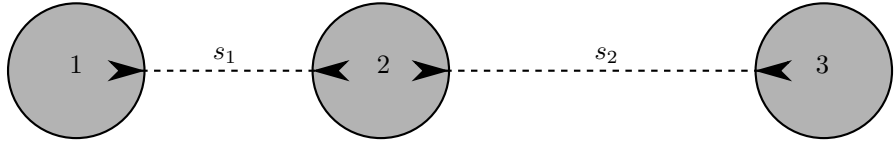}

	\caption{The linear 3-spheres swimmer is made of three spheres constrained to stay on a fixed line. The shape of the swimmer is determined by the lengths $s_1$ and $s_2$ of the bonds between spheres 1-2, and 2-3.} \label{FIG_trisphere} 
\end{figure}

\subsection{The Legendre's swimmer} \label{subsec:legendre} We call Legendre's swimmer, the Stokesian swimmer studied in \cite{LoheacScheid11}. In this article and with the notation of Section~\ref{SEC:modelling}, the authors consider axi-symmetric deformations of the unit ball having the form: $$\chi(\mathbf s,x) = x + \sum_{j=1}^p s_j P_{j+1}(\cos(\theta(x))x,\qquad x\in\mathbb R^3,$$ where $P_j$ is the $j$-th Legendre polynomial and $\theta(X)$ the polar angle (in spherical coordiates) of the point $X$. The shape variable $\mathbf s$ belongs to the shape space $\mathcal S=\mathbb R^p$ (the swimmer is hence trivialized). The shape at rest is $\mathbf s_{\mathsf i}=0$ and the swimmer is unconstrained ($\mathcal Q=0$). On the other hand, the swimmer is shown to be controllable for every $p\geqslant 2$ in the article \cite{LoheacScheid11} and therefore Theorems~\ref{main_theorem_properties}, \ref{main_theorem_properties_1} and \ref{main_theorem_properties_2} apply straightforwardly for the examples of cost functional of Subsection~\ref{ex:cost_func}. In particular, the authors focused on Problem~\ref{PROB_time} where $\mathcal K$ is the closed ball of $\mathbb R^p$ centered at the origin and of radius $c>0$. The existence of a minimizer for this problem (which is not proved in \cite{LoheacScheid11}) is provided in Theorems~\ref{main_theorem_properties}.

\subsection{A swimmer in a potential flow} \label{subsec:potential} We are now going to treat an in-depth example, starting from scratch. We have chosen to deal with a simplified version of a 2D swimmer in a perfect fluid introduced in \cite{ChambrionMunnier11a} and improved in \cite{Munnier:2011aa}. 
\subsubsection*{Shape changes} Recall that in 2D, following the notation of Section~\ref{SEC:modelling}, at every time, $\mathcal A_m(\mathbf s)$ is the image of the unit disk $D$ by a diffeomorphism $\chi(\mathbf s,\cdot)$ depending on the parameter $\mathbf s$, and whose form (with complex notation) is: 
\begin{equation}
	\chi(\mathbf s,z)=z+s_1\bar z+s_2\bar z^2+s_3\bar z^3,\qquad (z\in \mathbb C,\,\mathbf s=(s_1,s_2,s_3)\in\mathbb R^3). 
\end{equation}
We define the following norm in $\mathbb R^3$: $$\|\mathbf s\|_{\mathcal S}=\sup_{z\in 
\partial D}|s_1+2s_2 z+3s_3 z^2|,\qquad(\mathbf s\in\mathbb R^3),$$ and we claim (see \cite{ChambrionMunnier11a} for details): 
\begin{lemma}
	\begin{enumerate}
		\item The mapping $\chi(\mathbf s,\cdot)$ is a $C^\infty$ diffeormorphism from the the unit ball $D$ onto its image $\mathcal A_m(\mathbf s)$ if and only if $\|\mathbf s\|_{\mathcal S}<1$. 
		\item The measure of the area of $\mathcal A_m(\mathbf s)$ is $\pi(1-s_1^2-2s_2^2-3s_3^2)$. 
	\end{enumerate}
\end{lemma}
Since we want both conditions (i) the mapping $\chi(\mathbf s,\cdot)$ is a diffeomorphism and (ii) the area of $\mathcal A_m(\mathbf s,\cdot)$ is of constant (and nonzero) measure, to be fulfilled, we introduce for every $0<\mu<1$ the set (see Fig.~\ref{fig1} ): $$\mathcal S_\mu=\{\mathbf s\in\mathbb R^3\,:\,\|\mathbf s\|_{\mathcal S}<1\text{ and }s_1^2+2s_2^2+3s_3^2=\mu^2\}.$$ For any $0<\mu<1$, $\mathcal S_\mu$ is a 2D analytic submanifold of $\mathbb R^3$. It consists in the parts the ellipsoid surface $s_1^2+2s_2^2+3s_3^2=\mu^2$ lying inside the unit ball $\|\mathbf s\|_{\mathcal S}<1$ (See Fig.~\ref{fig1}). 
\begin{figure}
	[h] \centering 
	\begin{tabular}
		{|c|c|c|} \hline \subfigure [$\mu=0.52$] { 
		\includegraphics[height=.2 
		\textwidth]{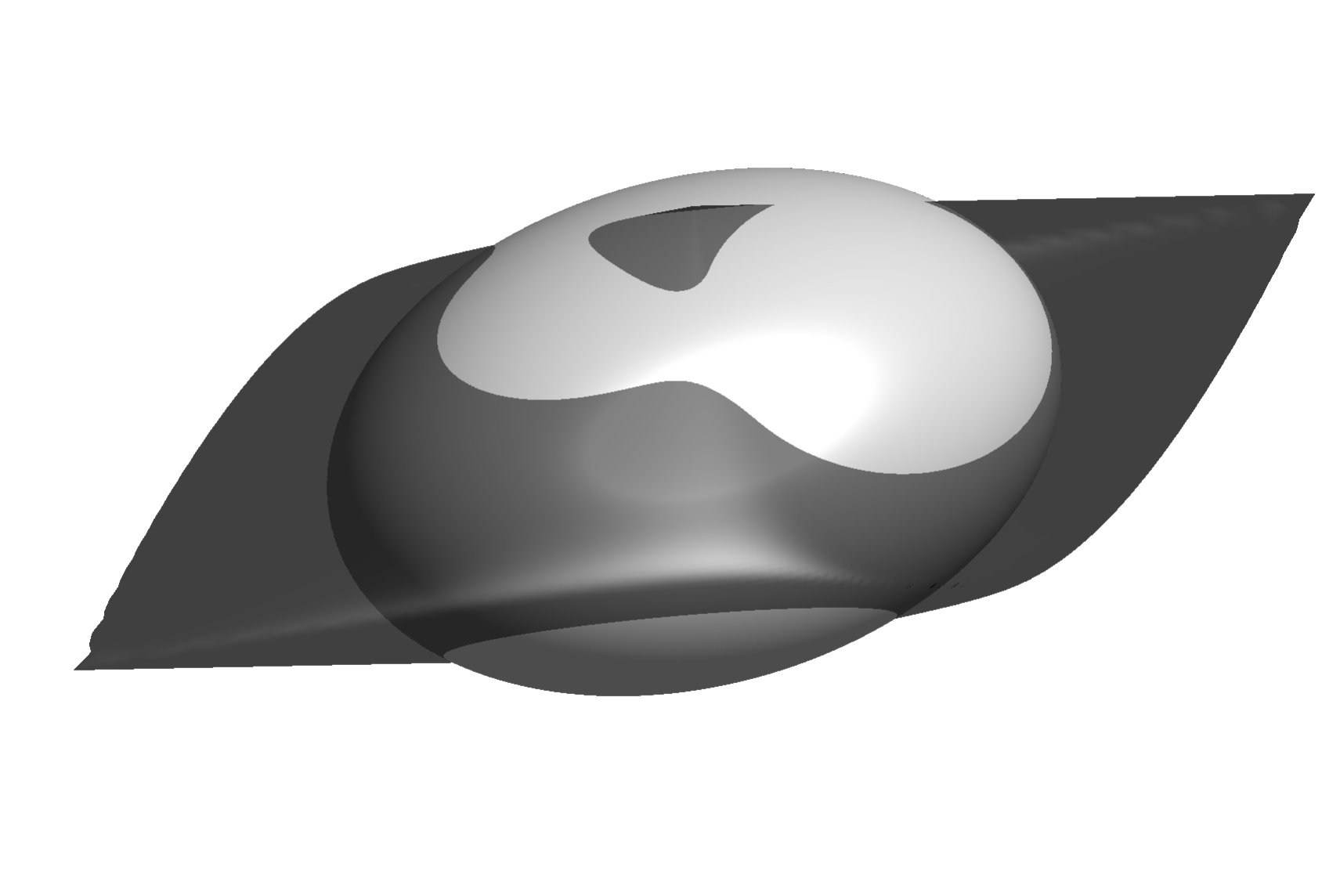}} & \subfigure [$\mu=0.45$] { 
		\includegraphics[height=.2 
		\textwidth]{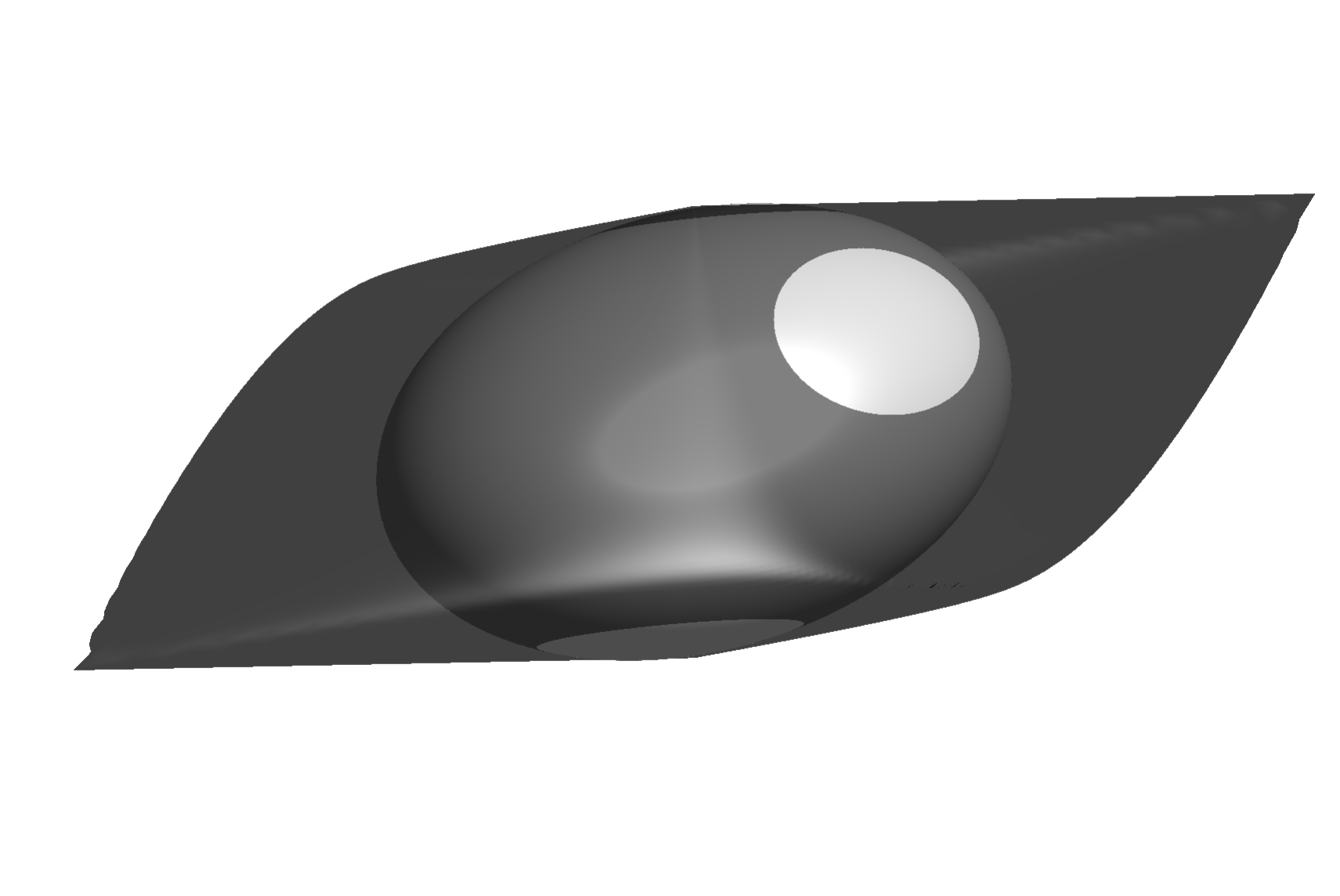}} & \subfigure [$\mu=0.3$] { 
		\includegraphics[height=.2 
		\textwidth]{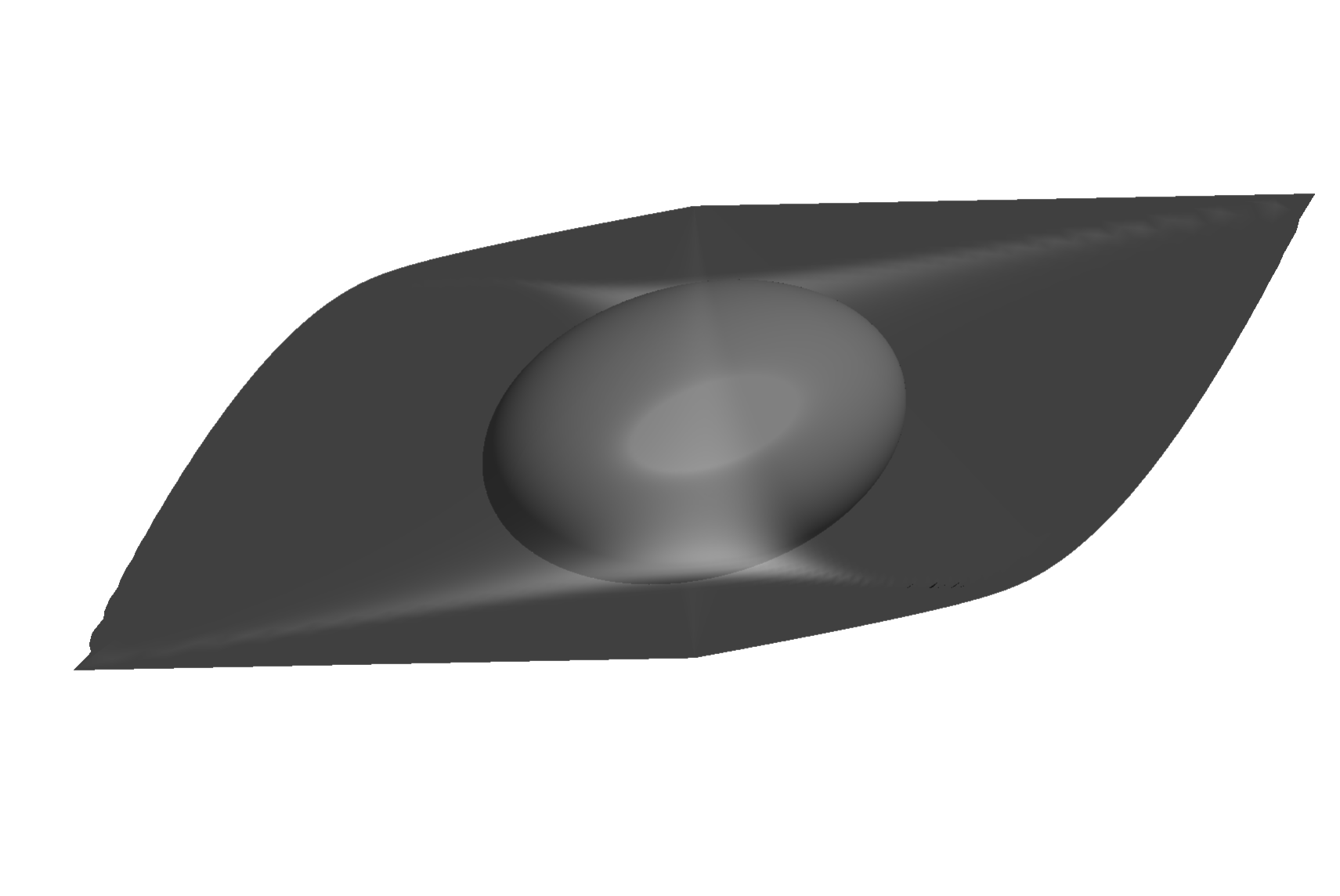}}\\
		\hline 
	\end{tabular}
	\caption{\label{fig1}Examples of manifolds $\mathcal S_\mu$ for different values of $\mu$. The unit ball $\|\mathbf s\|_{\mathcal S}<1$ is in dark grey while de ellipsoid $s_1^2+2s_2^2+3s_3^2=\mu^2$ is in light grey. The manifold $\mathcal S_\mu$ is the part of the ellipsoid lying inside the unit ball. For small values of $\mu$, $\mathcal S_\mu$ turns out to be merely the surface of the ellipsoid since it is entirely included in the unit ball.} 
\end{figure}
\begin{figure}
	[h] \centerline{ 
	\includegraphics[height=.4 
	\textwidth]{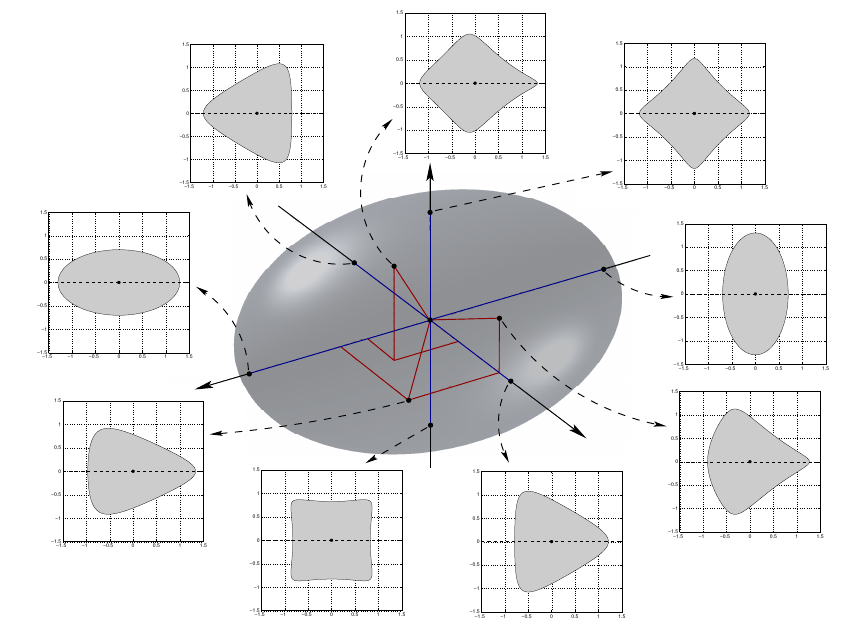}}

	\caption{\label{fig2}Some points on the ellipsoid $\mathcal S_\mu$ ($\mu=0.3$) and the corresponding shapes of the swimmer.} 
\end{figure}

To simplify, we will consider in the following that $\mu$ is small enough. In this case, the ellipsoid $s_1^2+2s_2^2+3s_3^2= \mu^2$ is included in the unit ball $\|\mathbf s\|_{\mathcal S}<1$, and hence $\mathcal S_\mu$ reduces merely to the surface of the ellipsoid.

As a conclusion, once $\mu$ (and therefore the measure of the swimmer) has been chosen and fixed, the shape changes over a time interval $[0,T]$ are described by means of a function: $$t\in[0,T]\mapsto\mathbf s(t)\in\mathcal S_\mu.$$ By direct computations, one verifies that the self-propelled constraints \eqref{self_prop}, ensuring that the swimmer can not modified its linear momentum by self-deforming, are automatically satisfied in this simplified case (case of unconstrained swimmer, i.e. $\mathcal Q^{\mathcal S_\mu}=0$). 

In Fig.~\ref{fig2}, we have pictured some points of the the ellipsoid and the corresponding shapes for the swimmer.

Using the conformal mapping $$\phi(\mathbf s,z):=z+\frac{s_1}{z}+\frac{s_2}{z^2}+\frac{s_3}{z^3},\qquad (z\in \mathbb C\setminus\bar D,\,\mathbf s=(s_1,s_2,s_3)\in\mathcal S_\mu),$$ which maps the exterior of the unit disk onto the fluid domain $\mathcal F_m(\mathbf s)$, we can compute explicitly the elementary kirchhoff's potentials $\varphi_r(\mathbf s,\cdot)$ and $\varphi_d(\mathbf s,\cdot)$ (again, we refer to \cite{ChambrionMunnier11a} or \cite{Munnier:2011aa} for the details). We finally get the following expressions for the mass matrices introduced in \eqref{def_MN_E}: 
\begin{subequations}
	\label{mass_matrices_ex} 
	\begin{align}
		M^r(\mathbf s)&= 2-2s_1\\
		\mathbf N(\mathbf s)&= 
		\begin{bmatrix}
			-3s_{{2}}+2s_{{2}}s_{{1}}+3s_{{2}}s_ {{3}}&-s_{{1}}-4s_{{3}}+{s_{{1}}}^{2}+3s_{{1}}s_{{3}}&-2s_{{2}}+ 3s_{{2}}s_{{1}} 
		\end{bmatrix}
		\\
		\mathbb M^d(\mathbf s)&= 
		\begin{bmatrix}
			4{s_{{2}}}^{2}-3s_{{3}}+\frac{9}{2}{s_{{3}}} ^{2}+1&2s_{{1}}s_{{2}}+6s_{{2}}s_{{3}}&4{s_{{2}}}^{2}-\frac{1}{2}s_{{1 }}+\frac{3}{2}s_{{1}}s_{{3}}\\
			2s_{{1}}s_{{2}}+6s_{{2}} s_{{3}}&{s_{{1}}}^{2}+6s_{{1}}s_{{3}}+9{s_{{3}}}^{2}+\frac{2}{3}&2s_{{1} }s_{{2}}+6s_{{2}}s_{{3}}\\
			4{s_{{2}}}^{2}-\frac{1}{2}s_ {{1}}+\frac{3}{2}s_{{1}}s_{{3}}&2s_{{1}}s_{{2}}+6s_{{2}}s_{{3}}&4{s_{{ 2}}}^{2}+\frac{1}{2}{s_{{1}}}^{2}+\frac{1}{2} 
		\end{bmatrix}
	\end{align}
\end{subequations}
The 1-form $\mathcal L$ defined in \eqref{1_forme_L} as well as the scalar product $\mathbf g$ defined in \eqref{scalar_prod} and the cost functional \eqref{exemp_cost_fonct} can now be explicitly computed. Instead of writing out the complicated expressions of $\mathcal L$ and $\mathbf g$, we compute rather the Ricci-curvature induced by $\mathbf g$ on the ellipsoid (the result is displayed on Fig~\ref{fig3}) and we have drawn on Fig~\ref{density_fct} the density function of the measure ${\rm d}\mathcal L$. 
\begin{figure}
	[h] \centerline{ 
	\includegraphics[height=.3 
	\textwidth]{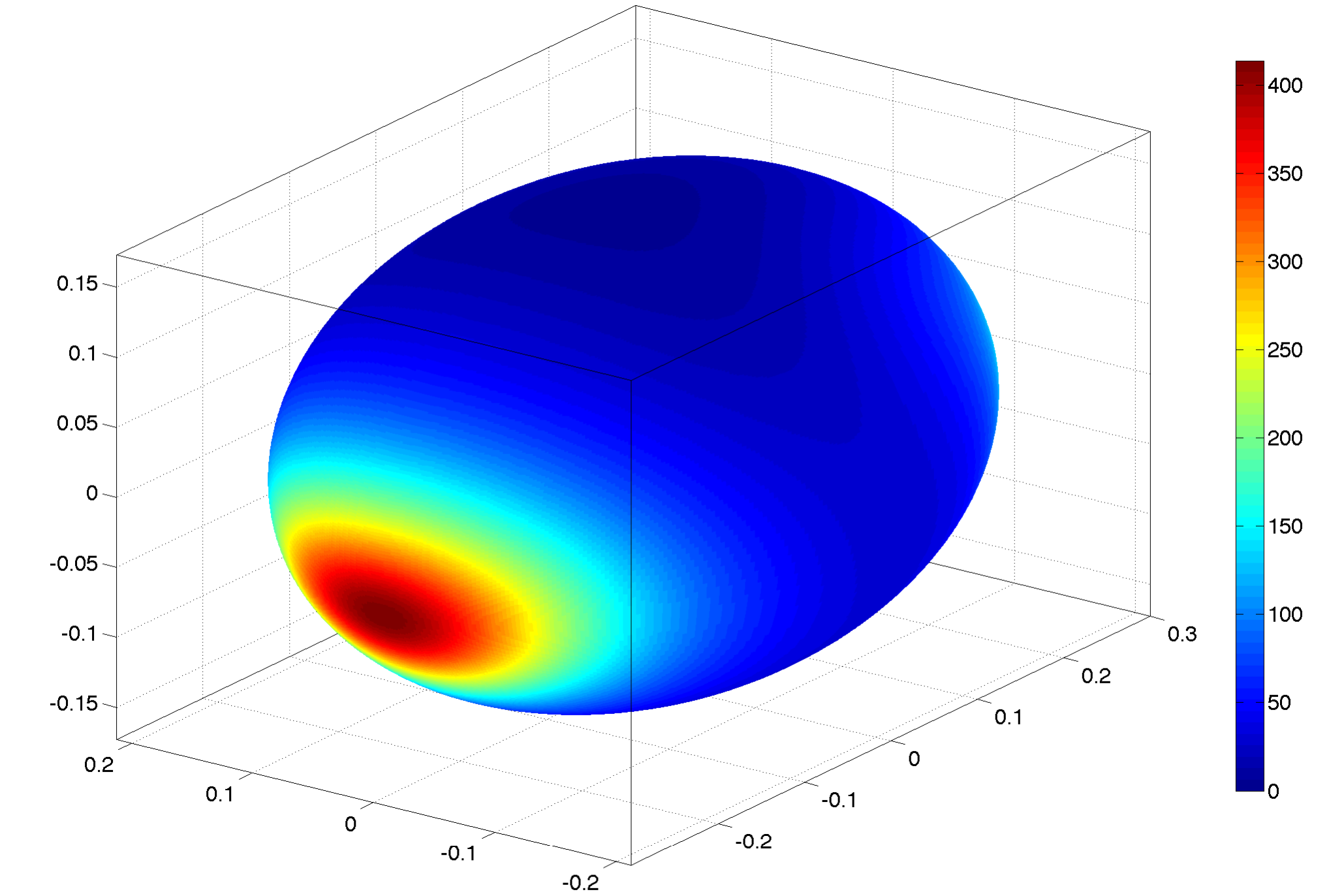}}

	\caption{\label{fig3}The Ricci curvature corresponding to the Riemannian metric $\mathbf g$ defined in \eqref{scalar_prod}, on the ellipsoid.} 
\end{figure}
\begin{figure}
	[h] \centerline{ 
	\includegraphics[height=.3 
	\textwidth]{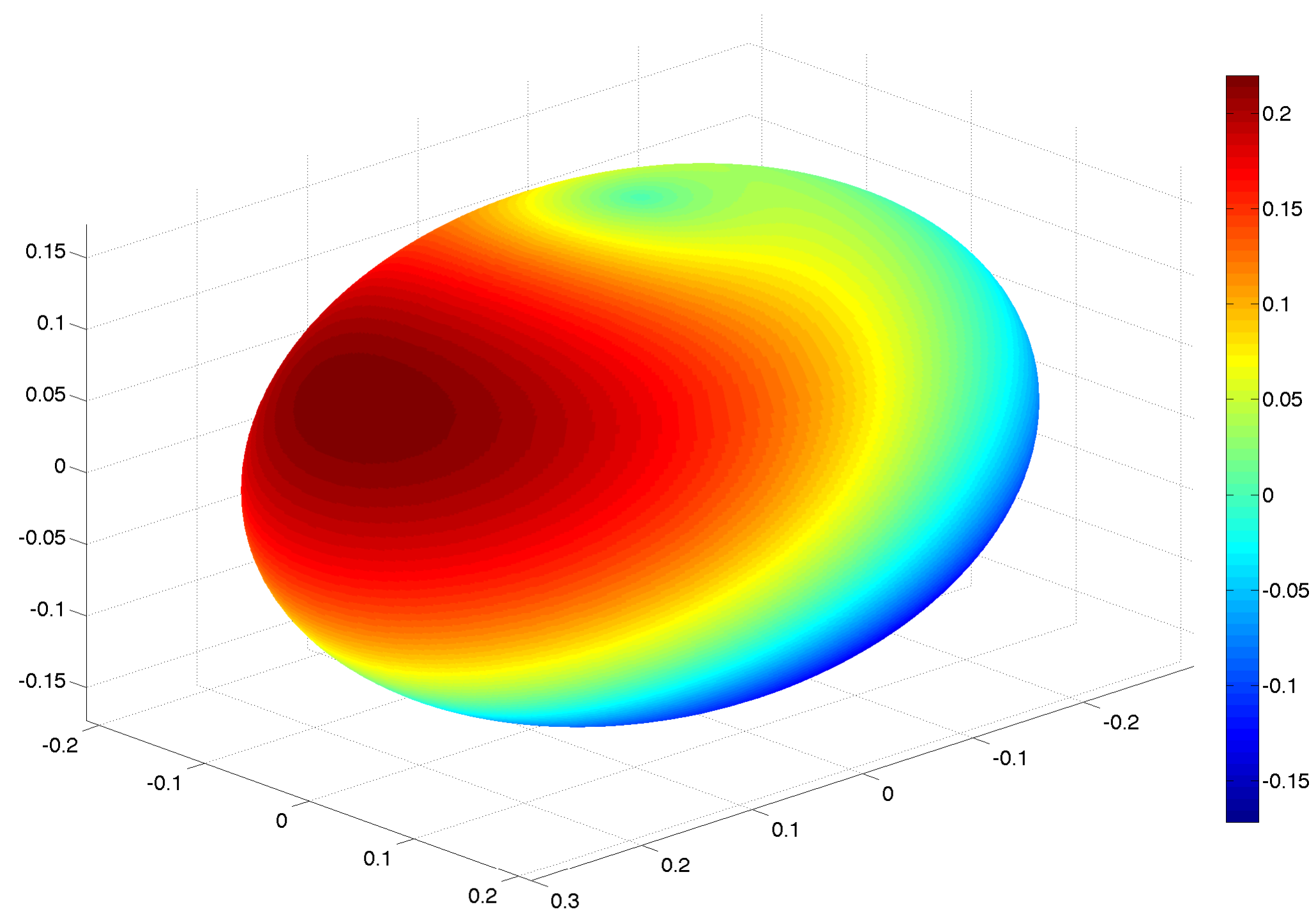}\qquad 
	\includegraphics[height=.3 
	\textwidth]{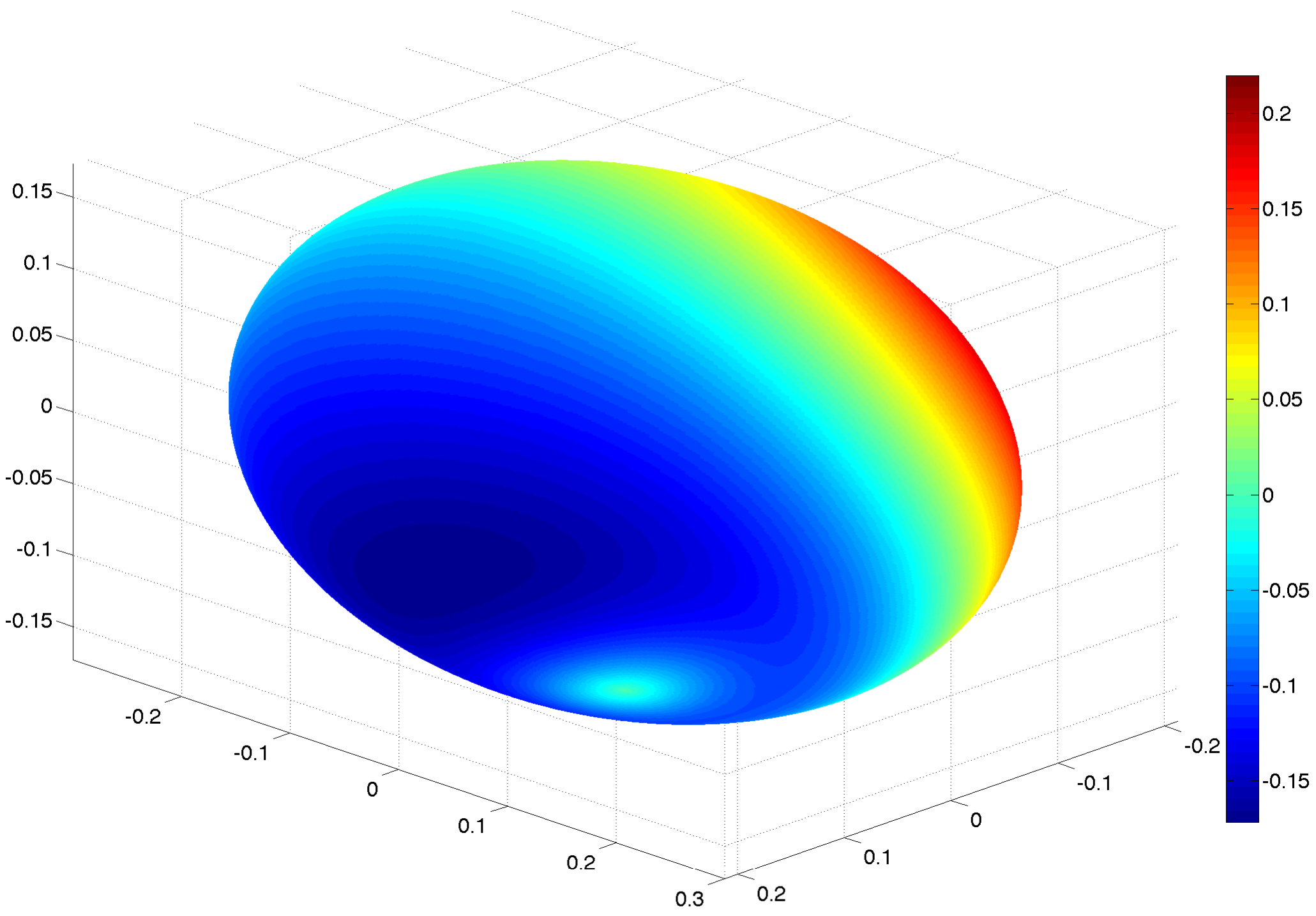}} \caption{\label{density_fct}Density fonction of the signed measure defined by $\,{\rm d}\mathcal L$ on the ellipsoid. A stroke being a closed curve on the ellipsoid, the resulting travelled distance is obtained by measuring the area of the enclosed surface for the measure $\,{\rm d}\mathcal L$. See also Fig~\ref{optim_1} and \ref{optim_2}.} 
\end{figure}

It remains to choose any point $\mathbf s_{\mathsf i}$ on the ellipsoid $\mathcal S_\mu$ (the {\rm default} shape of the swimmer) and every element of the quintuple ${\mathscr S}=(\mathcal{S}_\mu, \mathbf{g},\mathcal{Q}^{\mathcal{S}_\mu}, \mathbf{s}_{\mathsf i},\mathcal{L})$ is then defined. 

%
\subsubsection*{Controllability result}

In this case, as already mentioned before, we have $\mathcal Q^{{\mathcal S}_\mu} =\mathbf 0$ (the self-propelled constraints are always fulfilled) and hence $\Delta^{{\mathcal S}_\mu}=T\mathcal S_\mu$. We define the following vector fields which are, for every $\mathbf s\in\mathcal S_\mu$, an analytic spanning set of $T_{\mathbf s}\mathcal S_\mu$: $$ \mathbf X_1(\mathbf s):= 
\begin{bmatrix}
	3s_3(1-s_1)\\
	0\\
	s_1(s_1-1) 
\end{bmatrix}
,\quad \mathbf X_2(\mathbf s):= 
\begin{bmatrix}
	2s_2(1-s_1)\\
	s_1(s_1-1)\\
	0 
\end{bmatrix}
,\quad \mathbf X_3(\mathbf s):= 
\begin{bmatrix}
	0\\
	3s_3(1-s_1)\\
	2s_2(s_1-1) 
\end{bmatrix}
.$$ Notice that, for all $\mathbf s\in\mathcal S$, the vectors satisfy the linear identity $2s_2\mathbf X_1(\mathbf s)-3s_3\mathbf X_2(\mathbf s)-s_1\mathbf X_3(\mathbf s)=\mathbf 0$. From this family of vectors, we build the vectors $\mathbf Z_j$ ($j=1,2,3$) according to the definition \eqref{def_vectors_Z} and the expressions \eqref{mass_matrices_ex}. We get: 
\begin{align*}
	\mathbf Z_1(\boldsymbol\xi)&:= 
	\begin{bmatrix}
		3s_3(1-s_1)\\
		0\\
		s_1(s_1-1)\\
		\frac{9}{2}s_2s_3-3s_1s_2s_3-\frac{9}{2}s_2s_3^2-s_1s_2+\frac{3}{2}s_1^2s_2 
	\end{bmatrix}
	\\
	\mathbf Z_2(\boldsymbol\xi)&:= 
	\begin{bmatrix}
		2s_2(1-s_1)\\
		s_1(s_1-1)\\
		0\\
		3s_2^2-2s_1s_2^2-3s_2^2s_3-\frac{1}{2}s_1^2-2s_1s_3+\frac{1}{2}s_1^3+\frac{3}{2}s_1^2s_3 
	\end{bmatrix}
	\\
	\mathbf Z_3(\boldsymbol\xi)&:= 
	\begin{bmatrix}
		0\\
		3s_3(1-s_1)\\
		2s_2(s_1-1)\\
		\frac{3}{2}s_1s_3+6s_3^3-\frac{3}{2}s_1^2s_3-\frac{9}{2}s_1s_3^2-2s_2^2+3s_1s_2^2 
	\end{bmatrix}
	. 
\end{align*}
Obviously, we have again that $2s_2\mathbf Z_1(\boldsymbol\xi)-3s_3\mathbf Z_2(\boldsymbol\xi)-s_1\mathbf Z_3(\boldsymbol\xi)=\mathbf 0$ for all $\boldsymbol\xi=(\mathbf s,r)\in\mathcal M$. By direct calculation, one can check that for all $\boldsymbol\xi\in\mathcal M$: $$ [\mathbf Z_1(\boldsymbol\xi),\mathbf Z_2(\boldsymbol\xi)]= 
\begin{bmatrix}
	0\\
	3s_3(2s_1-1)(s_1-1)\\
	-2s_2(2s_1-1)(s_1-1)\\
	-\frac{3}{2}s_1s_3-3s_1^2s_2^2+\frac{3}{2}s_1^3s_3+\frac{9}{2}s_1^2s_3^2-\frac{21}{2}s_1s_3^2-s_1^2+6s_3^2+s_1^3+5s_1s_2^2-2s_2^2 
\end{bmatrix}
. $$ For $\boldsymbol\xi_{\mathsf i}:=(\mu,0,0,0)\in\mathcal M$, we also have: $$[\mathbf Z_1(\boldsymbol\xi_{\mathsf i}),\mathbf Z_2(\boldsymbol\xi_{\mathsf i}),[\mathbf Z_1(\boldsymbol\xi_{\mathsf i}),\mathbf Z_2(\boldsymbol\xi_{\mathsf i})]]= \mu(\mu-1) 
\begin{bmatrix}
	0&0&0\\
	0&1&0\\
	1&0&0\\
	0&\frac{1}{2}\mu^2&\mu^2 
\end{bmatrix}
,$$ and hence ${\rm dim}\,{\rm Lie}_{\boldsymbol\xi_{\mathsf i}}\{\mathbf Z_j,\,j=1,2,3\}=3$. According to Theorem~\ref{theo_control}, we deduce that this example of swimmer in a perfect fluid is controllable.

\subsubsection*{Optimal strokes} We can now address the optimal problems \ref{PROB_Moins_cher}-\ref{PROB_Plus_loin2}. Notice that the swimmer is not trivialized (according to the Hairy ball theorem, there exists no analytic basis for the ellipsoid $\mathcal S_\mu$). However, its suffices to replace $\mathcal S_\mu$ by $\tilde{\mathcal S}_\mu:=\mathcal S_\mu\setminus\{\mathbf s^\ast\}$ (i.e. to remove any point $\mathbf s^\ast$ of the ellipsoid) to get indeed a trivialized swimmer. Then, Theorems~\ref{main_theorem_properties}, \ref{main_theorem_properties_1} and \ref{main_theorem_properties_2} apply straightforwardly.

\subsection{Numerical tests} \label{swim:num} Let us continue with the swimmer in a potential flow. The existence of minimizers for Problems \ref{PROB_Moins_cher}-\ref{PROB_Plus_loin2} being now taken for granted, we are interested in finding numerical approximations of these minimizers. 

The method we have chosen consists in approximating any curve on the ellipsoid $\mathcal S_\mu$ by means of a linear combination of base functions written in spherical coordinates. More precisely, a curve on $\mathcal S_\mu$ is parameterized by: $$t\in[0,1]\mapsto \left(\frac{1}{\mu}\sin(\phi(t))\cos(\theta(t)),\frac{\sqrt{2}}{\mu}\sin(\phi(t))\sin(\theta(t)),\frac{\sqrt{3}}{\mu}\cos(\phi(t))\right),$$ where $$\phi(t)=\sum_{j=1}^p \beta_j \phi_j(t)\quad\text{and}\quad \theta(t)=\sum_{j=1}^p \alpha_i \theta_j(t),$$ with $\{\phi_1,\ldots,\phi_p\}$ a (free) family of functions from $[0,1]$ onto $[0,\pi]$ and $\{\theta_1,\ldots,\theta_p\}$ a (free) family of functions from $[0,1]$ onto $[-\pi,\pi]$.

In the following examples, we choose $p=10$ and the base functions $\theta_j$ and $\varphi_j$ are cubic splines. So, we dispose of 20 parameters ($\alpha_1,\ldots,\alpha_{10},\beta_1,\ldots,\beta_{10}$) and the optimal problems under consideration turn into finite dimensional optimal problems. To every set of parameters, we can associate a travelled distance and a cost. To solve the optimal problems, we use the optimal toolbox of Matlab. The main difficulty is to manage the change of chart. Indeed, starting with the classical spherical coordinates, a curve cannot pass through the south or north pole of the ellipsoid. So we have to switch the axes, in such a way that the north pole becomes a regular point in spherical coordinates.

\subsubsection*{The distance-cost function} We consider Problem~\ref{PROB_Moins_cher}. We set the shape at rest $\mathbf s_{\mathsf i}$ and we compute for every $\delta\in\mathbb R_+$, the optimal stroke (i.e. minimizing the cost) allowing the swimmer to cover the distance $\delta$. We choose as shape $\mathbf s_{\mathsf i}$, the converging point of all the curves on Fig~\ref{optim_2}. On this picture, we drawn on the left all the curves corresponding to the optimal strokes for different values of $\delta$ while on the right is pictured a sequence of shapes corresponding to the longest curve. Finally, the graph of the distance-cost function $\Phi_{{\mathscr S},{\mathcal K}}$ defined in \eqref{function_phi} is given in Fig~\ref{fig222}. 
\begin{figure}
	[h] \centerline{ 
	\includegraphics[height=.4 
	\textwidth]{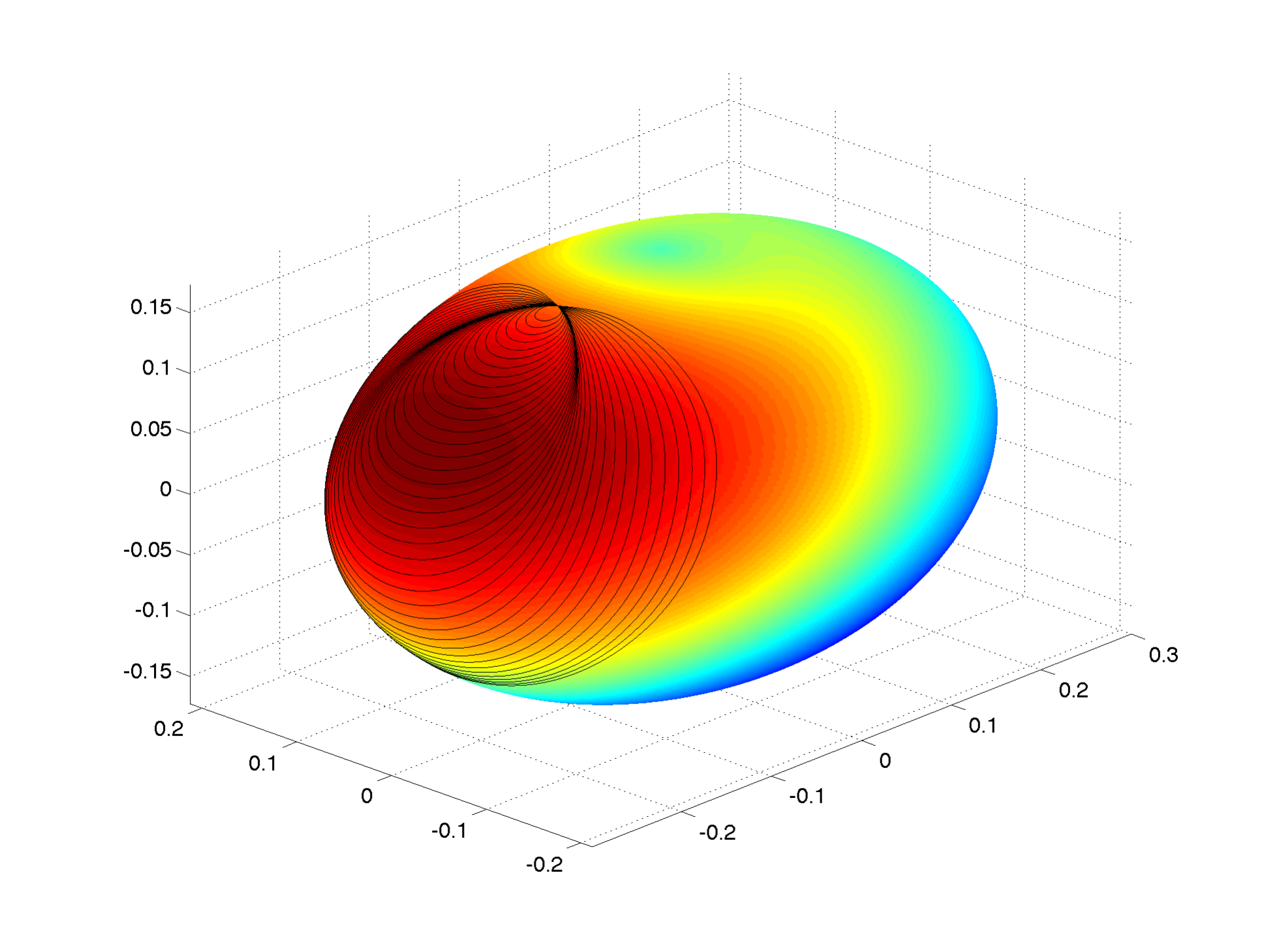}\quad 
	\includegraphics[height=.4 
	\textwidth]{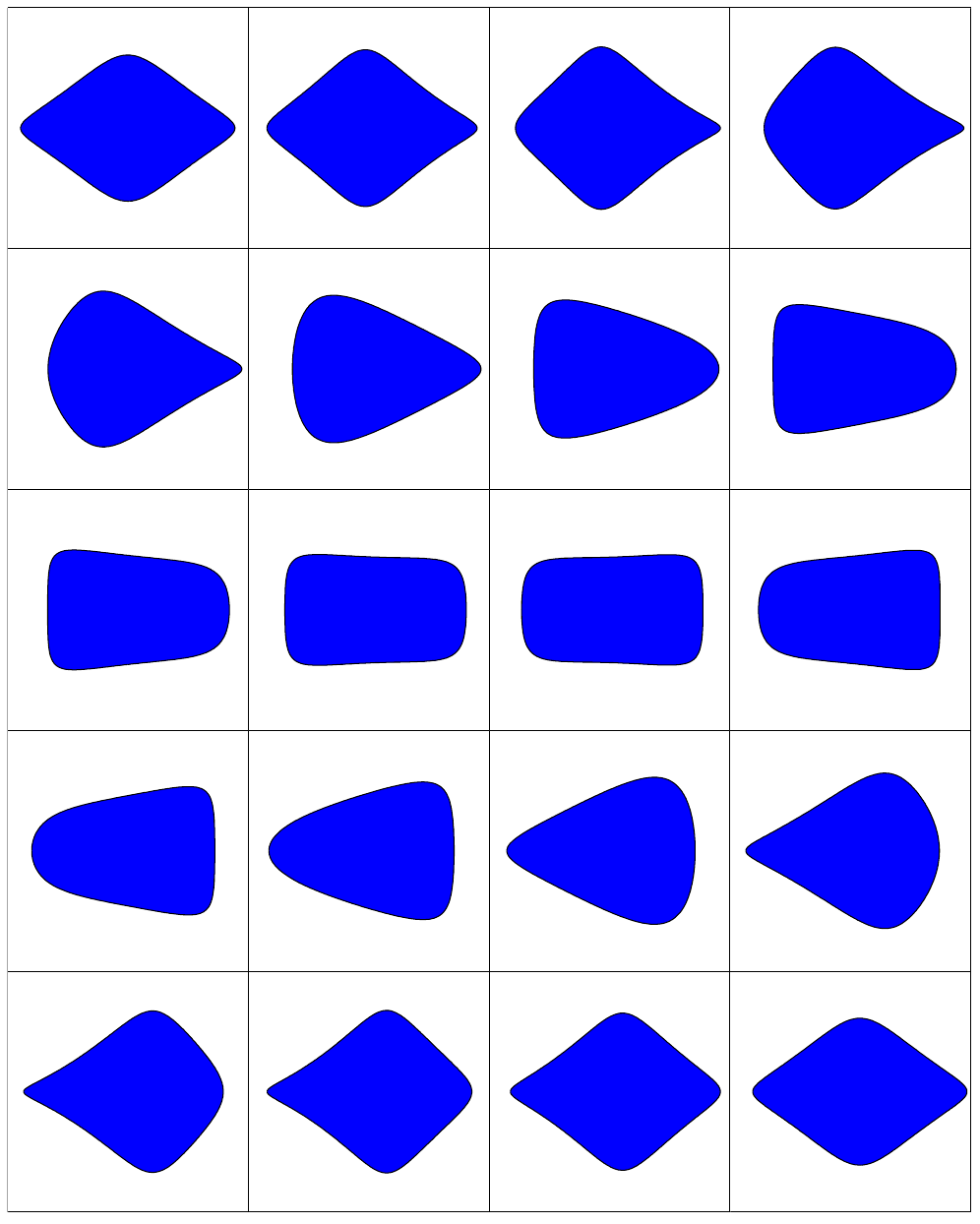}}

	\caption{\label{optim_2}On the left: Every closed curve on the ellipsoid corresponds to an optimal stroke (minimizing the cost for a given distance but also maximizing the distance for a given cost). On the right: A sequence of 20 time-equidistributed shapes corresponding to the longest curve on the ellipsoid.} 
\end{figure}

\begin{figure}
	[h] \centerline{ 
	\includegraphics[width=.5 
	\textwidth]{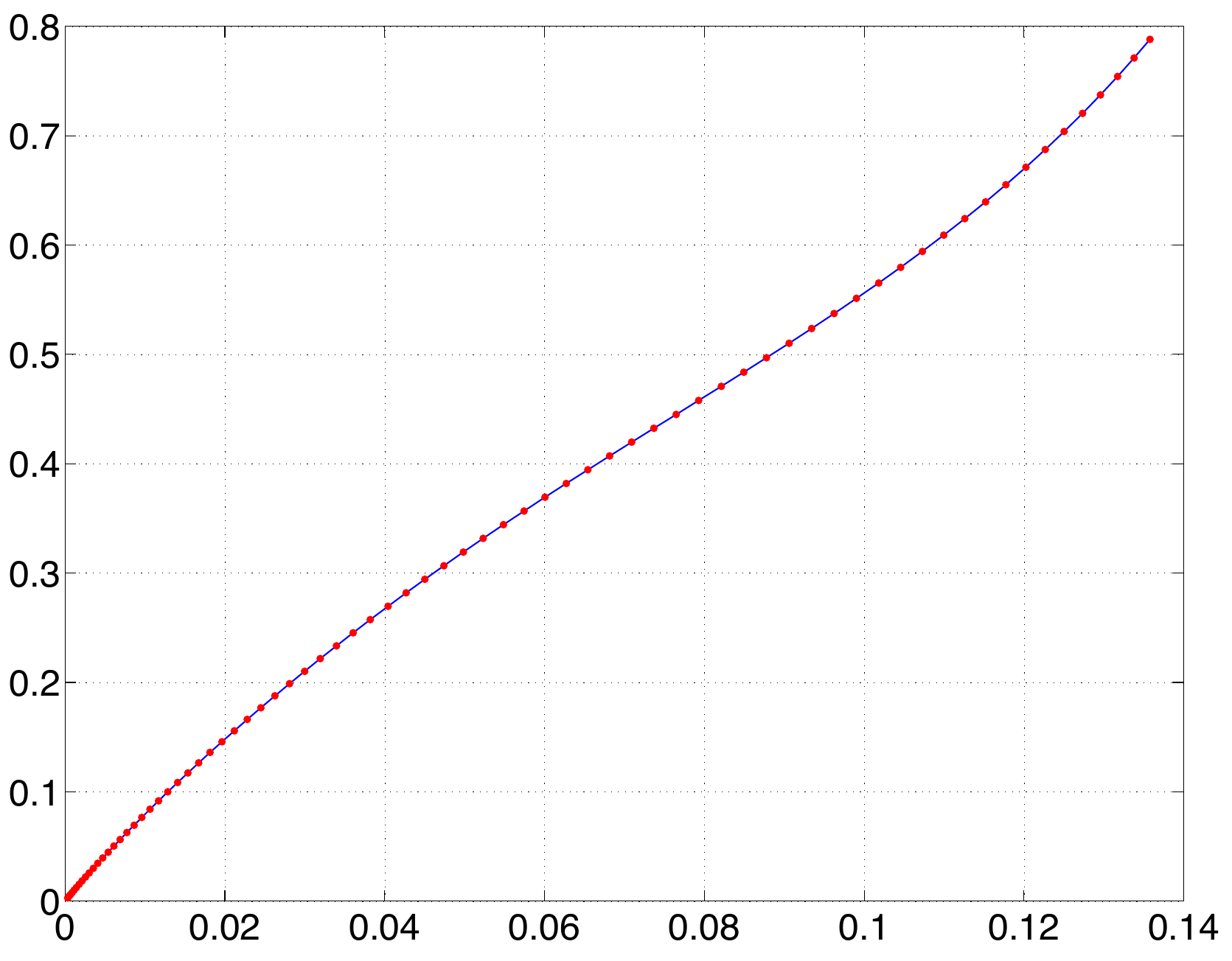}}

	\caption{\label{fig222}For every targeted distance (in abscissa) we compute the corresponding optimal cost (in ordinate).} 
\end{figure}

In this case, the cost-distance function $\Psi_{{\mathscr S},{\mathcal K}}$ defined in \eqref{function_psi} is merely the inverse of the distance-cost function $\Phi_{{\mathscr S},{\mathcal K}}$. Notice that, as already mentioned earlier, although we have always $\Phi_{{\mathscr S},{\mathcal K}}\circ\Psi_{{\mathscr S},{\mathcal K}}={\rm Id}$ (see the last point of Theorem~\ref{main_theorem_properties_2}), the identity $\Psi_{{\mathscr S},{\mathcal K}}\circ\Phi_{{\mathscr S},{\mathcal K}}={\rm Id}$ is false in the general case. To illustrate this, let us modify the shape manifold.

\subsubsection*{Non-monotonicity of the distance-cost function and discontinuity of the cost-distance function} We remove a part of the shape manifold $\mathcal S_\mu$ as on Fig.~\ref{fig224}, and consider now this ellipsoid with a ``hole'' as being the new shape manifold. 
\begin{figure}
	[h] \centerline{ 
	\includegraphics[width=.5 
	\textwidth]{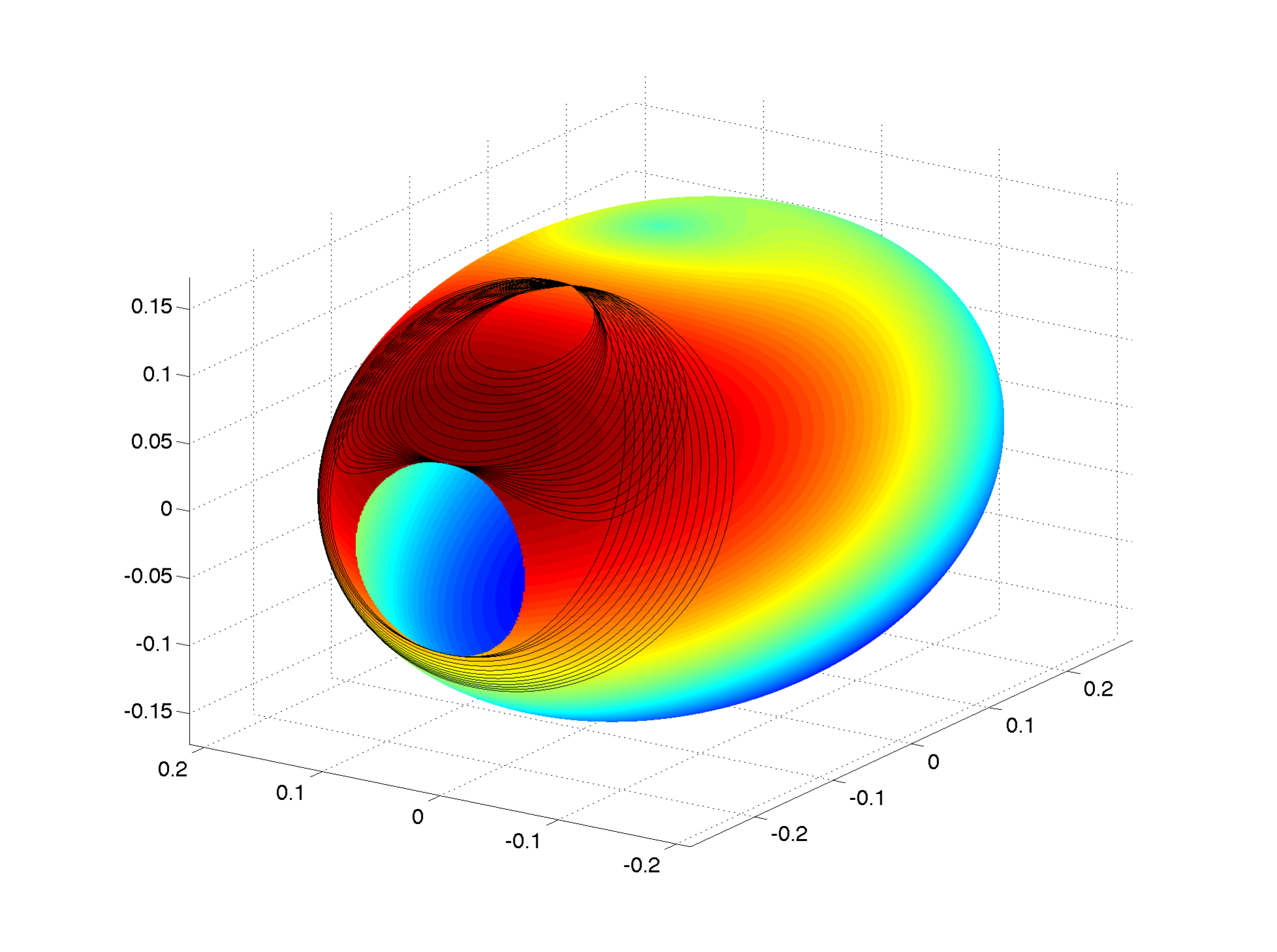}}

	\caption{\label{fig224}Some optimizers of the cost-distance and of the distance-cost functions.} 
\end{figure}

The graphs of the corresponding distance-cost function $\Phi_{{\mathscr S},{\mathcal K}}$ and cost-distance function $\Psi_{{\mathscr S},{\mathcal K}}$ are drawned on Fig.~\ref{fig225} and Fig.~\ref{fig226} respectively. 
\begin{figure}
	[h] \centerline{ 
	\includegraphics[width=.6 
	\textwidth]{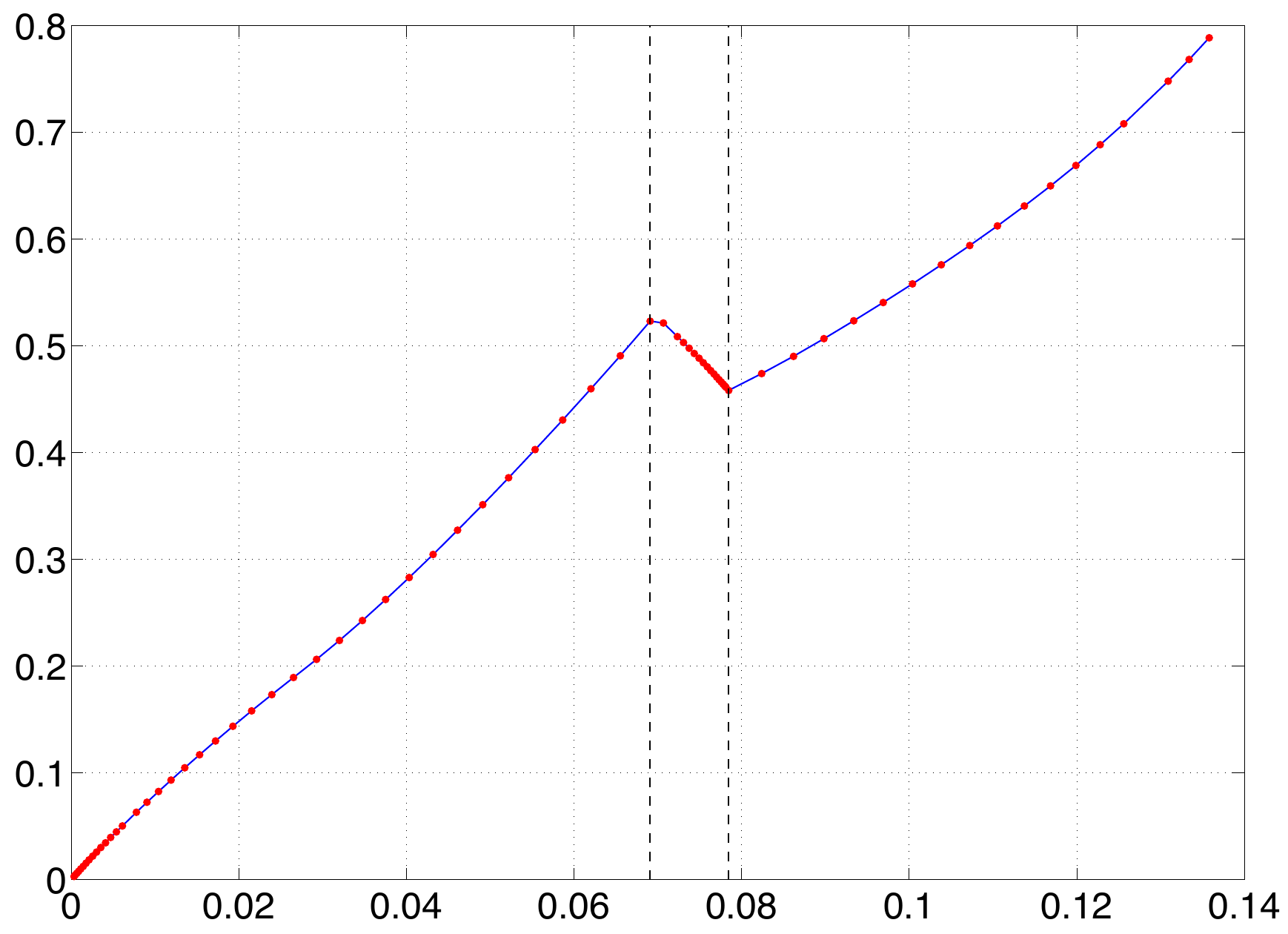}}

	\caption{\label{fig225}The distance-cost function $\Phi_{{\mathscr S},{\mathcal K}}$ is not monotonic in this case.} 
\end{figure}

\begin{figure}
	[h] \centerline{ 
	\includegraphics[width=.6 
	\textwidth]{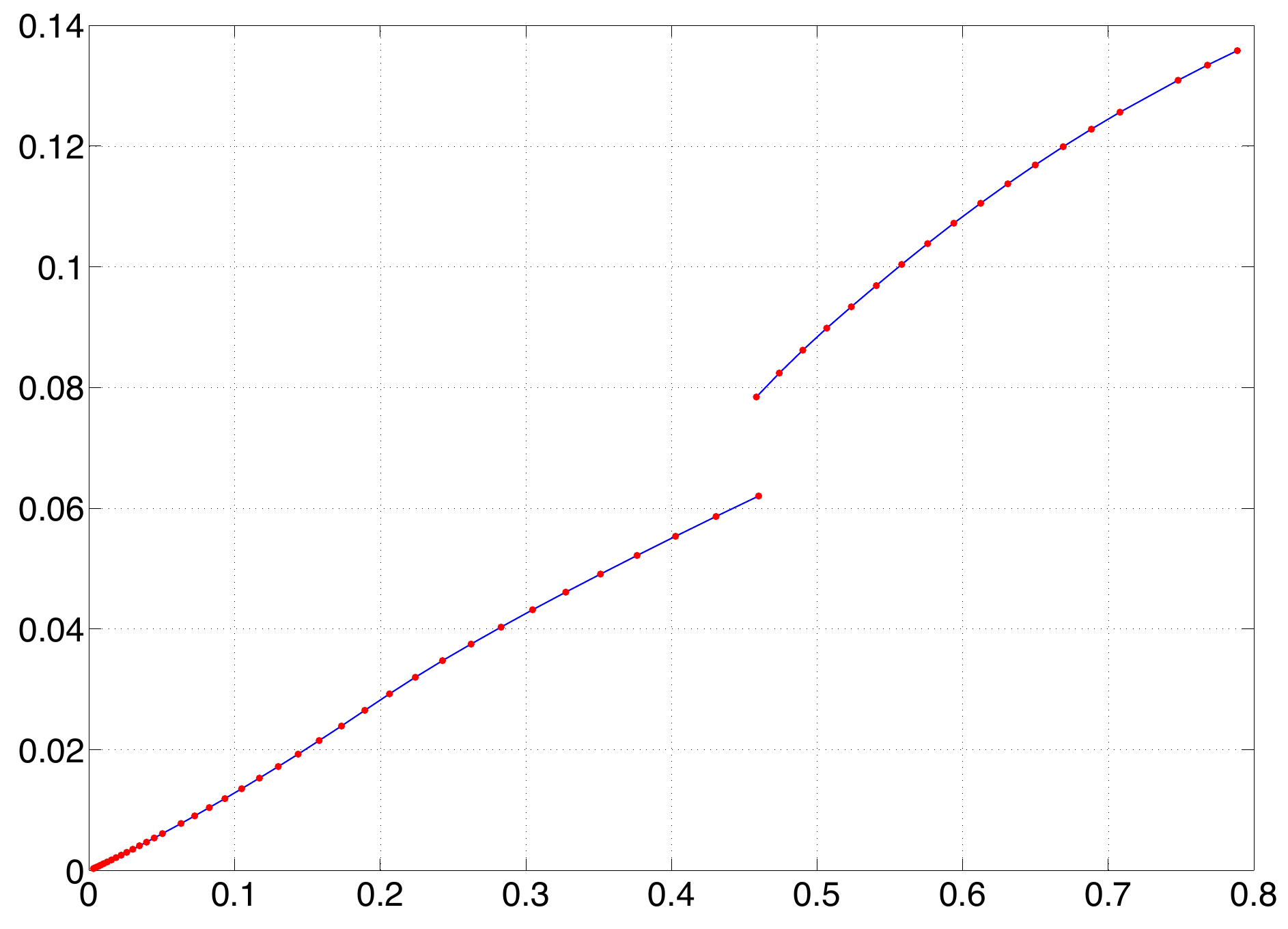}}

	\caption{\label{fig226}The cost-distance function $\Psi_{{\mathscr S},{\mathcal K}}$ is not continuous in this case.} 
\end{figure}

It is worth noticing that: 
\begin{itemize}
	\item The distance-cost function is not monotonic: for every distance between the dashed lines on Fig.~\ref{fig225}, it is cheaper to swim further than the destination point and then swim a little bit backward. On the shape manifold, the corresponding optimal control consists in a large clockwise loop followed by a small counterclockwise loop. 
	\item The cost-distance function is not continuous, but only right-continuous as stated in Theorem~\ref{main_theorem_properties_2}. The identity $\Phi_{{\mathscr S},{\mathcal K}}\circ\Psi_{{\mathscr S},{\mathcal K}}={\rm Id}$ is true but we no longer have $\Psi_{{\mathscr S},{\mathcal K}}\circ\Phi_{{\mathscr S},{\mathcal K}}={\rm Id}$. 
	\item Instead of adding a ``hole'' in the shape manifold $\mathcal S_\mu$, we could have obtained the same result by strongly increase the values of the density function of the measure ${\rm d}\mathcal L$ on the same area. Thus, the optimimal curves would have naturally avoided this very costly area of the shape manifold. 
\end{itemize}

\subsubsection*{Optimizing the cost among all the closed simple curves} We consider a closed simple curve on the ellipsoid and compute the corresponding covered distance by the swimmer. Then, we try to minimize the cost among all the closed simple curves for which the travelled distance is the same. Notice that it is not exactly what is stated in Problem~\ref{PROB_Moins_cher}, because here there is no fixed starting point and the curves are bound to be simple. The resulting curves on the ellipsoid are pictured on Fig~\ref{optim_1}, while the corresponding sequences of shapes are pictured on Fig~\ref{initial_guess}. 
\begin{figure}
	[h] \centerline{ 
	\includegraphics[height=.4 
	\textwidth]{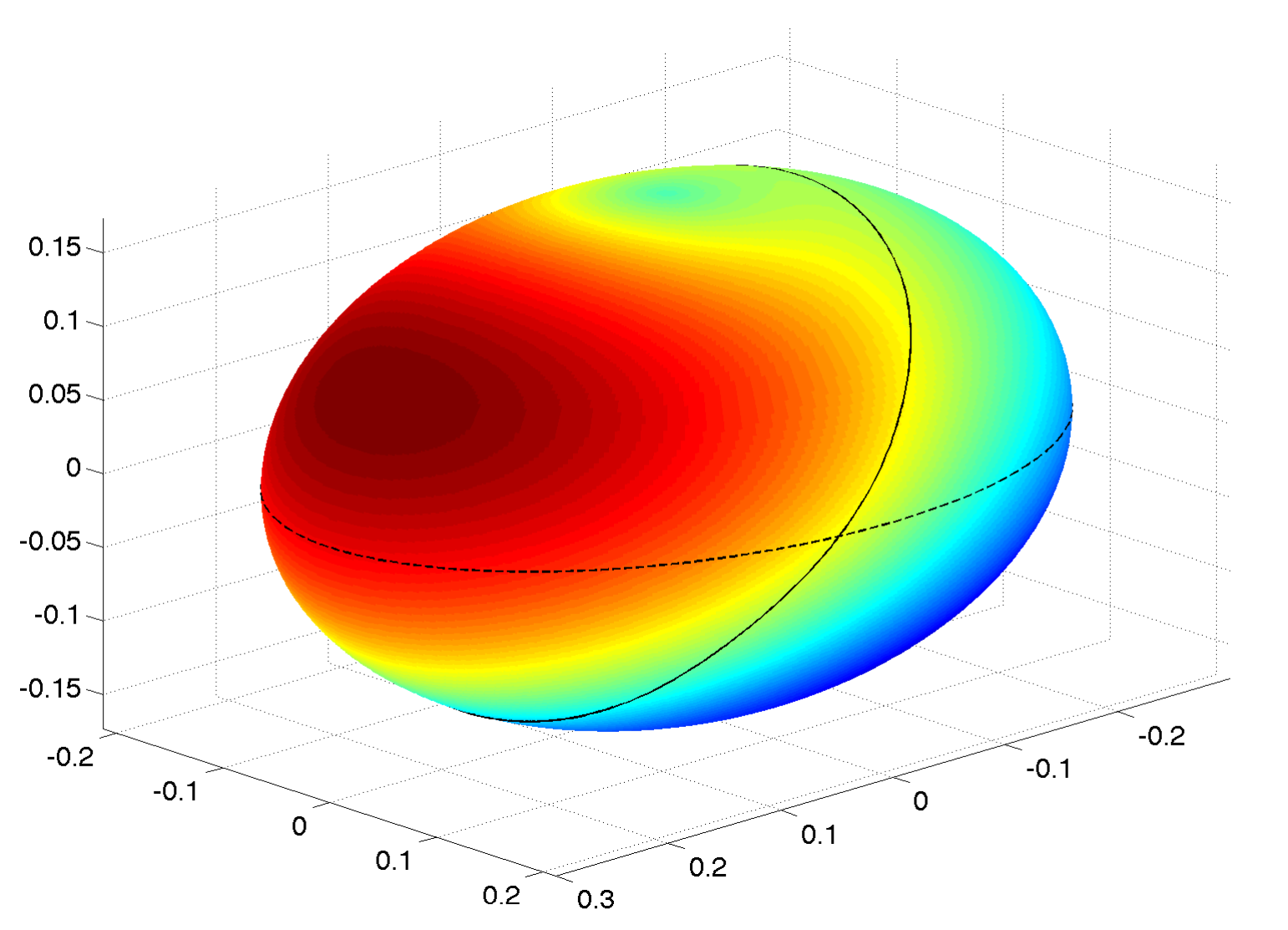}}

	\caption{\label{optim_1}The initial closed curve is the equator of the ellipsoid (the dashed line). The optimized curve (the continuous line) has the minimum cost for the same travelled distance. The colors are the same of in Fig~\ref{density_fct}.} 
\end{figure}

\begin{figure}
	[h] \centerline{ 
	\includegraphics[width=.35 
	\textwidth]{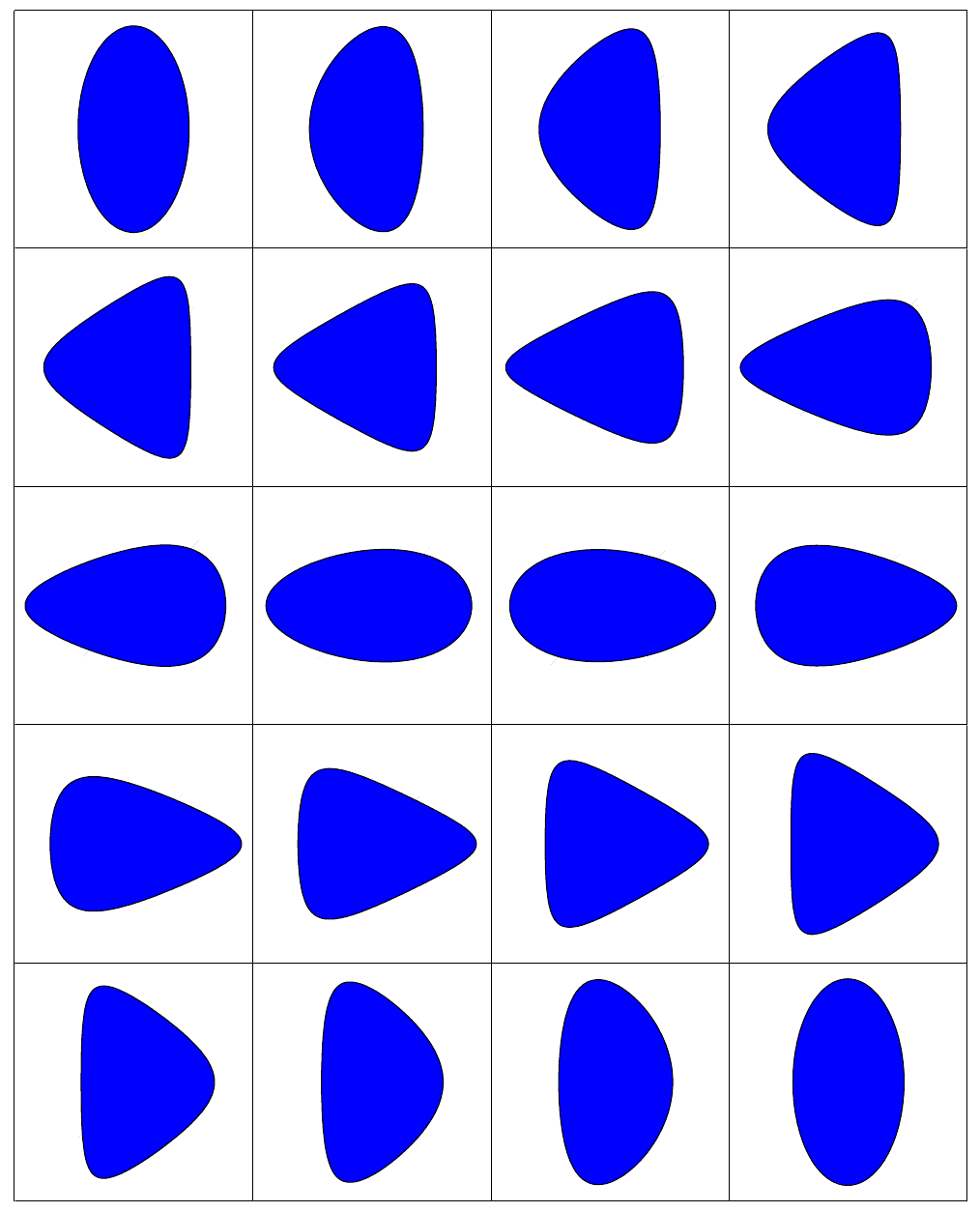}\qquad 
	\includegraphics[width=.35 
	\textwidth]{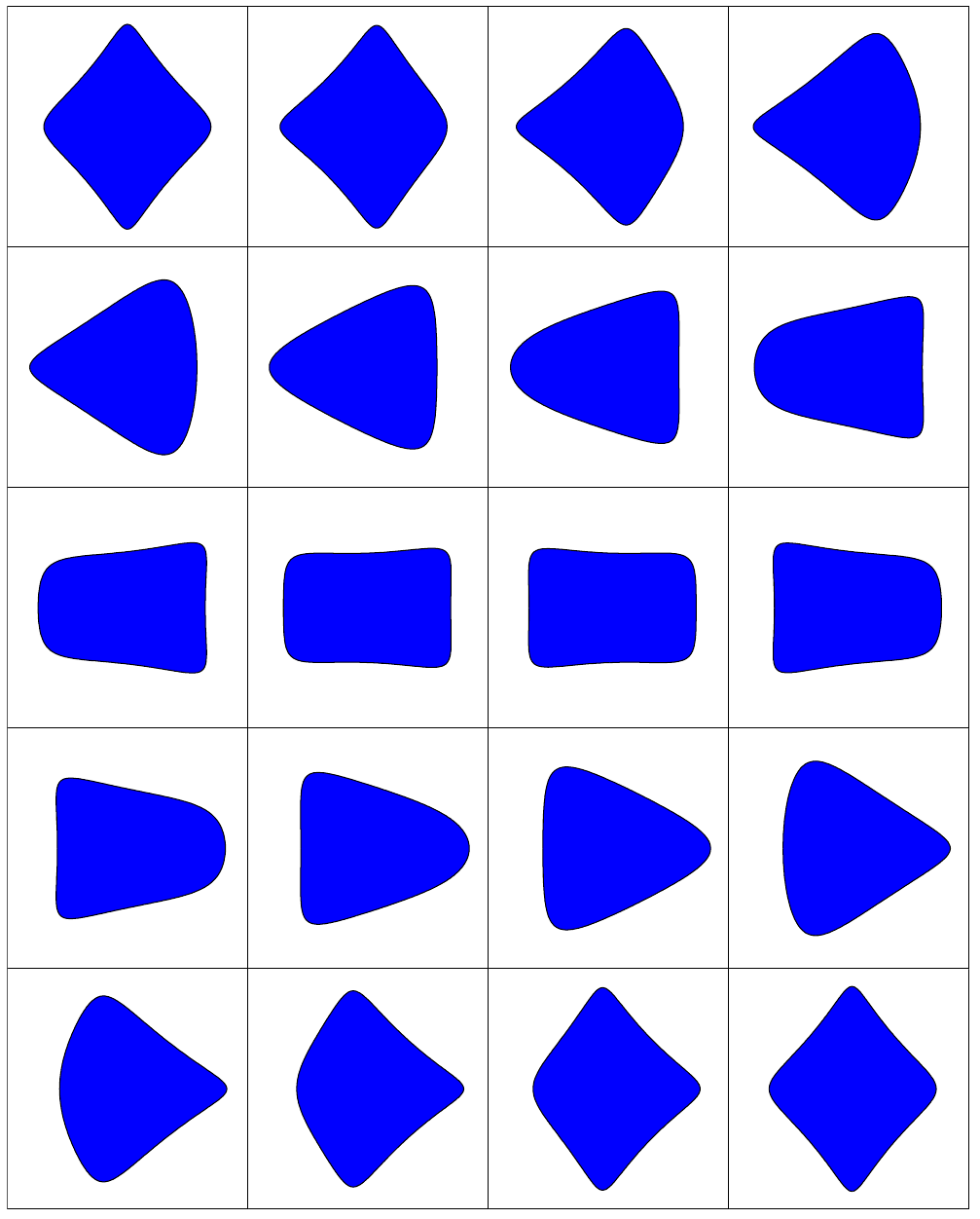}}

	\caption{\label{initial_guess}Sequence of 20 time-equidistributed shapes for the swimmer before optimization (left) and after optimization (right).} 
\end{figure}

%
\addcontentsline{toc}{section}{References} 
\bibliographystyle{plain}

\bibliography{Reference_Papier_Nancy}

\appendix

\section{Riemannian Geometry}\label{SEC_appendix_riemannian} Let $(M,g)$ be a Riemannian manifold. We denote by $\ell(\Gamma)$ the length of any rectifiable curve $\Gamma\subset M$ and for every $x\in M$ and $r>0$, we denote by $B(x,r)$ the Riemannian ball centered at $x$ and of radius $r$. The following Lemma ensures the existence of a small monotonic retract at any point of $M$. 
\begin{lemma}
	\label{LEM_retract_monotonic} Let $x_0$ be a point of $M$. Then there exists $r>0$ such that, for every path $\mathbf{\gamma}:[0,T]\mapsto M$ absolutely continuous with essentially bounded first derivative, such that: 
	\begin{enumerate}
		\item $\gamma(t)\in B(x_0,r)$ for every $t\in[0,T]$; 
		\item $\gamma(0)={\gamma}(T)=x_0$; 
	\end{enumerate}
	there exists a continuous function $\psi:[0,1]\times[0,T]\mapsto M$ satisfying: 
	\begin{enumerate}
		\item For every $s\in[0,1]$, $t\in[0,T]\mapsto\psi(s,t)$ is continuous with essentially bounded first derivative. 
		\item $\psi(1,\cdot)=\gamma$; 
		\item $\psi(0,\cdot)=x_0$; 
		\item The function $s\in[0,1]\mapsto \ell(\Gamma_s)$ (where $\Gamma_s$ is the curve parameterized by $t\in[0,T]\mapsto\psi(s,t)$) is increasing. 
	\end{enumerate}
\end{lemma}
\begin{proof}
	Let ${\rm inj}(x_0)$ be the injectivity radius at $x_0$ and denote $\mathcal V := B(x_0,r)$, for $0<r< {\rm inj}(x_0)$ (the constant $r$ will be fixed later on). Let $\gamma$ be a path included in $\mathcal V$ and satisfying the hypotheses of the lemma. Then define: $$\zeta:t\in[0,T]\mapsto \zeta(t) := \exp_{x_0}^{-1} (\gamma(t))\in T_{x_0}M.$$ This function has the same regularity as $\gamma$ and $\zeta(0)=\zeta(T)=0$. Define now, for every $(s,t)\in[0,1]\times[0,T]$: $$\psi(s,t)=\exp_{x_0}(s\zeta(t)).$$ This function has the required regularity and satisfies the equalities $\psi(1,\cdot)=\gamma$ and $\psi(0,\cdot)=x_0$. So it remains only to prove that for $r$ small enough, the length of $\Gamma_s$ is increasing in $s$.
	
	In the exponential map the metric $g$ have the following Cartan local development: 
	\begin{equation}
		\label{cartan} g_{ij}(x) = \delta_{i}^{j} -\frac13 \sum_{k,l} R_{iklj}(0)x_{k}x_{l}+ \mathcal O(\| x \|_{E}^{3}), 
	\end{equation}
	where $\delta_i^j$ is the Kronecker symbol, $R_{ijkl}$ are the coefficients of the Riemann curvature tensor and $ \| x \|_{E} $ stands for the Euclidean norm. The quantity we are interested in estimating is: 
	\begin{equation}
		\label{to_decrease} \ell(\Gamma_s) = \int_0^T \| 
		\partial_t \psi(s,t) \|_ {g(\psi(s,t))} \, \,{\rm d}t, 
	\end{equation}
	and in the local chart, according to \eqref{cartan}, we have: $$ \| 
	\partial_t \psi(s,t) \|^2_ {g(\psi(s,t))} = s^2\|\dot{\zeta}(t)\|_E^2 -s^4\left(\frac{1}{3} \sum_{i,j} \sum_{k,l} R_{iklj}(0)\zeta_{k}(t)\zeta_{l}(t)\dot{\zeta}^i(t)\dot{\zeta}^j(t)\right)+s^3\left(\|\dot{\zeta} (t) \|_{E}^{2} \mathcal O(\| \zeta(t) \|_E^{3} )\right). $$ So for $r$ small enough, $\zeta(t)$ is uniformly small and for every $t\in[0,T]$, the function $s\in[0,1]\mapsto \| 
	\partial_t \psi(s,t) \|^2_ {g(\psi(s,t))}$ is increasing. We draw the same conclusion for the quantity \eqref{to_decrease} and the proof is completed. 
\end{proof}

\section{A brief Survey of the Orbit Theorem}\label{SEC_Appendix_Orbit_Th} In this Appendix, we aim to recall the statement of the Orbit Theorem. The material presented below is now considered as a classical part of geometric control theory.

Throughout this section, $M$ is a real analytic manifold, and $\cal G$ a set of analytic vector fields on $M$. We do not assume in general that the fields from $\cal G$ are complete.

\subsection{Attainable sets} Let $f$ be an element of $\cal G$ and $q^{\ast}$ be an element of $M$. The Cauchy problem 
\begin{equation}
	\label{EQ-Cauchy} \dot{q} = f(q),\qquad q(0) = q^{\ast}, 
\end{equation}
admits a solution defined on the open interval $I(f,q^{\ast})$ containing $0$. For any real $t$ in $I(f,q^{\ast})$ we denote the value of the solution of (\ref{EQ-Cauchy}) at time $t$ by $e^{tf}(q^{\ast})$. We denote by $I(f,q^{\ast})^+=I(f,q^{\ast})\,\cap\, ]0,+\infty[$ the positive elements of $I(f,q^{\ast})$.

For any element $q_0$ in $M$ and any positive real number $T$, we define the \emph{attainable set at time $T$} of $\cal G$ from $q_0$ by the set ${\cal A}_{q_0}(T)$ of all points of $M$ that can be attained with $\cal G$ using piecewise constants controls in time $T$ 
\begin{multline*}
	{\cal A}_{q_0}(T)= \Big\{ e^{t_p f_p}\circ e^{t_{p-1} f_{p-1}} \circ \cdots\circ e^{t_1 f_1} (q_0)\,:\, p \in \mathbf{N},\, f_i \in {\cal G},\\
	t_i\in I(f_i, e^{t_{i-1} f_{i-1}}\circ \cdots \circ e^{t_1 f_1}(q_0))^+,\, t_1+\cdots+t_p=T \Big \}, 
\end{multline*}
the times $t_i$ and the fields $f_i$ being chosen in such a way that every written quantity exists. We define also the \emph{orbit} of $\cal G$ trough $q_0$ by the set ${\cal O}_{q_0}$ of all points of $M$ that can be attained with $\cal G$ using piecewise constant controls, \emph{at any positive or negative time} $${\cal O}_{q_0}(T)= \Big\{ e^{t_p f_p}\circ e^{t_{p-1} f_{p-1}} \circ \cdots\circ e^{t_1 f_1} (q_0)\,:\, p \in \mathbf{N},\, f_i \in {\cal G}, t_i\in I(f_i, e^{t_{i-1} f_{i-1}}\circ \cdots e^{t_1 f_1}(q_0)) \Big\}. $$ Of course, if $\cal G$ is a cone, that is if $ \lambda f \in {\cal G}$ for any positive $\lambda$ as soon as $f $ belongs to $\cal G$, the set ${\cal A}_{q_0}(T)$ does not depend on the positive $T$ but only on $q_0$. If $\cal G$ is assumed to be symmetric, that is if $-f$ belongs to $\cal G$ as soon as $f$ belongs to $\cal G$, then the orbit of $\cal G$ trough a point $q_0$ is the union of all attainable sets at positive time of $\cal G$ from $q_0$.

\subsection{Lie algebra of vector fields}

If $f_1$ and $f_2$ are two vector fields on $M$ and $q$ is a point of $M$, the \emph{Lie bracket} $[f_1,f_2](q)$ of $f_1$ and $f_2$ at a point $q$ is the derivative at $t=0$ of the curve $t \mapsto \gamma(\sqrt{t})$ where $\gamma$ is defined by $\gamma(t):=e^{-t f_2} e^{-t f_1} e^{t f_2} e^{t f_1} (q)$ for $t$ small enough. The Lie bracket of $f_1$ and $f_2$ at a point $q$ is an element of the tangent space $T_{q}M$ of $M$ at the point $q$. The Lie bracket is bilinear and skew-symmetric in $f_1$ and $f_2$, and measures the non-commutativity of the fields $f_1$ and $f_2$ (see \cite[Prop 2.6]{agrachev}). 
\begin{proposition}
	For any $f_1$, $f_2$ in $\cal G$, we have the equivalence: $$ e^{t_1 f_1} e^{t_2 f_2} = e^{t_2 f_2} e^{t_1 f_1} \Leftrightarrow [f_1,f_2]=0$$ for all times $t_1$ and $t_2$ (if any) for which the expressions written in the left hand side of the above equivalence make sense. 
\end{proposition}
Lie brackets of vectors fields are easy to compute with the following formulas (see \cite[Prop 1.3]{agrachev} and \cite[Exercise 2.2]{agrachev}). 
\begin{proposition}
	\label{PRO_calculLieBracket} For any $f_1$, $f_2$ in $\cal G$, for any $q$ in $M$, $$[f_1,f_2](q)=\frac{df_2}{dq}f_1(q)-\frac{df_1}{dq}f_2(q).$$ 
\end{proposition}
Further, we have the useful property: 
\begin{proposition}
	\label{PRO_LieBracket_Multiple} Let $f_1$ and $f_2$ be two smooth vector fields on $M$, and let $a,b:M\rightarrow \mathbb{R}$ be two smooth functions. Then $$[aX,bY]=a b[X,Y] +\left (\frac{db}{dq}X \right )Y -\left (\frac{da}{dq}Y \right )X.$$ 
\end{proposition}
From the Lie brackets, we can define the Lie algebra: 
\begin{definition}
	The \emph{Lie algebra} of ${\cal G}$ is the linear span of all Lie brackets, of any length, of the elements of $\cal G$ $$\mathrm{Lie}~ {\cal G}=\mathrm{span}\big\{ [f_1,[\ldots[f_{k-1},f_k]\ldots]],\, k \in \mathbf{N},\, f_i \in {\cal G} \big\},$$ which is a subset of all the vector fields on $M$. 
\end{definition}
We denote by $\mathrm{Lie}_q {\cal G}:=\big\{g(q),\, g \in \mathrm{Lie}~{\cal G}\big\}$ the evaluation $\mathrm{Lie}_q {\cal G}$ of the Lie algebra generated by $\cal G$ at a point $q$ of $M$.

\subsection{The Orbit Theorem}

The Orbit Theorem describes the differential structure of the orbit trough a point (see for instance \cite[Th 5.1]{agrachev} for a proof). 
\begin{theorem}
	[Orbit Theorem]\label{THE_Orbit} For any $q$ and $q_0$ in $M$: 
	\begin{enumerate}
		\item\label{conc1} ${\cal O}(q_0)$ is a connected immersed submanifold of $M$. 
		\item\label{conc2} If $q \in {\cal O}(q_0)$, then $T_q{\cal O}(q_0)=\mathrm{Lie}_{q}{\cal G}$. 
	\end{enumerate}
\end{theorem}
\begin{remark}
	The conclusion \eqref{conc1} of the Orbit Theorem holds true even if $M$ and $\cal G$ are only assumed to be smooth (and not analytic). The conclusion \eqref{conc2} is false in general when $\cal G$ is only assumed to be smooth. 
\end{remark}
The Orbit Theorem has many consequences, among them the following useful properties (see \cite[Th 5.2]{agrachev} for a proof and further discussion). 
\begin{theorem}
	[Rashevsky-Chow] If $\mathrm{Lie}_q{\cal G}=T_qM$ for every $q$ in $M$, then the orbit of $\cal G$ through $q$ is equal to $M$. 
\end{theorem}
\begin{proposition}
	\label{PRO_CompleteControl} If $\cal G$ is a symmetric cone such that $\mathrm{Lie}_q{\cal G}=T_qM$ for every $q$ in $M$, then the attainable set at any positive time of any point of $M$ is equal to $M$. 
\end{proposition}

\end{document}